\documentclass[a4paper,reqno]{amsart}

\title{A twist on ring morphisms and crepant contractions}
\author{Marina Godinho}

\providecommand{\lang}{UKenglish}
\providecommand{\geomoptions}{hmargin=3cm,vmargin=3.5cm}

\usepackage[utf8]{inputenc}

\usepackage[\lang]{babel}
\usepackage{csquotes}
\usepackage{lmodern}

\usepackage[T1]{fontenc}
\usepackage[\geomoptions]{geometry}
\usepackage{amsmath,amssymb,amsthm,amscd,mathrsfs,mathtools,upgreek,euscript,dsfont,amsfonts,colonequals,graphicx, pifont}
\usepackage{adjustbox}

\usepackage[dvipsnames,svgnames]{xcolor}
\usepackage[pdflang=en-UK, colorlinks, urlcolor=Navy, linkcolor=Navy, citecolor=Navy]{hyperref}
\usepackage[nameinlink]{cleveref}
\crefname{section}{\S}{\S\S} 
\usepackage{tikz}
	\usetikzlibrary{calc}
	\usetikzlibrary{arrows,decorations.markings}
\usepackage{tikz-cd}
\tikzcdset{scale cd/.style={every label/.append style={scale=#1},
    cells={nodes={scale=#1}}}}
\usetikzlibrary{arrows}

\usepackage{pgffor}
\usepackage{pstricks}
\usepackage{enumitem}
 
\newcommand{\donothing}[1]{}

\setlist[enumerate]{format=\normalfont}

\renewcommand{\geq}{\geqslant}
\renewcommand{\leq}{\leqslant}

\renewcommand{\alpha}{\upalpha}
\renewcommand{\beta}{\upbeta}
\renewcommand{\chi}{\upchi}
\renewcommand{\delta}{\updelta}
\renewcommand{\epsilon}{\upvarepsilon}
\renewcommand{\eta}{\upeta}
\renewcommand{\gamma}{\upgamma}
\renewcommand{\iota}{\upiota}
\renewcommand{\kappa}{\upkappa}
\renewcommand{\lambda}{\uplambda}
\renewcommand{\mu}{\upmu}
\renewcommand{\nu}{\upnu}
\renewcommand{\omega}{\upomega}
\renewcommand{\phi}{\upphi}
\renewcommand{\pi}{\uppi}
\renewcommand{\psi}{\uppsi}
\renewcommand{\rho}{\uprho}
\renewcommand{\sigma}{\upsigma}
\renewcommand{\tau}{\uptau}
\renewcommand{\theta}{\uptheta}
\renewcommand{\upsilon}{\upupsilon}

\renewcommand{\Delta}{\Updelta}
\renewcommand{\Gamma}{\Upgamma}
\renewcommand{\Lambda}{\Uplambda}
\renewcommand{\Omega}{\Upomega}
\renewcommand{\Psi}{\Uppsi}
\renewcommand{\Sigma}{\Upsigma}
\renewcommand{\Theta}{\Uptheta}
\renewcommand{\Upsilon}{\Upupsilon}
\renewcommand{\Xi}{\Upxi}

\DeclareMathOperator{\depth}{depth}

\DeclareMathOperator{\radd}{rad}
\DeclareMathOperator{\id}{id}

\DeclareMathOperator{\Hom}{Hom}

\DeclareMathOperator{\uHom}{\underline{Hom}}
\DeclareMathOperator{\End}{End}

\DeclareMathOperator{\uEnd}{\underline{End}}
\DeclareMathOperator{\Coker}{Coker}
\DeclareMathOperator{\Img}{Im}
\newcommand\Dtensor{\otimes^{\Lderived{}\kern-0.5em}
}

\newcommand{\op}{\mathrm{op}}
\newcommand{\bdd}{\mathrm{b}}

\DeclareMathOperator{\chain}{Ch}
\DeclareMathOperator{\Dcat}{D}
\DeclareMathOperator{\Kcat}{K}
\newcommand\Db{\Dcat^\bdd}
\newcommand\Kb{\Kcat^\bdd}

\DeclareMathOperator{\modu}{mod}

\DeclareMathOperator{\proj}{proj}
\DeclareMathOperator{\inj}{inj}

\DeclareMathOperator{\add}{add}

\DeclareMathOperator{\fl}{f\kern -.5pt l}
\DeclareMathOperator{\CM}{CM}

\DeclareMathOperator{\uCM}{\underline{CM}}

\DeclareMathOperator{\tr}{tr}
\DeclareMathOperator{\act}{act}

\DeclareMathOperator{\Mod}{Mod}

\DeclareMathOperator{\coh}{coh}
\DeclareMathOperator{\qcoh}{Qcoh}

\DeclareDocumentCommand \Rderived { o m } {%
  \IfNoValueTF {#1} {%
    {\rm\bf R}{#2}%
  }{%
    {\rm\bf R}^{#1}{#2}%
  }%
}

\DeclareDocumentCommand \Lderived { o m } {%
  \IfNoValueTF {#1} {%
    {\rm\bf L}{#2}%
  }{%
    {\rm\bf L}^{#1}{#2}%
  }%
}
\DeclareMathOperator{\Ext}{Ext}
\DeclareMathOperator{\Tor}{Tor}
\newcommand{\cohom}[1]{\mathrm{H}^{#1}}

\DeclareMathOperator{\cone}{cone}

\DeclareMathOperator{\pdim}{p.dim}
\DeclareMathOperator{\tot}{Tot}
\DeclareMathOperator{\ev}{ev}

\newcommand{\rmL}{\mathrm{L}}
\newcommand{\rmR}{\mathrm{R}}

\newcommand{\RA}{\mathrm{RA}}
\newcommand{\LA}{\mathrm{LA}}

\DeclareMathOperator{\spec}{Spec}
\DeclareMathOperator{\maxspec}{MaxSpec}

\newcommand{\con}{\mathrm{con}}
\newcommand*{\cHom}{\EuScript{H}\kern -.5pt \mathit{om}}
\newcommand*{\cEnd}{\EuScript{E}\kern -.5pt \mathit{nd}}

\newcommand{\tow}{\xrightarrow}

\newcommand{\N}{\mathbb{N}}
\newcommand{\Z}{\mathbb{Z}}

\newcommand{\F}{\mathbb{F}}

\newcommand{\cA}{\EuScript{A}}
\newcommand{\cB}{\EuScript{B}}
\newcommand{\cC}{\EuScript{C}}
\newcommand{\cD}{\EuScript{D}}
\newcommand{\cE}{\EuScript{E}}
\newcommand{\cF}{\EuScript{F}}
\newcommand{\cG}{\EuScript{G}}

\newcommand{\cI}{\EuScript{I}}
\newcommand{\cJ}{\EuScript{J}}

\newcommand{\cM}{\EuScript{M}}
\newcommand{\cN}{\EuScript{N}}
\newcommand{\cO}{\EuScript{O}}
\newcommand{\cP}{\EuScript{P}}
\newcommand{\cQ}{\EuScript{Q}}
\newcommand{\cR}{\EuScript{R}}
\newcommand{\cS}{\EuScript{S}}
\newcommand{\cT}{\EuScript{T}}

\newcommand{\cV}{\EuScript{V}}

\newcommand{\cX}{\EuScript{X}}
\newcommand{\cY}{\EuScript{Y}}

\DeclareMathOperator{\uE}{\underline{\cE}}

\usepackage[backend=bibtex,citestyle=alphabetic,bibstyle=alphabetic,giveninits=true,doi=false,isbn=false,url=false,eprint=true,useprefix=true]{biblatex}

\appto{\bibsetup}{\sloppy}

\DeclareFieldFormat{postnote}{#1}

\renewbibmacro{in:}{%
  \ifentrytype{article}{}{\printtext{\bibstring{in}\intitlepunct}}}

\DeclareFieldFormat[article,inbook,incollection,inproceedings,patent]{title}{#1}
\DeclareFieldFormat[thesis,unpublished]{title}{{\it #1}}

\DeclareFieldFormat[online]{date}{Preprint (#1)}

\renewbibmacro*{publisher+location+date}{%
  \printlist{publisher}%
  \setunit*{\addcomma\space}%
  \usebibmacro{date}%
  \newunit}
  
\DeclareFieldFormat{eprint:arxiv}{%
\ifentrytype{online}
  {\ifhyperref
    {\href{http://arxiv.org/abs/#1}{\nolinkurl{arXiv:#1}}}
    {\nolinkurl{arXiv:#1}}
   \iffieldundef{eprintclass}
    {}
    {{\tt \mkbibbrackets{\thefield{eprintclass}}}}}
  {\iffieldundef{eprintclass}
    {\mkbibparens{\ifhyperref
    {\href{http://arxiv.org/abs/#1}{\nolinkurl{arXiv:#1}}}
    {\nolinkurl{arXiv:#1}}}}
    {\mkbibparens{\ifhyperref
    {\href{http://arxiv.org/abs/#1}{\nolinkurl{arXiv:#1}}}
    {\nolinkurl{arXiv:#1}}
    {\tt \mkbibbrackets{\thefield{eprintclass}}}}}}}
\theoremstyle{plain}
\newtheorem{theorem}{Theorem}[section]
\newtheorem{theorem*}{Theorem}
\newtheorem{lemma}[theorem]{Lemma}
\newtheorem{cor}[theorem]{Corollary}
\newtheorem{cor*}{Corollary}
\newtheorem{prop}[theorem]{Proposition}

\newtheorem{thmx}{Theorem}

\theoremstyle{definition}
\newtheorem{definition}[theorem]{Definition}
\newtheorem{example}[theorem]{Example}

\newtheorem{setup}[theorem]{Setup}
\newtheorem{remark}[theorem]{Remark}
\newtheorem*{definition*}{Definition}

\theoremstyle{remark}
\newtheorem{notation}[theorem]{Notation}

\numberwithin{equation}{section}

\usepackage{upref}
\crefdefaultlabelformat{#2\textup{#1}#3}
\begin{document}

\begin{abstract}
    Given a ring morphism, this paper constructs the twist functor around the induced derived restriction of scalars functor. We prove that the twist around ring morphisms is a derived autoequivalence in the setting of twists induced by Frobenius exact categories. As a corollary, it is shown that the noncommutative twist introduced by Donovan and Wemyss is in fact a spherical twist around the restriction of scalars functor. We then use this technology to obtain new spherical twists for singular schemes, and discuss how our result extends previous works on spherical twists induced by crepant contractions. 
\end{abstract}

\maketitle

\tableofcontents

\section{Introduction}

Autoequivalence groups of derived categories, through their connections with the stability manifold and mirror symmetry, have become important objects in algebraic geometry and representation theory. The description of these autoequivalence groups in general is difficult, and so many authors have contributed to constructing autoequivalences for certain classes of derived categories (See, for example, \cite{add, BaBri, BO, Don, BT, DW4, O, ST} in algebraic geometry and \cite{BDA, GM, J, KS, MY, Qiu} in representation theory).

In this work, we are interested in constructing autoequivalences as spherical twists and cotwists. Twists around spherical objects were introduced by Seidel and Thomas \cite{ST} in order to obtain derived autoequivalences of complex smooth projective varieties. The construction was generalised by various authors to twists around spherical functors. For a survey on the development of this technology, see, for example, \cite{add} or \cite{seg}.

Consider a functor $S \colon \cA \to \cB$ between triangulated categories which admits a right adjoint $R$ and a left adjoint $L$. Then, the twist around $S$ is a functor $T$ which sits in a functorial triangle
    \begin{align*}
       SR \tow{\epsilon^\rmR} 1_{\cB} \to T \to^+  
    \end{align*}
where $\epsilon^\rmR$ is the counit of the adjoint pair $(S, R)$. For the precise definition of  twists and cotwists around functors see \ref{twist def} and \ref{funct cones}. When both the twist and cotwist are autoequivalences, then $T$ is said to be a spherical twist and $S$ is said to be a spherical functor.

\subsection{Twists around ring morphisms}

 The setting of this paper is that of twists around ring morphisms, in view of their applications to geometric settings induced by crepant contractions. First, we will use this technology to answer the open question of whether the noncommutative twist functor introduced by Donovan and Wemyss \cite[\S 5.3]{DW1} is a twist around a certain functor. We also obtain spherical twists in a more general settings, particularly for very singular schemes. In the process, we construct new spherical twists in algebraic settings.

Given a ring morphism $p \colon A \to B$ with $B \in \Kb(\proj A)$, our first result constructs the twist around the functor $F \colonequals - \Dtensor_B B {}_A \colon \Dcat(B) \to \Dcat(A)$. The cotwist around $F$ is computed, but details are left to \ref{cotwist gen}. 

\begin{thmx}[{\ref{twist t res}}]
    The functor
    \[ \cT \colonequals \Rderived{\Hom}_A({}_A \cone(p)[-1] {}_A, -) \colon \Dcat(A) \to \Dcat(A) \]
    is the twist around $F$. Note that when $p$ is an epimorphism, then $\cT \cong \Rderived{\Hom}_A({}_A \ker p {}_A, -)$.
\end{thmx}

In particular, the above result proves that the noncommutative twist introduced in \cite[section 5.3]{DW1} is in fact the twist around the appropriate derived restriction of scalars functor. This characterisation of the noncommutative twist affords much control over the functor. 

\subsection{Spherical twists induced by Frobenius exact categories}

Let $\cR$ be a ring. Given an object $X$ in an $\cR$-linear Frobenius exact category $\cE$  satisfying mild conditions (see \cref{frob set up sect} for the precise details), we are able to specify when the twist around the quotient morphism $\pi \colon \End_\cE(X) \to \uEnd_\cE(X)$ is spherical. The key point is roughly that $F$ is spherical if the action of the suspension functor on the object is periodic up to additive closure. The precise statement is the following.

\begin{thmx}[{\ref{summary theorem}}] \label{summary intro}
    Suppose that $\add_{\uE} \Sigma^t X = \add_{\uE} X$ for some integer $2 \leq t \leq \dim \cR$, and $\Ext^{i}_{\uE}(X, X) = 0$ for all $0 < i < t$. Write $\Lambda = \End_\cE(X)$ and $\Lambda_\con = \uEnd_\cE(X)$. Then, the functor
    \[ F \colonequals - \otimes_{\Lambda_\con} \Lambda_\con \colon \Db(\modu \Lambda_\con ) \to \Db(\modu \Lambda) \]
    is spherical. Moreover, the twist around $F$ is $\cT = \Rderived{\Hom_\Lambda}(\ker \pi, -)$ and the cotwist is
    \begin{align*}
        \cC = \Rderived{\Hom}_{\Lambda_\con}(\Tor^\Lambda_{k+1}(\Lambda_\con, \Lambda_\con), -)[-k+2]  \colon \Db(\modu \Lambda_\con) \to \Db(\modu \Lambda_\con)
    \end{align*}
\end{thmx}

In \cref{periodic suspension section}, we show that the suspension functor acts periodically on $X$ up to additive closure (and $X$ satisfies the Ext-vanishing assumptions) if and only if $\uEnd_\cE(X)$ has relatively spherical properties in the sense of \ref{def rel sph}. 

\begin{thmx}[{\ref{syz theorem}}]
    Fix $t \in \Z$ with $2 \leq t \leq d$. Then, the following are equivalent
    \begin{enumerate}
        \item  $\Lambda_\con = \uEnd_\cE(X)$ is perfect over $\Lambda = \End_\cE(X)$ and, further, $\Ext^k_\Lambda(\Lambda_\con, S) = 0$ for all $k \neq 0, t$ and all simple $\Lambda_\con$-modules $S$.
        \item   $\Ext^k_{\uE}(X, X) = 0$ for $0 < k < t-1$, and $\add_{\uE} \Sigma^k X = \add_{\uE} X$. 
    \end{enumerate}
\end{thmx}

It turns out that this theorem is the key technology which allows us, in \cref{crepant contractions affine section}, to apply \cref{summary intro} in geometric settings to construct autoequivalences induced by crepant contractions $X \to Y$, where $Y$ is a scheme which need not be affine over a complete local ring. 

In the process of proving \cref{summary intro}, the main novelty introduced is that  we construct and control what we call \textit{partially minimal projective resolutions} (defined in \cref{partial min res section}) in non Krull-Schmidt settings. 

\subsection{Spherical twists induced by crepant contractions onto affine schemes}

In \cref{crepant contractions affine section}, we apply section \ref{spherical by frob} to obtain derived autoequivalences of schemes with at worst Gorenstein singularities. 

Let $f \colon X \to Y \colonequals \spec R$ be a crepant contraction satisfying mild conditions and admitting a tilting bundle $\cV$ (see \ref{z local tilting setup} for the precise setup), then there is an equivalence
\[ \Psi_{\cV} \colonequals \Rderived{\Hom_{\Dcat(\coh \cX)}}(\cV, -) \colon \Db(\coh \cX) \to \Db(\modu \End_R(f_* \cP)). \]
Write $\Lambda \colonequals \End_R(f_* \cP)$, and let $[\add R]$ be the ideal of $\Lambda$ consisting of morphisms $f_*\cP \to f_* \cP$ which factor through a finite projective $R$-module. Consider the contraction algebra given by $\Lambda_\con \colonequals \End_R(f_*\cP) / [\add R]$.

\begin{thmx} [{\ref{z local tw and ctw}}] \label{z local intro}
    If $\Lambda_\con$ is perfect over $\Lambda$ and $t$-relatively spherical, then the functor 
    \[ \Psi^{-1}_\cV \cdot (-) \Dtensor_{\Lambda_\con} \Lambda_\con {}_\Lambda \colon \Db(\modu \Lambda_\con) \to \Db(\coh X). \]
    is spherical. Moreover, the twist $\cT_R$ and cotwist $\cC_R$ around $\Psi^{-1}_\cV \cdot (-) \Dtensor_{\Lambda_\con} \Lambda_\con {}_\Lambda$ are as in \ref{z local tw and ctw}. 
\end{thmx}

We explain in \ref{BB example} that this theorem extends \cite[5.18]{BB} by dropping the one-dimensional fibre assumption on $f$ and the hypersurface singularity assumption on $Y$.

\subsection{Spherical twists induced by crepant contractions more generally}

In \cref{crepant contractions section}, we consider more general crepant contractions $f \colon X \to Y$, admitting a \textit{relative tilting bundle} and where $Y$ need not be affine. For the precise setup, see \ref{crepant setup}. 

Through the technology of noncommutative schemes, we obtain a global analogue of \cref{z local intro}. Consider the sheaf of $\cO_Y$-algebras $\cA \colonequals f_* \cEnd_X(\cP)$ and write $\cA_ \con = \cA / \cI$ where $\cI$ is the ideal sheaf constructed in \cite[2.8]{DW4} and recalled in \ref{sheafy notation}.  Then, the main result of this section extends \cite[5.7]{DW4} and is the following.

\begin{thmx}[{\ref{twist geometric}}]
     Suppose that $\cA_\con$ is perfect in $\Db(\coh (Y, \cA))$. Then, the functor
    \[ \cT_Y \colonequals \Psi_\cP^{-1} \cdot \Rderived{\cHom_\cA}(\cI, -) \cdot \Psi_\cP \colon \Db(\coh X) \to \Db(\coh X)  \]
    is the twist around the functor $\Psi_\cP^{-1} \cdot (-) \Dtensor_{\cA_\con} \cA_\con$. Moreover, if $\cA_\con$ is $t$-relatively spherical, then 
    \[ \Rderived{\cHom_\cA}(\cI, -) \colon \Db(\coh (Y, \cA)) \to \Db(\coh (Y, \cA)) \]
    is an equivalence. Consequently,
       \begin{align*}
       \cT_Y \colon \Db(\coh X) \tow{\sim} \Db(\coh X) 
    \end{align*}
    is an autoequivalence of $\Db(\coh X)$.
\end{thmx}

\subsection{Conventions}

Modules will, by convention, be right modules. We will write $\Mod R$ for the category of modules over the ring $R$, $\modu R$ for the category of finitely generated modules, $\proj R$ for the category of finitely generated projective modules, and $\fl R$ for the category of finite length modules. An $R$-$S$-bimodule $L$ will be denoted ${}_R L {}_S$ and will by convention be a left $R$-module and a right $S$-module. We will often refer to an $R$-$S$-bimodule as an $S \otimes_\Z R^\op$-module. Consider bimodules ${}_R L {}_S$, ${}_R M {}_S$, and ${}_S N {}_R$. We will sometimes abbreviate the Hom-set $\Hom_S( {}_R L {}_S, {}_R M {}_S)$ as $\Hom_S( {}_R L, {}_R M)$. Similarly, the tensor product ${}_R L {}_S \, \otimes_S \, {}_S M {}_T$ will often be written as ${}_R L \otimes_S M {}_T$.

Let $\cA$ be a category. If $f$ and $g$ are morphisms in $\cA$, then we will denote the composition "$g$ then $f$" as $f \cdot g$. If, moreover, $\cA$ is additive and $X$ is an object in $\cA$, then $\add_\cA X$ denotes the full subcategory of $\cA$ whose objects are finite direct sums of summands of $X$.

Finally, given a triangulated category $\cT$ with suspension functor $\Sigma$, we will often abbreviate a distinguished triangle
\[ a \to b \to c \to \Sigma a \]
as
\[a \to b \to c \to^+. \]

\subsection*{Acknowledgements} This work was carried out as a part of the author's PhD thesis, and the author would like to thank her supervisor Michael Wemyss for his helpful guidance and relentless optimism. Moreover, many thanks goes to Wahei Hara, Matthew Pressland, and Franco Rota for helpful discussions and to Timothy De Deyn for suggestions which improved this work.

\subsection*{Funding} The author was supported by ERC Consolidator Grant 101001227 (MMiMMa).

\subsection*{Open access} For the purpose of open access, the author has applied a Creative Commons Attribution (CC:BY) licence to any Author Accepted Manuscript version arising from this submission.

    \section{Background on derived functors} \label{derived functors}
    
    The functors this paper considers often involve the derived Hom and derived tensor product functors, and our proof methods often involve taking $K$-projective, $K$-injective and $K$-flat resolutions. In this section, we briefly recall these standard constructions. 
    
    \begin{notation}
        To fix notation, let $A$ and $B$ be rings. By an $A$-$B$-bimodule, we mean a module over the ring $B \otimes_\Z A^\op$. Given any ring $S$, we will consider its associated homotopy category $\Kcat(S) \colonequals \Kcat(\Mod S)$, as well as its derived category $\Dcat(S) \colonequals \Dcat(\Mod S)$.
    \end{notation}
    
    \begin{definition} \label{hom-cplx}
    The Hom-complex functor 
\[ \Hom_A^*(-, -) \colon \Kcat(A \otimes_\Z B^\op)^\op \times \Kcat(A) \to \Kcat(B) \]
 sends complexes $L \in \Kcat(A \otimes_\Z B^\op)$ and $M \in \Kcat(A)$ to the complex $\Hom_A^*(L, M)$ which, at degree $k$, is given by the $B$-module
\[ \Hom_A^*(L, M)^k = \prod_{p + q = k} \Hom_A(L^{-p}, M^q). \]
The differential is defined by the rule
\[ d^k(f) = d_M \cdot f + (-1)^{k+1} f \cdot d_L \]
for all $f \in  \Hom_A^*(L, M)^k$.
    \end{definition}

\begin{definition} \label{tot tensor}
     Similarly, there is a tensor product functor
\[ \tot(- \otimes_B -) \colon \Kcat(B) \times \Kcat(A \otimes_\Z B^\op) \to \Kcat(A) \]
sending complexes $L \in \Kcat(B)$ and $M \in \Kcat(A \otimes_\Z B^\op)$ to the complex $\tot(L \otimes_B M)$ which, at degree $k$, is given by the $A$-module
\[ \tot(L \otimes_B M)^k = \bigoplus_{p+q=k} L^p \otimes_B M^q. \]
with differential at $k$ defined as
\[ d^k = \sum_{p+q =k} d_L^p \otimes 1_{M^q} + (-1)^p 1_{L^p} \otimes d_M^q \]
\end{definition}

\medskip
The functors in \ref{hom-cplx} and \ref{tot tensor} are used to define the derived Hom and derived tensor products as follows.

First, recall that a chain complex $I$ over an abelian category $\cA$ is \textit{$K$-injective} if for every acyclic complex $M \in \Kcat(\cA)$, then $\Hom_{\Kcat(\cA)}(M, I) = 0$.  Similarly, a complex $P$ over $\cA$ is \textit{$K$-projective} if  $\Hom_{\Kcat(\cA)}(P, M) = 0$ for every acyclic complex $M$.

Given a complex $M \in \Dcat(\cA)$, then a \textit{$K$-injective resolution} of $M$ is a chain map $M \to I$ which is a quasi-isomorphism and where $I$ is a $K$-injective complex. Moreover, we will say that $\Dcat(\cA)$ has \textit{enough $K$-injectives} if every complex $M \in \Dcat(\cA)$ admits a $K$-injective resolution. 

Furthermore, a complex $Q$ of $A^\op$-modules is \textit{$K$-flat} if for every acyclic complex $M$ of right $A$-modules, the total complex $\tot(M \otimes_A Q)$ is acyclic.

Given a complex $M$ of $A$-modules, then a $K$-flat (resp.\ $K$-projective) resolution of $M$ is a a chain map $P \to M$ and which is a quasi-isomorphism and where $P$ is a $K$-flat (resp.\ $K$-projective) complex.  

\begin{definition} \label{derived hom def}
    For complexes $M \in \Dcat(A \otimes_\Z B^\op)$ and $N \in \Dcat(\Mod A)$, let $N \to I$ be a $K$-injective resolution. Then
    \[ \Rderived{\Hom_A}(M, N) \colonequals \Hom^*_A(M, I) \in \Dcat(B). \]
    Further, the complex $\Rderived{\Hom_A}(M, N)$ does not depend (up to quasi-isomorphism) on the choice of $K$-injective resolution.
\end{definition}

\medskip
Let $P \to M$ be a $K$-projective resolution of $M$. Then, it is a classical fact that there is an quasi-isomorphism
\[ \Rderived{\Hom_A}(M, N) \cong \Hom^*_A(Q, N) \]
which is natural in $M$ and $N$.

\begin{definition}  \label{derived tensor def}
    For complexes $M \in \Dcat(\Mod A)$ and $N \in \Dcat(\Mod B \otimes_\Z A^\op)$, let $P \to N$ be a resolution in $\Dcat(\Mod B \otimes_\Z A^\op)$ which is $K$-flat as a complex of $B^\op$-modules. Then
    \[ M \Dtensor_A N \colonequals \tot(M \otimes_A P) \in \Dcat(B). \] 
\end{definition}

\section{Twist and cotwist around ring morphisms} \label{twist res scalars}

\subsection{Setting}

In this section, unless mentioned otherwise, let $p \colon A \to B$ be a ring morphism, and assume throughout that $B$ is perfect as an $A$-module. Then, there is a triangle 
\begin{equation} \label{ses for res ext twist}
    \cone(p)[-1] \tow{i} A \tow{p} B \to^+
\end{equation} 
of $A$-bimodules.

The morphism $p \colon A \to B$ induces the restriction-extension of scalars adjoint pairs $(F^\LA, F)$, $(F, F^\RA)$ on the derived category, as highlighted below. 

\begin{equation} \label{res-ext diagram}
    \begin{tikzcd}[column sep=10em]
        \Dcat(B) \arrow[rr, "F  = - \otimes^{\Lderived{}} {}_{B}  B {}_A \cong \Rderived{\Hom_{B}}{({}_A B {}_{B}, -)}" description, ""{name=F,above}] & {} & \arrow[ll, "F^{\mathrm{RA}} = \Rderived{\Hom_A}{( {}_{B} B {}_A, -)}"{name=RA}, bend left=12] \arrow[ll, "F^{\mathrm{LA}}  = - \otimes^{\Lderived{}} {}_A B {}_{B}"{name=LA}, swap, bend right=12] \Dcat(A)
        \arrow[phantom, from=LA, to=F, "\scriptstyle\boldsymbol{\perp}"description]
        \arrow[phantom, from=F, to=RA, "\scriptstyle\boldsymbol{\perp}"description]
    \end{tikzcd}
\end{equation}

In what follows, we will characterise the twist and cotwist around the restriction of scalars functor $F$, motivated by applications to geometric settings induced by crepant contractions. 

\begin{definition} \label{twist def} Following \cite{An, AL}, let $S \colon \cA \to \cB$ be a functor between triangulated categories which admits a right adjoint $R$ and a left adjoint $L$. Then,

\begin{enumerate}
    \item The \textit{twist} around $S$ is a functor $T$ which sits in a functorial triangle
    \begin{align} \label{twist triang in D}
       SR \tow{\epsilon^\rmR} 1_{\cB} \tow{\alpha_1} T \tow{\alpha_2}^+  
    \end{align}
   where $\epsilon^\rmR \colon SR \to 1_{\cB}$ is the counit of the adjunction $(S, R)$, so that there is
        \item The \textit{cotwist} around $S$ is a functor $C$  which sits in a functorial triangle
    \begin{align} \label{cotwist triang in D}
       C \tow{\beta_1} 1_{\cA} \tow{\eta^\rmR} RS \tow{\beta_2}^+  
    \end{align}
     where $\eta^\rmR \colon 1_{\cA} \to RS$ is the unit of the adjunction $(S, R)$.
\end{enumerate}
\end{definition}

\begin{remark} \label{funct cones}
    Definition \ref{twist def} is adapted from the definition in \cite{AL}. Since taking cones in triangulated categories is not functorial, it is not clear, without any additional structure on the triangulated categories $\cA$ and $\cB$, whether the functors $T$ and $C$ exist or whether they are uniquely defined (even up to isomorphism). However, this paper concerns bimodules and so taking cones is sensible in our context. In the remainder of this remark, we will specify more precisely why "taking cones is sensible in our context". Mostly, this discussion is for the reader who would like a more rigorous treatment of twists around a functors.
    
    Let $\cA$ be the DG-category with one object $x$ and with morphisms defined by $\cA(x, x) \colonequals A$, with $A$  seen as a DG-algebra concentrated in degree zero. Then, a DG-module over $\cA$ corresponds to a chain complex of $A$-modules, and so we can easily check that $\cA$ is a Morita enhancement of $\Dcat(\Mod A)$. We may define a similar enhancement for $B$ and $A \otimes_\Z B^\op$. 

    Next, consider the DG-bimodules $s \colonequals {}_B B {}_A$, $r = \Rderived{\Hom_A}({}_B B, {}_A A)$ and $l = {}_A B {}_B$. Then, the underlying exact functors are $F$, $F^{\RA}$, and $F^{\LA}$ which induce the standard restriction-extension of scalars adjunction \eqref{res-ext diagram}. 
    
    Following the definitions in \cite[\S 5]{AL}, the twist and cotwist around $s$ are the DG-bimodules
\begin{align*}
    &  t = \cone(\epsilon^\mathrm{R}_A \colon \Rderived{\Hom_A}({}_B B, {}_A B) \Dtensor_B {}_B B {}_A \to {}_A A {}_A ) \\
    &  c = \cone(\eta^{\mathrm{R}}_B \colon {}_B B {}_B \to\Rderived{\Hom_A}({}_B B, {}_B B \Dtensor_B B {}_A ) )[-1]
\end{align*}
where $\epsilon^\mathrm{R}_A$ and $\eta^{\mathrm{R}}_B$ are the counit and unit of the adjunction $(-) \Dtensor_B {}_B B {}_A \dashv \Rderived{\Hom_A}({}_B B, -)$ evaluated at the bimodules ${}_A A {}_A$ and ${}_B B {}_B$, respectively. Thus, the twist and cotwist around $F$ are $ T =  (-) \Dtensor_A t$ and $ C =  (-) \Dtensor_B c$, respectively. 
\end{remark}

\bigskip
The twist and the cotwist around a functor are closely related, so that properties of $C$ afford control over properties of $T$. Therefore, it is an interesting  problem to characterise a functor $\cT \colon \Dcat(\Mod A) \to \Dcat(\Mod B)$ as a twist around a functor $S \colon \Dcat(\Mod B) \to \Dcat(\Mod A)$ whose domain category $\Dcat(\Mod B)$ is simpler (e.g. $\Dcat(\Mod B)$ is the derived category of a finite dimensional algebra) than the codomain category $\Dcat(\Mod A)$ (e.g. $\Dcat(\Mod A)$ is the category of an infinite dimensional algebra). In particular, it is often possible to prove that the twist around a functor $S$ is an equivalence, given knowledge of the cotwist. This intuition is made precise by the "2 out of 4" lemma. 

\begin{lemma}[{\cite[5.1]{AL}}] \label{spherical criteria}
    Let $S \colon \Dcat(\Mod B) \to \Dcat(\Mod A)$ be as in \ref{twist def}. If any two of the following conditions hold, then the other two also hold. 
\begin{enumerate}
    \item The twist $T \colon \Dcat(\Mod A) \to \Dcat(\Mod A)$ is an equivalence,
    \item The cotwist $C \colon \Dcat(\Mod B) \to \Dcat(\Mod B)$ is an equivalence,
    \item With $\alpha_2$ as in \eqref{twist triang in D}, there is a natural isomorphism
    \[LT[-1] \tow{L(\alpha_2[-1])} LSR \tow{\epsilon^\rmR_R} R \]
    \item \label{iso c relates adjoints} With $\beta_2$ as in \eqref{cotwist triang in D} and letting $\eta^\rmL$ denote the unit of the adjunction $(L, S)$, there is a natural isomorphism
    \[ R \tow{\eta^\rmL_R} RSL \tow{(\beta_2)_L} CL[1] \]
\end{enumerate}
\end{lemma}

\begin{definition}[{\cite[1.1]{AL}}] \label{spherical def}
    A functor $S \colon \Dcat(\Mod B) \to \Dcat(\Mod A)$ as above is \textit{spherical} if any two of the conditions in \ref{spherical criteria} hold. When $S$ is spherical, $T$ is called a \textit{spherical twist}.
\end{definition}

In practice, it is hard to check whether the specific natural transformation in \ref{spherical criteria} \eqref{iso c relates adjoints} is an isomorphism, and so the following result will be key. 

\begin{theorem}[{\cite[1]{add}, see also \cite{An} and \cite{Rou}}] \label{sph criterion add}
       With notation as above, if the cotwist $C \colon \Dcat(\Mod B) \to \Dcat(\Mod B)$ is an equivalence, and if there exists an isomorphism of functors $R \cong CL[1]$, then $S$ is spherical. 
\end{theorem}

\subsection{The twist}

The goal of this subsection is to show that 
\[ \cT \colonequals \Rderived{\Hom}_A({}_A \cone(p)[-1] {}_A, -) \colon \Dcat(A) \to \Dcat(A) \]
is the twist around $F$ in the sense of \ref{twist def}. In order to do so, the following lemma will be important.

\begin{lemma} \label{standard counit}
     The counit $\epsilon$ of the derived tensor-hom adjunction has component at $a \in \Dcat(A)$
       \[ \epsilon_a \colon \Rderived{ \Hom_{A} }({}_{B} B {}_{A}, a) \Dtensor_{B}  B {}_{A} \to a \]
      given as the composition $j_a^{-1} \cdot \gamma_a$, where $j_a \colon a \to I$ is a $K$-injective resolution of $a$ and $\gamma_a$ is the chain map 
     \[ \gamma_a = (\gamma^j_a) \colon \tot( \Hom_{A}^*({}_{B} B {}_{A}, I)  \otimes_{B}  B {}_{A} ) \to I \]
     which at degree $j$ is specified by the evaluation map
     \[ \gamma^j_a \colon \Hom_{A}({}_{B} B {}_{A}, I^j) \otimes_{B} B {}_{A} \to I^j \]
     sending $f \otimes x \mapsto f(x)$.
\end{lemma}
\begin{proof}
This lemma is a straightforward diagram chase of the identity along the isomorphism
\[  \scriptstyle
\begin{tikzcd}[column sep=1.5em]
    \Hom_{\Dcat(B)}(\Rderived{ \Hom_{A} }(B, -), \Rderived{ \Hom_{A} }(B, -) ) \arrow[r, "\sim"] &
    \Hom_{\Dcat(A )}(\Rderived{ \Hom_{A} }(B, -) \Dtensor_{B}  B, -) 
\end{tikzcd}
\]
which defines the derived tensor-hom adjunction. As we could not find a proof in the literature, we present one in \cref{dth adj}.
\end{proof}

\begin{lemma} \label{twist cd}
There is a commutative diagram
\[
\begin{tikzcd}[column sep =4em]
    \Rderived{\Hom_A}({}_{A} B, -) \arrow[r, "{\Rderived{\Hom_A}(p, -)}"] \arrow[d, "\sim"{anchor=north, rotate=90}] & \Rderived{\Hom_A}({}_{A} A, -) \arrow[d, "\sim"{anchor=north, rotate=90}] \\
   \Rderived{\Hom_A}({}_{B} B, -) \Dtensor_{B} B {}_A \arrow[r, "\epsilon"] &  1_{\Dcat(A)} 
\end{tikzcd}
\]
of endofunctors on $\Dcat(A)$.
\end{lemma}
\begin{proof}
    To prove the lemma, it suffices to construct natural isomorphisms
    \begin{align*}
        & m \colon  \Rderived{\Hom_A}({}_{B} B, -) \Dtensor_{B} B {}_A \to \Rderived{\Hom_A}({}_{A} B, -) \\
        & n \colon \Rderived{\Hom_A}({}_{A} A, -) \to 1_{\Dcat(A)}
    \end{align*}
    such that the diagram
    \[
\begin{tikzcd}[column sep =4em]
    \Rderived{\Hom_A}({}_{A} B, a) \arrow[r, "{\Rderived{\Hom_A}(p, a)}"] \arrow[d, "\sim"{anchor=north, rotate=90}, "m_a^{-1}"'] & \Rderived{\Hom_A}({}_{A} A, a) \arrow[d, "\sim"{anchor=north, rotate=90}, "n_a"'] \\
   \Rderived{\Hom_A}({}_{B} B, a) \Dtensor_{B} B {}_A \arrow[r, "\epsilon_a"] &  a
\end{tikzcd}
    \]
    commutes for all  $a \in \Dcat(A)$. For each such $a$, fix a $K$-injective resolution $j_a \colon a \to I$ of $a$. Moreover, note that the functor $- \otimes_B B {}_A \colon B \to A$ is exact and so for any acyclic complex $M$ in $\Dcat(B)$, $\tot(M \otimes_B B {}_A)$ is acyclic. Thus, $B$ is $K$-flat in $\Dcat(B^\op)$. Hence, by the definition of the derived hom and tensor functors,
    \begin{align*}
        &  ( \Rderived{\Hom_A}({}_{B} B, a) \Dtensor_{B} B {}_A)^j = \Hom_A({}_{B} B, I^j) \otimes_{B} B {}_A, \\
        & \Rderived{\Hom_A}({}_{A} B, a)^j = \Hom_A({}_{A} B, I^j), \\
        & \Rderived{\Hom_A}({}_{A} A, a)^j = \Hom_A({}_{A} A, I^j),
    \end{align*}
    Let $m_a = (m_a^j) \colon \Rderived{\Hom_A}({}_{B} B, a) \Dtensor_{B} B {}_A \to \Rderived{\Hom_A}({}_{A} B,a)$ be the chain map which, at degree $j$, is the natural multiplication isomorphism
    \[ m_a^j \colon \Hom_A({}_{B} B, I^j) \otimes_{B} B {}_A \to \Hom_A({}_{B} B, I^j). \]
    Furthermore, set $n_a  \colon \Rderived{\Hom_A}({}_{A} A, a) \to a$ as the composition of the quasi-isomorphism $j_a^{-1}$ with the chain isomorphism 
    \[ \ev_1 = (\ev_1^j) \colon  \Rderived{\Hom_A}({}_{A} A, a) \tow{\sim} I \]
    where 
    \[ \ev_1^j \colon \Hom_A({}_{A} A, I^j) \tow{\sim} I^j \colon f \mapsto f(1). \]
    
    It is easy to check by diagram chasing that both $m = (m_a)_{a \in \Dcat(A)}$ and $n = (n_a)_{a \in \Dcat(A)}$ define natural transformations. We claim that for every $a \in \Dcat(A)$, the diagram
       \[
\begin{tikzcd}[column sep =4em]
    \Hom^*_A({}_{A} B, I) \arrow[r, "{\Rderived{\Hom_A}(p, a)}"] \arrow[d, "\sim"{anchor=north, rotate=90}, "m_a^{-1}"'] & \Hom^*_A({}_{A} A, I) \arrow[d, "\sim"{anchor=north, rotate=90}, "n_a"'] \\
     \tot( \Hom^*_A({}_{B} B, I) \otimes_{B} B {}_A) \arrow[r, "\epsilon_a"] &  a \\
\end{tikzcd}
    \]
    commutes. From \ref{standard counit}, $\epsilon_a = j_a^{-1} \cdot \gamma_a$. Hence, the above diagram commutes if and only the outer diagram in the following
       \[
\begin{tikzcd}[column sep =4em]
    \Hom^*_A({}_{A} B, I) \arrow[r, "{\Rderived{\Hom_A}(p, a)}"] \arrow[d, "\sim"{anchor=north, rotate=90}, "m_a^{-1}"'] & \Hom^*_A({}_{A} A, I) \arrow[d, "\sim"{anchor=north, rotate=90}, "\ev_1"'] \\
     \tot( \Hom^*_A({}_{B} B, I) \otimes_{B} B {}_A) \arrow[r, "\gamma_a", dashed] \arrow[d, "\gamma_a"] &  I \arrow[d,"j_a^{-1}"] \\
     I \arrow[r, "j_a^{-1}"] &  a 
\end{tikzcd}
    \]
    commutes. To see that this holds, consider first the composition $\gamma_a \cdot m^{-1}_a$. We may write this as a chain map which at degree $j$ is
    \[ \gamma_a^j \cdot (m^{-1}_a)^j \colon \Hom_A( {}_A B, I^j) \to I^j \]
    where, for $f \in \Hom_A( {}_A B, I^j)$, 
    \[  \gamma_a^j \cdot (m^{-1}_a)^j(f) = \gamma_a^j( f \otimes 1 ) = f(1) \]
    Next, consider the composition $\ev_1 \cdot \Rderived{\Hom_A(p, a)}$. At degree $j$, 
    \[ \ev_1^j \cdot \Rderived{\Hom_A(p, a)}^j(f) = \ev_1^j( f \cdot p) = f( p(1) ) = f(1). \]
    Thus, since $(\gamma_a \cdot m^{-1}_a)^j = (\ev_1 \cdot \Rderived{\Hom_A(p, a)})^j$ for all $j$, it follows that the diagram commutes. 
\end{proof}

\begin{cor} \label{twist t res}
    The functor
    \[ \cT \cong \Rderived{\Hom}_A({}_A \cone(p)[-1] {}_A, -) \colon \Dcat(A) \to \Dcat(A) \]
    is the twist around $F$.
\end{cor}
\begin{proof}
     Applying $\Rderived{\Hom_A}(-, a)$ to the triangle \eqref{ses for res ext twist} for each $a \in \Dcat(A)$ induces the functorial triangle
   \[ \Rderived{\Hom_A}({}_A B, -) \tow{f}  \Rderived{\Hom_A}({}_A A, -)  \tow{g} \Rderived{\Hom_A}(\cone(p)[-1], -) \to^+ \]
   where $f = \Rderived{\Hom_A}(p, -)$ and $g = \Rderived{\Hom_A}(i, -)$. 
   
   Given \ref{twist cd}, there is commutative diagram
   \[ \begin{tikzcd}[column sep =2.5em] \label{diagram twist}
       \Rderived{\Hom_A}(B, -) \arrow[d, "\sim"{anchor=north, rotate=90}]&[-3.5em] {}  \arrow[r, "{f}"{yshift=0.5}] & {} &[-3.5em]  \Rderived{\Hom_A}(A, -) \arrow[d, "\sim"{anchor=north, rotate=90}] &[-3.5em] {} \arrow[r, "{g}"{yshift=0.5}] &  {} &[-3.5em] \Rderived{\Hom_A}(\cone(p)[-1], -) \arrow[d, equals] &[-3.5em] {} \arrow[r] & {} &[-3.5em] {}^+ \\
        \Rderived{\Hom_A}(B, -) \Dtensor_{B} B  & {} \arrow[r, "\epsilon"{yshift=0.5}] &  {} & 1_{\Dcat(A)} & {}  \arrow[r, "u"{yshift=0.5}]  &  {} & \Rderived{\Hom_A}(\cone(p)[-1], -) &  {}  \arrow[r] & {} &[-3.5em] {}^+ 
   \end{tikzcd}
   \]
   where $u$ is defined so that the diagram commutes. Hence, since the first row is exact, so is the bottom. Thus, because $\epsilon$ is the counit, it follows from \ref{twist def} that $\cT$ is the twist around $F$.

   Alternatively, if taking the definition of the twist around $F$ as in \ref{funct cones}, then the twist around $F$ is $T = - \Dtensor_A t$ for a specific complex $t$ of $A$-bimodules. Since $\epsilon$ is the counit of the appropriate tensor-hom adjunction, and the vertical morphisms in the diagram \eqref{diagram twist} are functorial, then there is an isomorphism $t \cong \Rderived{\Hom}_A(\cone(p)[-1], A)$ of complexes of bimodules. Moreover, notice that $B \in \Kb(\proj A)$, and so 
   \[ T = - \Dtensor_A \Rderived{\Hom_A}(\cone(p)[-1], A) \cong \Rderived{\Hom_A}(\cone(p)[-1], -) = \cT \]
   by \cite[2.10(1)]{IR}.
\end{proof}

\begin{cor} \label{twist for epi}
    Suppose further that $p \colon A \to B$ is a surjection. Then, the twist around $F$ is 
    \[ \cT \colonequals \Rderived{\Hom}_A({}_A \ker p {}_A, -) \colon \Dcat(A) \to \Dcat(A) \]
\end{cor}
\begin{proof}
    If $p$ is a surjection, then the short exact sequence
    \[ 0 \to \ker p \to A \to B \to 0 \]
    induces a triangle
    \[ \ker p \to A \to B \to^+ \]
    so that we may conclude that $\ker p = \cone(p)[-1]$. 
\end{proof}

\subsection{The cotwist} 

The goal of this section is to compute the cotwist around $F$. This requires a handful of lemmas.

First, let $Q \to B$ be a $K$-projective resolution of $B \in \Dcat(A \otimes_\Z B^\op)$. Then, it is straightforward to check that the quasi-isomorphism $Q \to B$ induces a quasi-isomorphism
 \[ s \colon {}_B B \Dtensor_A B {}_B \to \tot({}_B Q \otimes_A B {}_B ).\]

\begin{lemma}
    There are functorial isomorphisms
    \begin{align}
        \nu & \colon \Rderived{\Hom_B}( {}_B B {}_B, -) \to 1_{\Dcat(B)}, \label{id iso B} \\
        \Rderived{\Hom_B}(s, -) & \colon \Rderived{\Hom_B}(\tot({}_B Q \otimes_A B {}_B), -) \to \Rderived{\Hom_B}({}_B B \Dtensor_A B {}_B, -). \label{qiso 2}
    \end{align}
\end{lemma}
\begin{proof}
    Let $q \colon Q \to B$ be a $K$-projective resolution of $B \in \Dcat(A \otimes_\Z B^\op)$ and fix a complex $c \in \Dcat(B)$.
    
   The construction of $\nu$ in \eqref{id iso B} is essentially the same as the construction of $n$ in \ref{twist cd}. To construct the quasi-isomorphism in \eqref{qiso 2}, observe that the quasi-isomorphism
   \[ s \colon {}_B B \Dtensor_A B {}_B \to \tot({}_B Q \otimes_A B {}_B )\]
    induces a quasi-isomorphism
   \[ \Rderived{\Hom_B}(s, c) \colon \Rderived{\Hom_B}( \tot({}_B Q \otimes_A B {}_B), c) \to \Rderived{\Hom_B}({}_B B \Dtensor_A B {}_B, c) \]
   which is natural in $c$.
\end{proof}

\begin{lemma}
    Let $\cI$ denote the subcategory of $\Dcat(B)$ consisting of $K$-injective complexes. Then, in $\cI$, there are quasi-isomorphisms 
    \begin{align}
        & f \colon  \Hom_B^*( \tot({}_B Q \otimes_A B {}_B), -) \to \Hom_A^*({}_B Q {}_A, \tot(- \otimes_B B {}_A)), \label{iso3} \\
        &  g \colon \Hom_A^*({}_B Q {}_A, \tot(- \otimes_B Q {}_A) ) \to \Hom_A^*({}_B Q {}_A, \tot(- \otimes_B B {}_A) ). \label{iso4}
    \end{align}
\end{lemma}
\begin{proof}
    Let $I$ be a $K$-injective complex in $\Dcat(B)$ and let $q \colon Q \to B$ be a $K$-projective resolution of $B \in \Dcat(A \otimes_\Z B^\op)$.
    
   To prove \eqref{iso3}, define the chain map 
   \[ f_I = (f_I^j)\colon \Hom_B^*( \tot({}_B Q \otimes_A B {}_B), I) \to \Hom_A^*(Q, \tot(I \otimes_B  B {}_A)). \]
    where
   \[ f_I^j \colon \prod_{p+q = j} \Hom_B(Q^{-p} \otimes_A B_B, I^q) \to \prod_{p+q=j} \Hom_A(Q^{-p}, I^q \otimes_B B_A) \]
  is such that, given a tuple $(\alpha^{p, q}) \in \prod_{p+q = j} \Hom_B(Q^{-p} \otimes_A B_B, I^q)$, 
   \[f^j_I(\alpha^{p, q}) \colon x \mapsto \alpha^{p, q}(x \otimes_A 1) \otimes_B 1. \]
   The morphism $f_I^j$ is an isomorphism because it is a product of the composition of isomorphisms
   \begin{align*}
       & v_I^{p,q} \colon \Hom_B(Q^{-p} \otimes_A B_B, I^q) \to \Hom_A(Q^{-p}, \Hom_B( {}_A B {}_B, I^q))  \\
       & w_I^{p, q} \colon (Q^{-p}, \Hom_B( {}_A B {}_B, I^q)) \to (Q^{-p}, I^q \otimes_B B_{A})
   \end{align*}
   where $(v_I^{p,q}(\alpha^{p,q}))(x)(y) = \alpha^{p,q}(x \otimes y)$ and $(w+I^{p, q}(\beta^{p,q}))(x) = \beta^{p, q}(x)(1) \otimes_B 1$. It is not hard to check that $f$ defined this way commutes with the differentials, so that it is indeed a chain map.

   To prove \eqref{iso4}, note that the quasi-isomorphism $q \colon Q \to {}_B B {}_A$ induces the quasi-isomorphism
   \[ g \colon \Hom_A^*({}_B Q {}_A, \tot(- \otimes_B Q {}_A) ) \to \Hom_A^*({}_B Q {}_A, \tot(- \otimes_B B {}_A) ). \qedhere \] 
\end{proof}

To compute the cotwist, it will be helpful to know an explicit formulation of the unit of the derived tensor-hom adjunction. 

\begin{lemma}
    Let $c \in \Dcat(B)$ and $p_c \colon P \to c$ be a $K$-projective resolution of $c$. Moreover, let $q \colon Q \to B$ be a $K$-projective resolution of $B \in \Dcat(A \otimes_\Z B^\op)$. Then, under the quasi-isomorphism 
    \[ h \colon \Rderived{\Hom_A}( {}_B B_{A}, c \Dtensor_B B_{A} ) \tow{\sim} \Hom_A^*({}_B Q {}_A, \tot( P \otimes_B Q {}_A) ), \]
    the unit $\eta$ of the derived tensor-hom adjunction has component $ \eta_c = \lambda_c \cdot j_c$ where $\lambda_c$ is the chain map 
    \[ \lambda_c  = (\lambda_c^j) \colon P \to \Hom_A^*({}_B Q {}_A, \tot( P \otimes_B Q {}_A)  ) \]
    with
    \[ \lambda_c^j \colon P^j \to \prod_{p+q = j} \Hom_A( Q^{-p}, \bigoplus_{r+s = q} P^r \otimes_B Q^s) \]
    sending $x \mapsto x \otimes - \in \Hom_A(Q^{-p}, P^j \otimes_B Q^{-p})$. 
\end{lemma}
\begin{proof}
    Similar to the proof of \ref{standard counit},  this lemma is proved by diagram chasing the identity morphism $\id \in \Hom_{\Dcat(A)}(c \Dtensor_B B {}_A, c \Dtensor_B B {}_A)$ across the standard isomorphism defining the tensor-hom adjunction.
\end{proof}

\begin{remark} \label{beta remark}
    Note that we may take the resolution $q \colon Q \to {}_B B {}_A$ to be a chain map $(q^j)$ which vanishes in all degrees but zero. In which case, there is a chain map
\[ \beta = (\beta^j) \colon \tot({}_B Q \otimes_A B {}_B) \to {}_B B {}_B \] 
where $\beta^j = 0$ for $j \neq 0$ and $\beta^0$ is induced by composing $q^0 \otimes_A 1_B \colon Q^0 \otimes_A B {}_B \to {}_B B \otimes_A B {}_B$ with the multiplication isomorphism $m \colon  {}_B B \otimes_A B {}_B \to  {}_B B {}_B$.
\end{remark}

Since there is a quasi-isomorphism $s \colon {}_B B \Dtensor_A B {}_B \to \tot(Q \otimes_A B {}_A)$, it thus follows from \ref{beta remark} that there is a map 
\[ s \cdot \beta \colon {}_B B \Dtensor_A B {}_B \to {}_B B {}_B. \] 

\begin{lemma} \label{cotwist cd}
    There is a commutative diagram 
    \[
    \begin{tikzcd}[column sep=5.5em]
        \Rderived{\Hom_B}({}_B B {}_B, -) \arrow[d, "\sim"{anchor=north, rotate=90}] &[-6
        em] {} \arrow[r, "{\Rderived{\Hom_B}(s \cdot \beta, -)}"] & {} &[-6.5em] 
 \Rderived{\Hom_B}( {}_B B \Dtensor_A B {}_B, -) \arrow[d, "\sim"{anchor=north, rotate=90}] \\
        1_{\Dcat(B)} & {} \arrow[r, "\eta"] & {} & 
 \Rderived{\Hom_A}({}_B B {}_A, - \Dtensor_B B {}_A)
    \end{tikzcd}
    \]
\end{lemma}
\begin{proof}
    
    Let $c \in \Dcat(B)$. Fix $K$-injective and $K$-projective resolutions $j_c \colon c \to I$ and $p_c \colon P \to c$, respectively. We claim that the following diagram commutes.

    \begin{tikzcd}
        \Rderived{\Hom_B}({}_B B {}_B, c) \arrow[d, equals] \arrow[r, "{\Rderived{\Hom_B}(s \cdot \beta, c)}"] &  \Rderived{\Hom_B}( {}_B B \Dtensor_A B {}_B, c) \arrow[d, "{\Rderived{\Hom_B}(s, c)^{-1}}"] \\
        \Rderived{\Hom_B}({}_B B {}_B, c) \arrow[d, equals] \arrow[r, "{\Rderived{\Hom_B}(\beta, c)}"] &  \Rderived{\Hom_B}( \tot({}_B Q \otimes_A B_B), c)  \arrow[d, equals] \\
        \Hom^*_B({}_B B {}_B, I) \arrow[r, "{\Hom_B^*(\beta, I)}"] \arrow[d, "\ev_1"', "\sim"{anchor=north, rotate=90}] & \Hom_B^*(\tot({}_B Q \otimes_A B_B), I ) \arrow[d, equals] \\
        I \arrow[d, "p_c \cdot j_c"', "\sim"{anchor=north, rotate=90}] & \Hom_B^*(\tot({}_B Q \otimes_A B_B), I )  \arrow[d, equals] \\
        P \arrow[d, "\lambda_c"'] & \Hom_B^*(\tot({}_B Q \otimes_A B_B), I ) \arrow[d, "f"] \\
        \Hom_A^*({}_B Q {}_A, \tot( P \otimes_B Q {}_A) ) \arrow[d, "{\Hom_A^*({}_B Q {}_A, \tot(p_c \cdot j_c \otimes_B Q {}_A)^{-1})}"', "\sim"{anchor=north, rotate=90}] &  \Hom_A^*({}_B Q {}_A, \tot(I \otimes_B B_A) ) \arrow[d, equals] \\
        \Hom_A^*({}_B Q {}_A, \tot( I \otimes_B Q {}_A) ) \arrow[d, "{\Hom_A^*({}_B Q {}_A, \tot(I \otimes_B q))}"', "\sim"{anchor=north, rotate=90}] & \Hom_A^*({}_B Q {}_A, \tot(I \otimes_B B_A) ) \arrow[d, equals] \\
        \Hom_A^*({}_B Q {}_A, \tot( I \otimes_B B {}_A) ) \arrow[r, equals] & \Hom_A^*({}_B Q {}_A, \tot( I \otimes_B B {}_A) ) \\
    \end{tikzcd}

    It is clear that the top two squares commute. Let $\alpha \in \Hom^*_A({}_B B {}_B, I)$ be a degree $j$ morphoism. Then, $\alpha = \alpha^j \in \Hom_A(B, I^j)$. Let us first chase the map $\alpha$ along the left-hand column. Well, $\ev_1(\alpha^j) = \alpha^j(1) \in I^j$. Hence, $p_c \cdot i_c( \alpha^j(1) ) \in P^j$ so that 
    \[ \lambda_c \cdot p_c \cdot i_c( \alpha^j(1) ) =  \lambda^j_c \cdot p_c \cdot i_c( \alpha^j(1) ) = p_c \cdot i_c( \alpha^j(1) ) \otimes (-) \in \prod_{p} \Hom_A( Q^{-p}, P^j \otimes_B Q^{-p}). \]
    Note that we may rewrite this map as
    \[\tot(p_c \cdot i_c \otimes_B Q {}_A)( \alpha^j(1) \otimes(-) )\]
    Hence,
    \[ \big( \Hom_A^*({}_B Q {}_A, \tot(p_c \cdot i_c \otimes_B Q {}_A)^{-1} \big)( p_c \cdot i_c( \alpha^j(1) ) \otimes (-) ) = \alpha^j(1) \otimes (-). \]
    Finally, 
    \[ \big(\Hom_A^*({}_B Q {}_A, \tot(I \otimes_B q)) \big)(\alpha^j(1) \otimes (-) ) = \prod_{p} \alpha^j(1) \otimes_B q^{-p}(-) = \alpha^j(1) \otimes_B q^0(-) \]
    because $q^j$ vanish for all $j \neq 0$ by the choice of $q$ in \ref{beta remark}. Moreover, Since $q_0(x) \in B$ for all $x \in Q^0$ and $\alpha^j$ is a right $B$-module homomorphism, 
    \[ \alpha^j(1) \otimes_B q^0(-) = \alpha^j(q^0(-)) \otimes_B 1. \]

    Next, let us chase $\alpha = \alpha^j$ along the right-hand column. Well, 
    \[ \Hom^*_B(\beta, I)(\alpha^j) \in \Hom_B(Q^{0} \otimes B, I^j) \] 
    is such that 
    \[ \Hom^*_B(\beta, I)(\alpha^j)(x \otimes y) = \alpha^j \cdot \beta(x \otimes y) = \alpha^j \cdot m \cdot (q^0 \otimes 1)(x \otimes y) = \alpha^j(q^0(x) y) = \alpha^j \cdot q^0(xy).\] 
    Moreover, 
    \[ f( \alpha^j \cdot \beta)(x) = \alpha^j \cdot \beta(x \otimes_A 1) \otimes_B 1 = \alpha^j(q^0(x)) \otimes_B 1. \]
    Thus, the diagram commutes. By composing each column with the functorial isomorphisms
    \[ \Hom_A^*({}_B Q {}_A, \tot( I \otimes_B B {}_A) ) \tow{\sim} \Rderived{\Hom_A}({}_B B {}_A, I \Dtensor_B B {}_A) \tow{\sim} \Rderived{\Hom_A}({}_B B {}_A, c \Dtensor_B B {}_A) \]
    it follows that the diagram in the statement commutes.  
\end{proof}

\begin{cor} \label{cotwist gen}
    Let $C(s \cdot \beta)$ denote the cone of the morphism $s \cdot \beta \colon {}_B B \Dtensor_A B {}_B \to {}_B B {}_B$ in $\Dcat(B \otimes_\Z B^\op)$. Then, the functor
   \[ \cC \cong \Rderived{\Hom_B}( {}_B C(s \cdot \beta) {}_B, -) \]
   is the cotwist around $F$.
\end{cor}
\begin{proof}
    For each $c \in \Dcat(B)$, applying $\Rderived{\Hom_B}(-, c)$ to the triangle 
    \[ {}_B B \Dtensor_A B {}_B \tow{s \cdot \beta} {}_B B {}_B \to C(s \cdot \beta) \to^+ \]
    induces first exact row in the diagram
    \[ \begin{tikzcd}[column sep=4.5em] \label{diagram cotwist}
        \Rderived{\Hom_B}(C(s \cdot \beta, c), c) \arrow[d, equals] &[-5.5em] \arrow[r] &[-2em] {} &[-5.5em] \Rderived{\Hom_B}(B , c) \arrow[d, "\sim"{anchor=north, rotate=90}] &[-5.5em] \arrow[r, "{\Rderived{\Hom_B}(s \cdot \beta, c)}"{yshift=0.4em}] & {} &[-5.6em] \Rderived{\Hom_B}( B \Dtensor_A B, c) \arrow[d, "\sim"{anchor=north, rotate=90}]&[-5.5em] \arrow[r] &[-3em] {} &[-5.5em] {}^+\\
        \Rderived{\Hom_B}(C(s \cdot \beta, c) {}_B, c) & \arrow[r, "v"] & {} & 1_{\Dcat(B)} & \arrow[r, "\eta"] & {} & \Rderived{\Hom_A}(B, - \Dtensor_B B) & \arrow[r] & {} & {}^+
    \end{tikzcd}
    \]
    where the morphism $v$ is defined so that the left-hand square commutes. The right-hand square commutes by \ref{cotwist cd}, and so the bottom row is also exact. Thus, we may conclude from \ref{twist def} that $\cC \cong  \Rderived{\Hom_B}(C(s \cdot \beta, c) {}_B, c)$. 

    Alternatively, if taking the definition of the cotwist as in \ref{funct cones}, recall that the cotwist around $F$ is $C = - \Dtensor_A c [-1]$ for a specific complex $c$ of $B$-bimodules. Since $\eta$ is the unit of the appropriate tensor-hom adjunction, and since the vertical morphisms in diagram \eqref{diagram cotwist} are functorial, then $c[-1] \cong \Rderived{\Hom_B}( {}_B C(s \cdot \beta) {}_B, B)$ as complexes of bimodules. Because $B \in \Kb(\proj A)$, then 
   \[ C = - \Dtensor_B \Rderived{\Hom_B}( {}_B C(s \cdot \beta) {}_B, B) \cong \Rderived{\Hom_B}( {}_B C(s \cdot \beta) {}_B, -) = \cC \]
   by \cite[2.10(1)]{IR}.
\end{proof}

\begin{prop} \label{ctw res cohom in 2 degrees}
    Fix $t \in \Z$ with $t \notin \{1, 0, -1\}$. Suppose that $\cohom{k}( {}_B B \Dtensor_A B {}_B) = 0$ for all $k \neq 0, t$. Then,
    \[ \cC = \Rderived{\Hom_B}( {}_B \Tor_{-t}^A(B, B) {}_B, -)[t-1]  \]
\end{prop}
\begin{proof}
    The triangle of $B$-bimodules
    \[ {}_B B \Dtensor_A B {}_B \tow{s \cdot \beta} {}_B B {}_B \to {}_B C(s \cdot \beta) {}_B \to^+ \]
    induces the long exact sequence
    \[ \hdots \to \cohom{k}(B \Dtensor_A B) \tow{\cohom{k}(s \cdot \beta)} \cohom{k}(B) \to \cohom{k}(C(s \cdot \beta)) \to \cohom{k+1}(B \Dtensor_A B) \to \hdots \]
    on cohomology. Since $B$ is concentrated in degree zero and $\cohom{k}( {}_B B \Dtensor_A B {}_B)$ vanishes for all $k \neq 0, t$, it must be that $\cohom{k}(C(s \cdot \beta)) = 0$ for all $k \neq -1, 0, t-1$. Moreover, note that $\cohom{0}(s \cdot \beta)$ is an isomorphism. Hence,
    \[ \cohom{-1}(C(s \cdot \beta)) = \cohom{0}(C(s \cdot \beta)) = 0. \]
    Therefore, $C(s \cdot \beta)$ is concentrated in degree $t-1$. It  follows from this and the long exact cohomology sequence that
    \[ C(s \cdot \beta) \cong \cohom{t-1}(C(s \cdot \beta))[t-1] \cong \cohom{t}(B \Dtensor_A B)[t-1]. \]
    Thus, the statement holds.
\end{proof}

\section{Frobenius exact setting and fundamentals} \label{setup and fun}

With geometric applications in mind, we would like to construct a spherical twist around the restriction of scalars functor. This section introduces an algebraic setting which allows us to do so.

Given an object $X$ in a Frobenius exact category $\cE$  satisfying mild conditions, we will specify when the restriction of scalars functor $F$ induced by quotient morphism $\pi \colon \End_\cE(X) \to \uEnd_\cE(X)$ is spherical. The key point is roughly that $F$ is spherical if the action of the suspension functor on the object $x$ is periodic up to additive closure. This is proved in \cref{spherical by frob}. 

In this section, we introduce the setup and some fundamental results necessary for \cref{spherical by frob}. Most notably,  we show that the suspension functor acts periodically on $X$ up to additive closure if and only if $X$ has relatively spherical properties. This characterisation is important to our applications because relatively spherical properties are easier to control in the geometric setting. In the process, \cref{partial min res section} constructs what we call partially minimal projective resolutions. 

\subsection{Setup} \label{frob set up sect}

 Recall that an additive category $\cA$ is said to be $\cR$-linear, where $\cR$ is a ring, if the Hom-sets in the category are finitely generated $\cR$-modules, and composition respects this structure. Given such an additive category, the collection of projective objects in $\cA$ will be denoted by $\proj \cA$, and an object $P \in \cA$ will be called an additive progenerator if $\add_\cE P = \proj \cA$.

With an eye towards applications, we will often work within the following set up.
\begin{setup} \label{frob set up}
   Let $\cR$ be a Cohen-Macaulay (CM) noetherian local ring of Krull dimension $d$ with coefficient field $k$ and canonical module $\omega_R$. Let $\cE$ be an $\cR$-linear idempotent complete Frobenius exact category with an additive progenerator $P$. Assume that the stable category $\cD$ of $\cE$ is Krull-Schmidt (see \cref{app: self inj} for a brief overview of Frobenius exact categories and their stable categories).
\end{setup}

\begin{notation}
    Since $\cD$ is triangulated, we will denote its suspension functor by $\Sigma$. 
\end{notation}

We are particularly interested in examples where $\cE = \CM \cR$, the category of maximal Cohen-Macaulay modules over a Gorenstein ring $\cR$, as these find natural applications in geometry.

\begin{notation} \label{notation for algs}
Suppose that $\cE$ is as in Setup \ref{frob set up}. Then, given an object $X \in \cE$, set $\Lambda = \End_\cE(X)$ and let $[\proj \cE]$ be the ideal in $\Lambda$ of maps $X \to X$ which factor through a projective object in $\cE$. Furthermore, let 
\[ \Lambda_\con = \End_\cD(X) = \uEnd_{\cE}(X) \colonequals \Lambda / [\proj \cE]. \]
\end{notation}

   As the algebra $\Lambda$ is module finite over $\cR$, its derived category is equipped with the Nakayama functor defined in \cite[\S 3]{IR}, namely 
    \[ - \Dtensor_{\Lambda} \omega_{\Lambda}  \colon \Dcat^{-}(\modu \Lambda) \to \Dcat^{-}(\modu \Lambda) \] 
    where $\omega_{\Lambda} \colonequals \Rderived{\Hom}_\cR(\Lambda, \omega_\cR)$ is a dualizing complex for $\Dcat(\modu \Lambda)$ in the sense of \cite{Ye}. Moreover, \cite[3.5,3.7]{IR} proves a form of Serre duality by showing the existence of the following functorial isomorphisms.
    \begin{enumerate} 
        \item For any $a \in \Db(\modu \Lambda)$ and $b \in \Kb(\proj \Lambda)$
        \begin{align} \label{derived serre duality} 
             \Rderived{\Hom}_{\Lambda}(a, b \Dtensor_{\Lambda} \omega_{\Lambda}) \cong \Rderived{\Hom}_\cR(\Rderived{\Hom}_{\Lambda}(b, a), \omega_\cR).
        \end{align}
        \item For any $a \in \Db(\fl \Lambda)$ and $b \in \Kb(\proj \Lambda)$
        \begin{align} \label{serre duality} 
            \Hom_{\Dcat(\modu \Lambda)}(a, b \Dtensor_{\Lambda} \omega_{\Lambda} [d]) \cong D  \Hom_{\Dcat(\modu \Lambda)}(b, a) .
        \end{align}
    \end{enumerate}

    \begin{remark}
        Although \cite[\S 3]{IR} states the above results for Gorenstein rings, their proof applies word for word to any CM ring equipped with a canonical module once $\cR$ is replaced by $\omega_\cR$.
    \end{remark}
    
    We will often place additional assumptions on $X$, which are naturally satisfied in the applications we have in mind.
    
    \begin{setup} \label{X setup}
        With the assumptions as in \ref{frob set up} and notation as above, let $X \in \cE$ be such that $X \cong P \oplus \bigoplus_{i = 1}^n X_i^{a_i}$, where $a_i \in \N$ with $a_i > 0$, and $X_i$ are pairwise non-isomorphic, indecomposable objects in $\cD$.
    \end{setup} 
   
    \begin{remark} \label{sigma X satisfies X setup}
        Since the suspension functor $\Sigma$ is an automorphism of $\cD$, it is clear that if $X \in \cE$ satisfies \ref{X setup}, then so does $P \oplus \Sigma^k X$.
    \end{remark}
    
    Note that for any $X \in \cE$, there is a $Q \in \proj \cE$ such that $Q \oplus X$ satisfies \cref{X setup}. To see this, note that, since $\cD$ is Krull-Schmidt, any $X \in \cE$ is stably isomorphic to $\bigoplus_{i = 1}^n X_i^{a_i}$ with $X_i$ and $a_i$ as in \cref{X setup}. Consequently, by \ref{lift of stable iso}, there is an isomorphism $Q_1 \oplus X \cong Q_2 \oplus \bigoplus_{i = 1}^n X_i^{a_i}$ in $\cE$ with $Q_1, Q_2 \in \proj \cE$. Thus, letting $Y_1 = Q_2 \oplus X_1$ and $Y_i = X_i$ for $i \neq 1$, we may write
    \[ (P \oplus Q_1 \oplus Q_2^{a_1-1}) \oplus X \cong P \oplus \bigoplus_{i = 1}^k Y_i^{b_i}. \]

\begin{notation} \label{def of ei}
    Given $X = P \oplus X'$ with $X' = \bigoplus_{i=1}^n X_i^{a_i}$ as in \ref{X setup}, let $I = \{1, 2, \hdots, n\}$ and consider the following idempotents in $\Lambda$
    \begin{enumerate}
        \item $e_i \colon X \to X_i$, projection onto the summand $X_i$ for $i \in I$,
        \item $e_0 \colon X \to P$, projection onto $P$,
        \item $e_{X'} \colon X \to X'$, projection onto $X'$. 
    \end{enumerate}
    With this notation, 
    \begin{align*}
        & \cE(X, X_i) = e_i \Lambda \ \ \text{ and } \cD(X, X_i) = e_i \Lambda_\con, \\
        & \cE(X, X') = e_{X'} \Lambda \text{ and } \cD(X, X') = e_{X'} \Lambda_\con, \\
        & \cE(X, P) = e_0 \Lambda.
    \end{align*}
\end{notation}

\begin{lemma}  \label{finitely many simples}
        Under \cref{X setup}, $\Lambda_\con$ is semi-perfect (see e.g.\ \cref{app: self inj} for some background on semi-perfect algebras) and, therefore, has finitely many simple modules up to isomorphism. 
\end{lemma}
\begin{proof}
        Since $\add_{\cD}(X)$ is Krull-Schmidt (and thus idempotent complete), there is an equivalence $\cD(X, -) \colon \add_\cD (X) \to \proj \Lambda_\con$ \cite[2.3]{HK14}. Therefore, $\proj \Lambda_\con$ is Krull-Schmidt. It follows from \ref{semi perf def} that  $\Lambda_\con$ is semi-perfect. The fact that semi-perfect rings have finitely many simple modules up to isomorphism is well-known (see, for example, \ref{semi perfect gives finitely many simples}). 
\end{proof}

Since $\Lambda_\con$ is semi-perfect and $e_i \Lambda_\con$ is an indecomposable summand of $\Lambda_\con$, it follows that $e_i \Lambda_\con / \radd(e_i \Lambda_\con)$ is simple by \cite[4.1 (3) $\Rightarrow$ (1)]{HK14}. In fact, these are all of the simple $\Lambda_\con$ modules up to isomorphism.

\begin{notation} \label{simples not}
    Write $S_i \colonequals e_i \Lambda_\con / \radd(e_i \Lambda_\con)$.
\end{notation}

 We conclude this section by recalling well-known results which will often be used throughout the paper. 

    \begin{lemma} \label{F is faithful on ext1}
    For any $A, B \in \modu \Lambda_\con$,
        \[ \Ext^1_{\Lambda_\con}(A, B) \cong \Ext^1_{\Lambda}(A, B) \]
    \end{lemma}
    \begin{proof}
        This is standard. Since multiplying by idempotents is exact, we may view $\modu \Lambda_\con$ as an extension closed subcategory of $\modu \Lambda$, and so the result follows.
    \end{proof}   

\begin{lemma} \label{useful lemma}
     Let $\cA$ be an additive idempotent complete category and consider objects $X, Y, Z \in \cA$. 
     \begin{enumerate}
         \item If $Y \in \add_\cA X$, then functor $\Hom_\cA(X, -)$ induces an isomorphism
    \[ \Hom_\cA(Y, Z) \to \Hom_{\End_{\cA}(X)}( \Hom_{\cA}(X, Y), \Hom_{\cA}(X, Z)). \]
    which is natural in $Y$ and $Z$. \label{yoneda}
        \item The canonical composition map
    \[ \Hom_\cA(Y, Z) \otimes_{\End_\cA(Y)} \Hom_\cA(X, Y) \to \Hom_\cA(X, Z) \]
    is an isomorphism if $Z \in \add_\cA(Y)$ or $X \in \add_\cA(Y)$. \label{composition is iso}
     \end{enumerate}
\end{lemma}

\subsection{Partially minimal projective resolutions} \label{partial min res section}

The key tool for characterising when the action of the suspension functor on an object is periodic, up to additive closure, will be a (partially) minimal projective resolution of $\Lambda_\con$.  Ideally, we would like to construct the minimal projective resolution of $\Lambda_\con$ as a $\Lambda$-module, but since $\Lambda$ is not necessarily semi-perfect, such a resolution need not exist. However, it is possible to construct a projective resolution which has sufficiently nice properties.

\begin{notation} \label{maximal semisimple}
   Under \cref{X setup} and assuming the notation \ref{simples not}, write 
    \[ \cS \colonequals \bigoplus_{i \in I} S_i. \]
\end{notation}

\begin{definition} 
     Under \cref{X setup}, we will say that a projective resolution
     \[ \hdots \tow{f_3} Q_2 \tow{f_2} Q_1 \tow{f_1} Q_0 \tow{f_0} M \to 0 \]
     of finitely generated $\Lambda$-modules is \textit{partially minimal} if $\Hom_\Lambda(f_i, \cS) = 0$ for all $i >0$. 
\end{definition}

We will prove in \ref{existence_min_res} that under \cref{X setup} every finitely generated $\Lambda$-module admits such a resolution. In order to construct it, we will need to establish some technical results. Our approach is to define an ideal which is somewhat similar to the radical and then to adapt the standard proof that minimal resolutions exist to our more general setting. 

\begin{definition} \label{rad0 def}
   Under \cref{X setup}, let $M \in \modu \Lambda$. Set 
    \[ I(M) = \left \{ Y \subseteq M \: \middle\vert \: \parbox{0.6 \linewidth}{$Y$ is a submodule with $M e_0 \Lambda \subseteq Y$, and \\ if there is a submodule $V$ with $Y \subseteq V \subseteq M$, then $V = M$} \right\} \]
    and  
    \[ \radd_0 M = \bigcap_{Y \in I(M)} Y. \]
\end{definition}

\begin{remark} \label{dragons}
    Note that in \ref{rad0 def}, $Y$ is not necessarily proper. Consequently, $\radd_0 M$ is the intersection of $M$ with all of the maximal submodules of $M$  which contain the set $M e_0$. It is clear that $\radd M \subseteq \radd_0 M$, since $\radd M$ is the intersection of all maximal submodules of $M$. Moreover, if $M$ does not have any proper submodules which contain the set $M e_0$, then $\radd_0 M = M$. For example, $\radd_0 (e_0 \Lambda) = e_0 \Lambda$.
\end{remark}

The following lemma is an adaptation of the usual characterisation of essential epimorphisms.

\begin{lemma} \label{partially essential}
    Under \cref{X setup}, let $\phi \colon M \to N$ be an epimorphism in $\modu \Lambda$. Then, the following are equivalent.
    \begin{enumerate}
        \item $\ker \phi \subseteq \radd_0 M$. \label{pe c1}
        \item For any submodule $V \subseteq M$ with $M e_0 \subseteq V$ and $\ker \phi + V = M$, then $V = M$. \label{pe c2}
        \item Let $\alpha \colon L \to M$ be a morphism in $\modu \Lambda$ with $M e_0 \subseteq \Img \alpha$. Then, $\alpha$ is an epimorphism provided that the composition $\phi \cdot \alpha$ is an epimorphism. \label{pe c3}
    \end{enumerate}
\end{lemma}
\begin{proof}
    Suppose that \eqref{pe c1} holds. and consider $V$ as in \eqref{pe c2}. If $V = M$, then there is nothing to prove. If $V \subsetneq M$ with $M e_0 \subseteq V$ and $\ker \phi + V = M$, then there exists a maximal submodule $W \subsetneq M$ with $M e_0 \subseteq V \subseteq W$. It follows that $M = \ker \phi + W = W$, which is a contradiction. Hence, $V = M$, so that \eqref{pe c2} holds.

    Next, suppose that \eqref{pe c2} holds. Then, given a morphism $\alpha \colon L \to M$ with $M e_0 \subseteq \Img \alpha$ and $\phi \cdot \alpha$ an epimorphism, it easy to see that $\Img \alpha + \ker \phi = M$, since for all $m \in M$, there exists an $\ell \in L$ such that $\phi(m) = \phi(\alpha(\ell))$. Hence, $m - \alpha(\ell) \in \ker \phi$ and we conclude that $\Img \alpha + \ker \phi = M$. Thus, by \eqref{pe c2}, $\Img \alpha = M$ and \eqref{pe c3} holds.

    Finally, suppose that \eqref{pe c3} holds. If $M$ does not have proper submodules $V \subseteq M$ with $M e_0 \subseteq M$, then by \ref{dragons}, $\radd_0 M = M$ and so \eqref{pe c1} clearly holds. So, suppose that there is a maximal submodule $V \subset M$ with $M e_0 \subseteq V$. Suppose for a contradiction that $\ker \phi \nsubseteq V$. Note that the inclusion $\alpha \colon V \to M$ is such that $M e_0 \subseteq \Img \alpha$. Moreover, since $V$ is maximal, $\ker \phi + V = M$. This implies that the composition $\phi \cdot \alpha$ is an epimorphism. It thus follows that $\alpha$ is an epimorphism. That is, $V = M$. This is a contradiction, and so $\ker \phi \subseteq V$. Since this holds for any $V \in I(M)$, then $\ker \phi \subseteq  \bigcap_{V \in I(M)} V = \radd_0 M$ and so \eqref{pe c1} holds.
\end{proof}

\begin{lemma} \label{composition of part essential}
     Under \cref{X setup}, let $f \colon X \to Y$ and $g \colon Y \to Z$ be epimorphisms which satisfy the equivalent assumptions of \ref{partially essential}. Then, $g \cdot f$ satisfies \ref{partially essential}.
\end{lemma}
\begin{proof}
    We test the third condition in \ref{partially essential}. Let $\alpha \colon A \to X$ be a morphism with $X e_0 \subseteq \Img \alpha$ such that $g \cdot f \cdot \alpha$ is an epimorphism. Since $f$ is an epimorphism, it is clear that
    \[ Y e_0 = f(X) e_0 = f(X e_0) \subseteq \Img (f \cdot \alpha). \]
    Now since $g$ satisfies \ref{partially essential}, it follows that $f \cdot \alpha$ is an epimorphism. Whence $\alpha$ is an epimorphism since $f$ satisfies \ref{partially essential}.
\end{proof}

The strategy now is to show that each $\Lambda$-module admits an epimorphism as in \ref{partially essential}, and then use this fact to show that any $\Lambda$-module admits a partially minimal projective resolution. 

\begin{lemma} \label{ss admits partial cover}
    Under \cref{X setup}, let $M$ be a semi-simple $\Lambda$-module. Then, there exists a $\Lambda$-module epimorphism $\phi \colon Q \to M$, where $Q$ is a projective $\Lambda$-module and $\ker \phi \subseteq \radd_0 Q$.
\end{lemma}
\begin{proof}
    To first see that each simple $\Lambda_\con$-module admits such an epimorphism, consider the quotient morphisms $\pi_i \colon e_i \Lambda \to e_i \Lambda_\con$ and $\alpha_i \colon e_i \Lambda_\con \to e_i \Lambda_\con / \radd( e_i \Lambda_\con) \cong S_i$ for $i \in I$. Note that 
    \begin{align*}
        & \ker \alpha_i = \radd( e_i \Lambda_\con) \subseteq \radd_0(e_i \Lambda_\con), \\
        & \ker \pi_i = e_i [\proj \cE] = e_i \Lambda e_0 \Lambda.
    \end{align*}
    Since $\radd_0(e_i \Lambda)$ is the intersection of modules $Y$ with $(e_i \Lambda) e_0 \Lambda \subseteq Y \subseteq e_i \Lambda$, then necessarily $\ker \pi_i =  e_i \Lambda e_0 \Lambda \subseteq \radd_0(e_i \Lambda)$.
 
    It thus follows from \ref{composition of part essential} that the composition 
    \[\phi_i \colon e_i \Lambda \tow{\pi_i} e_i \Lambda_\con \tow{\alpha_i} S_i \]
    is such that $\ker \phi_i \subseteq \radd_0(e_i \Lambda)$. 

    Next, let $S$ be a simple $\Lambda$-module which is not a $\Lambda_\con$-module. Then, $S$ is not annihilated by $e_0$, so that  
    \[ \Hom_\Lambda(e_0 \Lambda, S) \cong \Hom_\Lambda(\Lambda, S) e_0 \cong S e_0 \cong S \neq 0. \]
    Hence, we may choose a nonzero $\phi \colon e_0 \Lambda \to S$ which, since $S$ is simple, is surjective. By \ref{dragons}, $\ker \phi \subseteq e_0 \Lambda = \radd_0 (e_0 \Lambda).$ 

    In summary, given a simple $\Lambda$-module $T$, write 
    \[ P(T) \colonequals \begin{cases}
        e_i \Lambda & \text{ if $T \cong S_i$ for $i \in I$,} \\
        e_0 \Lambda & \text{ otherwise.}
    \end{cases}\]
    
   Then, given an arbitrary semi-simple module $M = \bigoplus_{j =1}^k T_j$, by the above there exists an epimorphism $\phi = \bigoplus_{j =1}^k \phi_j$ with 
    \[ \ker \phi \subseteq \bigoplus_i^k \radd_0 P(T_j)= \radd_0 \left( \bigoplus_i^k P(T_j) \right). \qedhere \] 
\end{proof}

\begin{prop} \label{module admits partial cover}
    Under \cref{X setup}, let $M \in \modu \Lambda$. There exists projective $\Lambda$-module $Q$ and an epimorphism $\psi \colon Q \to M$ such that $\ker \psi \subseteq \radd_0(Q)$.
\end{prop}
\begin{proof}
    Since $\Lambda$ is module finite over a local ring, it is semilocal by \cite[20.6]{Lam}. In particular, for any $M \in \modu \Lambda$, it is well-known that $M / \radd M$ is semi-simple. To see this, note that $M/ \radd M = M / M \radd \Lambda$ is a semi-simple $\Lambda / \radd \Lambda$-module. Since simple $\Lambda$-modules are precisely the simple $\Lambda/\radd \Lambda$ modules, it follows that $M/\radd M$ is a semi-simple $\Lambda$-module. 
    
    By \ref{ss admits partial cover}, there is an epimorphism $\phi \colon W \to M/\radd M$ with $W$ projective and $\ker \phi \subseteq \radd_0 W$. Since $W$ is projective, $\phi$ must factor as  $\psi \colon W \to M$ followed by $q \colon M \to M / \radd M$. Since $\ker q \subseteq \radd M$ and $\phi$ is an epimorphism, so is $\psi$, by the classical version of \ref{partially essential} (see, for example, \cite[3.3]{HK14}). Finally, since $\phi = q \cdot \psi$, we have $\ker \psi \subseteq \ker \phi \subseteq \radd_0 W$, as required.
\end{proof}

\begin{theorem} \label{existence_min_res}
    Under \cref{X setup}, let $M \in \modu \Lambda$. Then, $M$ admits a partially minimal projective resolution.
\end{theorem}
\begin{proof}
    From \ref{module admits partial cover}, there is an epimorphism $f_0 \colon Q_0 \to M$ with $\ker f_0 \subseteq \radd_0 Q_0$ and $Q_0$ projective. We will define a projective resolution $f \colon Q \to X$ inductively. Suppose that $f_n \colon Q_n \to Q_{n-1}$ is such that $\ker f_n \subseteq \radd_0 Q_n$. Then, by \ref{module admits partial cover}, we may choose an epimorphism $\pi_{n+1} \colon Q_{n+1} \to \ker f_n$ with $Q_{n+1}$ projective and $\ker \pi_{n+1} \subseteq \radd_0 Q_{n+1}$. Let $i_n \colon \ker f_n \to Q_n$ denote the natural inclusion and set $f_{n+1} = i_n \cdot \pi_{n+1}$.
    
    We claim that this projective resolution is partially minimal. To see this, let $S$ be a simple $\Lambda_\con$-module and suppose that $g \colon Q_n \to S$ is nonzero. Then, $\ker g$ is maximal and, moreover, contains $Q_n e_0$ since $S$ is annihilated by $e_0$. Therefore, $\radd_0 Q_n \subseteq \ker g$ and so
    \[ \Img f_{n+1} = \ker f_n \subseteq \radd_0 Q_n \subseteq \ker g. \]
    It follows that $g \cdot f_{n+1} = 0$, and thus $\Hom_\Lambda(f_{n+1}, \cS) = 0$.
\end{proof}

\subsection{Relatively spherical properties} \label{relatively sph section}

The goal of this section is to prove that if $\Lambda_\con$ has certain $\Ext$-vanishing properties, then it admits a projective resolution as in \eqref{pres ei acon t} below. The existence of this partially minimal resolution is an essential step to understand the main theorem in \cref{periodic suspension section}.

\begin{definition} \label{def rel sph}
    Under \cref{X setup}, fix $t \in \Z$ with $2 \leq t \leq d$. We will say that $\Lambda_\con$ is \textit{t-relatively spherical} if $\Ext^k_\Lambda(\Lambda_\con, \cS) = 0$ for $k \neq 0, t$.
\end{definition}

The definition we use of $t$-relatively spherical is based on \cite[4.3]{DW4}. In the case where $\Lambda_\con$ is basic, the definitions are the same.

We require the following two technical lemmas. 

\begin{lemma}\label{projective cover of cS}
     Under \cref{X setup}, let $S_i$ be the simple $\Lambda_\con$-module as in \ref{simples not}. Then, the following hold.
     \begin{enumerate}
         \item $\Hom_\Lambda(\cE(X, X_i), S_j) \cong \delta_{ij} S_j$ as vector spaces for all $i \in \{0, 1, \hdots, n\}$ and $j \in I$. 
         \item If $U \in \add e_0 \Lambda$, then $\Hom_\Lambda(U, S_j) = 0$ for all $j \in I$.
         \item If $U \in \add \Lambda$ and $\Hom_\Lambda(U, S_j) = 0$ for all $j \in I$, then $U \in \add e_0 \Lambda$.
     \end{enumerate}
\end{lemma} 
\begin{proof}    
    (1) First, note that $\Hom_\Lambda(e_i \Lambda, S_i) \neq 0$ since the composition of quotient maps 
    \[ e_i \Lambda \tow{\pi_i} e_i \Lambda_\con \tow{\alpha_i} S_i \]
    is nonzero.Thus, there is a vector space isomorphism
    $\Hom_\Lambda(e_i \Lambda, S_i) \cong S_i$
    sending $\alpha \mapsto \alpha(e_i)$. Similarly, $\Hom_\Lambda(\Lambda, S_i) \cong S_i$ so that, by counting dimensions, it must be that $\Hom_\Lambda(e_i \Lambda, S_j) = 0$ for $i \neq j$.

    (2) The statement is clear from the fact that $\Hom_\Lambda(e_0 \Lambda, S_j) = 0$ for all $j \in I$. 

    (3) Suppose that $U \in \add \Lambda$. Then, by additive equivalence \ref{useful lemma} \eqref{yoneda}, $U = \cE(X, Y)$ for some $Y \in \add_\cE X$. It follows that $Y$ is stably isomorphic to $\bigoplus_{i = 1}^n X_i^{b_i}$ for $b_i \in \N$. Consequently, by \ref{lift of stable iso}, there is an isomorphism $Q \oplus Y \cong Q' \oplus \bigoplus_{i = 1}^n X_i^{b_i}$ in $\cE$ with $Q, Q' \in \add P$. Now, 
    \[ \Hom_\Lambda( \cE(X, Q \oplus Y), S_j) = \Hom_\Lambda( \cE(X, Q), S_j) \oplus \Hom_\Lambda(\cE(X, Y), S_j) \]
    Since $\cE(X, Q) \in \add e_0 \Lambda$, (2) implies that $\Hom_\Lambda( \cE(X, Q), S_j) = 0$. Moreover, since $U = \cE(X, Y)$, it follows by assumption that $\Hom_\Lambda(\cE(X, Y), S_j) = 0$. Thus, 
    \[0 = \Hom_\Lambda( \cE(X, Q \oplus Y), S_j) = \Hom_\Lambda( \cE(X,  Q' \oplus \bigoplus_{i = 1}^n X_i^{b_i}), S_j). \]
    In particular, this implies that
    \[ \Hom_\Lambda( \cE(X, X_j^{b_j}), S_j) = 0 \]
    for all $j \in \{1, \hdots, n\}$. By (1), it must be that $b_j = 0$ for all $j$. Thus, $Y$ vanishes in $\cD$. That is, $Y \in \add P$ and $U = \cE(X, Y) \in \add \cE(X, P)  = \add e_0 \Lambda$.
\end{proof}

\begin{lemma} \label{res lemma}
    Under \cref{X setup}, fix $i \in I$. Then, there exists a partially minimal projective resolution of $e_i \Lambda_\con$ 
    \begin{align} \label{p res U}
         \hdots \tow{q_2} \cE(X, W_1) \tow{q_1} \cE(X, X_i) \to e_i \Lambda_\con \to 0
    \end{align}
    such that 
    \begin{enumerate}
        \item $q_1 = \cE(X, \alpha_1)$ for some morphism $\alpha_1 \colon W_1 \to X_i  \in \cE$. 
        \item $\alpha_1 \colon W_1 \to X_i$ is an admissible epi and, hence,  admits a kernel $ \ker \alpha_1 \in \cE$.
        \item $K_1 \colonequals  \ker q_1 \cong \cE(X, \ker \alpha_1)$
    \end{enumerate}
    Further, for $j >1$, the following hold.
    \begin{enumerate}[resume]
        \item $q_j = \cE(X, \alpha_j)$ for some morphism $\alpha_j \colon W_j \to W_{j-1}  \in \cE$.
        \item $\alpha_j$ is an admissible epimorphism onto $\ker \alpha_{j-1}$, and so it admits a kernel $\ker \alpha_j$.
        \item $K_j \colonequals \ker q_j \cong \cE(X, \ker \alpha_j)$.
    \end{enumerate}
\end{lemma}
\begin{proof}
   We will build the projective resolution inductively. 

    \begin{enumerate}[label=(\alph*)]
        \item Consider first the case where $j = 1$. By the proof of \ref{existence_min_res} and \ref{ss admits partial cover}, there is an exact sequence
    \begin{align}
        & 0  \to K_1 \tow{j_1} \cE(X, U_1) \tow{f_1} \cE(X, X_i) \tow{q_0} e_i \Lambda_\con \to 0 \label{splitting 0}
    \end{align}
    where $\Hom_\Lambda(f_1, \cS) = 0$. 
    
    Additive equivalence applied to $X \in \cE$ implies that $f_1 = \cE(X, \beta_1)$ for some morphism $\beta_1 \colon U_1 \to X_i$. By \ref{dugas lemma}, there is an exact sequence
    \[ 0 \to \ker (p_1, \beta_1) \to Q_1 \oplus U_1 \tow{(p_1, \beta_1)} X_i \to 0 \]
     in $\cE$ where $Q_1 \in \add_\cE P$. Applying $\cE(X, -)$ to this sequence induces the exact sequence
     \[
     \begin{tikzcd} 
     0 \arrow[r] & \cE(X, \ker (p_1, \beta_1)) \arrow[r, "{\cE(X, (p_1, \beta_1) )}"]  &[2em] \cE(X, X_i) \arrow[r] & \Coker \cE(X, (p_1, \beta_1) ) \arrow[r] & 0. 
     \end{tikzcd}
     \]
     We claim that there is an isomorphism $e_i \Lambda_\con \cong \Coker \cE(X, (p_1, \beta_1) )$,which holds if and only if $\Img \cE(X, (p_1, \beta)) = e_i [\proj \cE]$. Let $g = \begin{pmatrix} g_1 \\ g_2 \end{pmatrix} \in \cE(X, Q_1 \oplus U_1)$.
     
    First, observe that by exactness of \eqref{splitting 0}, $\Img  \cE(X, \beta_1)  = \ker q_0$ and by the proof of \ref{ss admits partial cover}, $\ker q_0 = e_i [\proj \cE]$. Consequently, $\beta_1 g_1 \colon X \to X_i$ factors through $\proj \cE$. Since $Q_1 \in \proj \cE$, it is clear that $p_1 g_1$ also factors through $\proj \cE$. Hence, $(p_1, \beta_1) \cdot g = p_1 g_1 + \beta_1 g_1\in e_i [\proj \cE]$. That is, $\Img \cE(X, (p_1, \beta_1) ) \subseteq e_i [\proj \cE]$. 

    Next, suppose that $f \colon X \to X_i$ factors through $Q \in \proj \cE$, say via
    \[
    \begin{tikzcd}[row sep = small, column sep = small]
        X \arrow[rr, "f"] \arrow[dr, "a"] & {} & X_i \\
        {} & Q \arrow[ur, "b"] & {}
    \end{tikzcd}
    \]
    Then, since $(p_1, \beta_1) \colon Q_1 \oplus U_1 \to X_i$ is an epi, there is a map $c \colon Q \to  Q_1 \oplus U_1$ such that $b = (p_1, \beta_1) \cdot c$. Thus, $f = (p_1, \beta_1) \cdot c \cdot a$. Hence, $f \in \Img  \cE(X, (p_1, \beta_1))$, and so $\Img \cE(X, (p_1, \beta_1) ) = e_i [\proj \cE]$. Consequently, $e_i \Lambda_\con = \Coker \cE(X, (p_1, \beta_1) )$, as desired. 
    
    Let $\alpha_1 = (p_1, \beta_1)$, $q_1 = \cE(X, \alpha_1)$ and $W_1 = Q_1 \oplus U_1$. Then, there is an exact sequence 
    \begin{equation} \label{splitting 0 new}
        \begin{tikzcd}
            0 \arrow[r] &  \cE(X, \ker \alpha_1) \arrow[r, "i_1"] &  \cE(X, Q_1) \oplus \cE(X, U_1) \arrow[d, phantom,""{coordinate,name=Z}] \arrow[r, "{\cE(X, \alpha_1)}"] & \cE(X, X_i) \arrow[dl, rounded corners, to path={ -- ([xshift=2ex]\tikztostart.east) |- (Z) [near end]\tikztonodes -| ([xshift=-2ex]\tikztotarget.west) -- (\tikztotarget)}] \\
            {} & {} & e_i \Lambda_\con \arrow[r] & 0
        \end{tikzcd}
    \end{equation}
    where $\alpha_1 \colon U_1 \to X_i$ is an admissible epimorphism. This proves (1)-(3). Furthermore, to prove the partially minimal property, consider a map of $\Lambda$-modules $f \colon \cE(X, X_i) \to \cS$. Then, 
    \[ f \cdot \cE(X, \alpha_1) = f \cdot (\cE(X, p_1), \cE(X, \beta_1)) = (f \cdot \cE(X, p_1), f \cdot \cE(X, \beta_1)). \]
    Since $\Hom_\Lambda(f_1, \cS) = 0$ by choice of $f_1 = \cE(X, \beta_1)$, necessarily $f \cdot \cE(X, \beta_1) = 0$. In addition, $f \cdot \cE(X, p_1) \in \Hom_\Lambda(\cE(X, Q_1), \cS)$ where $\cE(X, Q_1) \in \add e_0 \Lambda$. Therefore, by \ref{projective cover of cS}(2), $f \cdot \cE(X, p_1) = 0$. Consequently, $\Hom_\Lambda((\cE(X, \alpha_1), \cS) = 0$.

    \item Consider the case $j =2$. Observe that the proof of \ref{existence_min_res} implies that there is an exact sequence
    \[ 0 \to K_2 \to \cE(X, U_2) \tow{f_2} \cE(X, \ker \alpha_1) \to 0. \]
    where $\Hom_\Lambda(i_1 \cdot f_2, \cS) = 0$. 
    
    Additive equivalence applied to $X \in \cE$ implies that $f_2 = \cE(X, \beta_2)$ for some morphism $\beta_2 \colon U_2 \to \ker \alpha_1$. By \ref{dugas lemma}, there is an exact sequence
    \[ 0 \to \ker (p_2, \beta_2) \tow{i_2} Q_2 \oplus U_2 \tow{(p_2, \beta_2)} \ker \alpha_1 \to 0. \]
    with $Q_2 \in \add_\cE P$. Applying $\cE(X, -)$ induces the exact sequence
    \[ 0 \to \cE(X, \ker (p_2, \beta_2)) \to \cE(X, Q_2) \oplus \cE(X, U_2) \tow{(\cE(X, p_2), f_2)} \cE(X, \ker \alpha_1) \]
    Moreover, the last map must be an epi because component $f_2$ is an epi. Hence, by letting $W_2 = Q_2 \oplus U_2$, $\alpha_2 = (p_2, \beta_2)$ and $q_2 = \cE(X, \alpha_2)$,  there is an exact sequence
    \begin{align*}
        0 \to \cE(X, \ker \alpha_2) \tow{i_2} \cE(X, Q_2) \oplus \cE(X, U_2) \tow{q_2} \cE(X, \ker \alpha_1) \to 0
    \end{align*}
    which satisfies (4)-(6). Additionally, given $h \colon \cE(X, W_1) \to \cS$, then 
    \[ h \cdot i_1 \cdot q_2 = (h \cdot i_1 \cdot \cE(X, p_2), h \cdot i_1 \cdot \cE(X, \beta_2)). \] 
    Since $Q_2 \in \add P$, the map $h \cdot i_1 \cdot \cE(X, p_2) \colon \cE(X, Q_2) \to \cS$ vanishes by \ref{projective cover of cS}(2). Moreover, by choice of $f_2 = \cE(X, \beta_2)$, $\Hom_\Lambda(i_1 \cdot f_2, \cS) = 0$, and so $h \cdot i_1 \cdot f_2 = 0$. We may conclude that $\Hom_\Lambda(i_1 \cdot q_2, \cS) =0$, proving the partially minimal property.
    
    \item The case for general $j \geq 2$ is similar and yields exact sequences
     \begin{align}
        0 \to \cE(X, \ker \alpha_j) \to \cE(X, Q_j) \oplus \cE(X, U_j) \tow{q_j} \cE(X, \ker \alpha_j) \to 0 \label{splitting j}
    \end{align}
    which satisfy (4)-(6) and, moreover, $\Hom_\Lambda(i_{j-1} \cdot q_j, \cS) =0$.

\medskip

By splicing \eqref{splitting 0 new} and \eqref{splitting j}, we obtain the required partially minimal projective resolution. \qedhere
    \end{enumerate}
\end{proof}

\begin{setup} \label{spherical setup}
    Under \cref{X setup}, assume further that $\Lambda_\con$ is $t$-relatively spherical and is perfect over $\Lambda$.
\end{setup}

\begin{prop} \label{min res gen}
    Under \cref{spherical setup}, fix  $i \in I$. Then, $e_i \Lambda_\con$ admits a partially minimal projective resolution
    \begin{equation} \label{pres ei acon t}
     \begin{tikzcd}[column sep=0.8em]
     0 \arrow[r] &[0.8em] \cE(X, Q_{t}) \oplus \cE(X, \Sigma^{-t+1} X_i) \arrow[r] &[0.8em] \cE(X, Q_{t-1})  \arrow[d, phantom, ""{coordinate, name=Z}] \arrow[r] & \hdots \arrow[r] & \cE(X, Q_{1}) \arrow[dll, rounded corners, to path={ -- ([xshift=2ex]\tikztostart.east) |- (Z) [near end]\tikztonodes -| ([xshift=-2ex]\tikztotarget.west) -- (\tikztotarget)}] \\ 
     {} & {} & \cE(X, X_i) \arrow[r] &  e_i \Lambda_\con \arrow[r] & 0 
     \end{tikzcd}
\end{equation}
where 
    \begin{enumerate}
    \item Each $Q_j \in \add P$,
    \item There exists a permutation $\tau$ of the set $I$ such that $X_{\tau(i)} \cong \Sigma^{-t+1} X_i $ in $\cD$.
    \end{enumerate}
\end{prop}
\begin{proof}
    If $X$ is projective, then $\Lambda_\con$ vanishes and so it is clear \eqref{pres ei acon t} can be constructed. So, suppose that $X$ is not projective. 
    
    By \ref{existence_min_res}, $\Lambda_\con$ admits a partially minimal projective resolution
     \begin{align} \label{p res W}
         \hdots \tow{q_2} \cE(X, W_1) \tow{q_1} \cE(X, X_i) \to e_i \Lambda_\con \to 0
    \end{align}
    and, since  under \cref{spherical setup} $\Lambda_\con$ is perfect as a $\Lambda$-module, we may assume that \eqref{p res W} has finite length $s$. We will break this proof into several steps, where each step further refines \eqref{p res W}.  

    \begin{enumerate}[label=(\alph*)]
        \item We will first show that $W_j \in \add P$ for $0 < j \leq s$ with $j \neq t$. 
        
        Recall that by the definition of partially minimal projective resolution, applying the functor $\Hom_\Lambda(-, \cS)$ to \eqref{p res W} yields 
        \begin{align*}
            \Hom_\Lambda( \cE(X, W_j), \cS) \cong \Ext^j_\Lambda(\Lambda_\con, \cS)
    \end{align*} 
    which is zero since $j \neq 0, t$ and $\Lambda_\con$ is relatively spherical. Thus, by \ref{projective cover of cS} (3), the module $\cE(X, W_j) \in \add e_0 \Lambda$. By additive equivalence \ref{useful lemma} \eqref{yoneda}, $W_j \in \add P$.  \label{item a}
    
    \item \label{item b} The next step is to show that we may assume $s=t$. This requires some technical results. Observe that by \ref{item a} and \ref{res lemma}, \eqref{p res W} can be split into sequences 
    \begin{align}
        & 0  \to \cE(X, \ker \alpha_1) \tow{i_1} \cE(X, W_1) \tow{\cE(X, \alpha_1)} \cE(X, X_i) \tow{q_0} e_i \Lambda_\con \to 0 \label{splitting 0 2}\\
        & 0 \to \cE(X, \ker \alpha_j) \tow{i_j} \cE(X, W_j) \tow{\cE(X, \alpha_j)} \cE(X, \ker \alpha_{j-1}) \to 0. \label{splitting n 2} \\
        & 0 \to \cE(X, W_s) \tow{i_j} \cE(X, W_{s-1} ) \tow{\cE(X, \alpha_{s-1})} \cE(X, \ker \alpha_{s-2}) \to 0. \label{splitting s 2}
    \end{align}
    for $1 < j < s$. Here, $\alpha_1 \colon W_1 \to X_i$ and $\alpha_j \colon W_k \to \ker \alpha_{j-1}$ are admissible epimorphisms in $\cE$ such that there are exact sequences 
    \begin{align}
        & 0 \to \ker \alpha_1 \to W_1 \to X_i \to 0 \label{ker alpha 1} \\
        & 0 \to \ker \alpha_{j} \to W_j \to \ker \alpha_{j-1} \to 0. \label{ker alpha j}
    \end{align}
    in $\cE$. 
    
    We will specify $\ker \alpha_j$ for $1 \leq j \leq s$.
    \begin{enumerate}[label*=(\roman*)]
        \item When $1 \leq j < t$, then $W_j \in \add P$ by \ref{item a}. Hence, \eqref{ker alpha 1} implies that there is a projective $P_1$ with $\ker \alpha_1 \cong P_1 \oplus \Sigma^{-1} X$. Thus, inductively, there is an isomorphism $\ker \alpha_j \cong P_j \oplus \Sigma^{-j} X$ with $P_j \in \add P$. \label{b i}
        \item For $j > t$, it follows from \ref{item a} that $W_j \in \add P$, so that the exact sequence \eqref{ker alpha j} implies that $\ker \alpha_j = P_j \oplus \Sigma^{-j+t} \ker \alpha_t$. \label{b ii}
    \end{enumerate}
    
    \item \label{item c} We now show that we may assume that $s = t$. If $s > t$, then, by \eqref{splitting s 2} and \ref{b ii},
    \[ \cE(X, W_s) \cong \cE(X, \ker \alpha_{s-1}) = \cE(X, P_{s-1} \oplus \Sigma^{-s+1+t} \ker \alpha_t). \]
    Since $\cE(X, W_s)$ is projective $\Lambda$-module, so is $\cE(X, P_{s-1} \oplus \Sigma^{-s+1+t} \ker \alpha_t)$. It follows that $W_s$, $P_{s-1} \oplus \Sigma^{-s+1+t} \ker \alpha_t \in \add X$. Hence, additive equivalence \ref{useful lemma} \eqref{yoneda} implies that $W_s \cong P_{s-1} \oplus \Sigma^{-s+1+t} \ker \alpha_t$ in $\cE$. 
    
    Since $W_s \in \add P$ by \ref{item a}, it must be that $\Sigma^{-s+1+t} \ker \alpha_t = 0$ in $\cD$. Because $\Sigma$ is an automorphism of $\cD$, this implies that $\ker \alpha_t = 0$ in $\cD$. Thus, the exact sequence \eqref{ker alpha j} for $j = t$ induces an isomorphism $W_t \cong \ker \alpha_{t-1}$ in $\cD$. 
    
    Lifting this to $\cE$ using \ref{lift of stable iso}, $Q \oplus W_t \cong Q' \oplus \ker \alpha_{t-1}$ in $\cE$ with  $Q, Q' \in \add P$. Since $Q \oplus W_t \in \add_\cE X$, it follows that $\ker \alpha_{t-1} \in \add_\cE X$. Therefore, $K_{t-1} = \cE(X, \ker \alpha_{t-1})$ is projective. Whence, we may truncate the projective resolution at $K_{t-1}$ and, thus, may assume $s = t$.

    \item  \label{item d} We claim that we may assume $W_t = P_t \oplus \Sigma^{-t+1} X_i$ in $\cE$ for $P_t \in \add P$. To see this, note that it is a direct a consequence of \ref{item c} that we may take $W_t = \ker \alpha_{t-1}$ in $\cE$. Moreover, \ref{b i} implies that $\ker \alpha_{t-1} \cong P_t \oplus \Sigma^{-t+1} X_i$ in $\cE$ for $P_t \in \add P$. Hence, the claim follows.

    \item \label{item e} Since $W_t \in \add_\cD X$ by construction, $\Sigma^{-t+1} X_i \in \add_\cD X$. Moreover, $\Sigma^{-t+1} X_i$ must be indecomposable in $\cD$ because $X_i$ is indecomposable and $\Sigma$ is an automorphism. Thus, $\Sigma^{-t+1} X_i \cong X_j$ for some $j \in I$.
    \end{enumerate}

    By combining \ref{item a}, \ref{item c}, \ref{item d} and \ref{item e}, we may write \eqref{p res U} as 
    \begin{align} \label{p res U i}
    0 \to \cE(X, Q_t) \oplus \cE(X, \Sigma^{-t+1} X_i) \to \hdots \to \cE(X, Q_{1}) \to \cE(X, X_i) \to e_i \Lambda_\con \to 0
    \end{align}
    where $Q_t \in \add P$. Thus, (1) holds.
    
    To prove (2), note that \ref{item e} implies that $\Sigma^{-t+1} X_i \cong X_j$ in $\cD$ for some $j \in I$. Since $\Sigma$ is an automorphism, the association $\tau \colon i \mapsto j$ is a permutation.
\end{proof}

If we place additional assumptions on $\Lambda_\con$, we obtain finer control over the resolution. The assumptions that we place on $\Lambda_\con$ will imply that it is self-injective and, as a result, the permutation that we get in \ref{min res gen} is the Nakayama permutation of $\Lambda_\con$ (see \cref{app: self inj} for a brief overview of self-injective algebras and the Nakayama permutation).

\begin{setup} \label{spherical tensor setup}
    Under \cref{spherical setup}, assume $t = d \colonequals \dim \cR$ with $d \geq 2$. Suppose further that $\dim_k \Lambda_\con < \infty$, and that there is a $\Lambda$-bimodule isomorphism $\Lambda_\con \Dtensor_{\Lambda} \omega_\Lambda \cong \Lambda_\con$ .
\end{setup}

\begin{cor} \label{min res cor}
    Under \cref{spherical tensor setup}, the following hold.
    
    \begin{enumerate}
        \item $\Lambda_\con$ is self-injective (for a brief overview of self-injective algebras, see \cref{app: self inj}),
        \item for $i \in I$, $e_i \Lambda_\con$ admits the following partially minimal projective resolution
  \[
     \begin{tikzcd}[column sep=0.8em]
     0 \arrow[r] &[0.8em] \cE(X, Q_{d}) \oplus \cE(X, X_{\sigma(i)}) \arrow[r] &[0.8em] \cE(X, Q_{d-1})  \arrow[d, phantom, ""{coordinate, name=Z}] \arrow[r] & \hdots \arrow[r] & \cE(X, Q_{1}) \arrow[dll, rounded corners, to path={ -- ([xshift=2ex]\tikztostart.east) |- (Z) [near end]\tikztonodes -| ([xshift=-2ex]\tikztotarget.west) -- (\tikztotarget)}] \\ 
     {} & {} & \cE(X, X_i) \arrow[r] &  e_i \Lambda_\con \arrow[r] & 0 
     \end{tikzcd}
\]
 where $Q_j \in \add P$ for all $1 \leq j \leq d$ and $\sigma$ is the Nakayama permutation of $\Lambda_\con$.
    \end{enumerate}
    
\end{cor}
\begin{proof}
    Since $\Lambda_\con$ is finite-dimensional, it is self-injective if and only if $\Ext^1_{\Lambda_\con}(\cS, \Lambda_\con) = 0$. This latter condition holds by the following chain of vector space isomorphisms
    \begin{align*}
         \Ext^1_{\Lambda_\con}(\cS, \Lambda_\con)  & \cong \Ext^1_{\Lambda}(\cS, \Lambda_\con) \tag{by \ref{F is faithful on ext1}} \\
         & \cong \Ext^1_{\Lambda}(\cS, \Lambda_\con \Dtensor_\Lambda \omega_\Lambda) \tag{by assumptions in \ref{spherical tensor setup}} \\
         & \cong D \Ext^{d-1}_\Lambda(\Lambda_\con, \cS) \tag{by Serre duality \eqref{serre duality}} \\
         & = 0 \tag{$\Lambda_\con$ is $d$-relatively spherical with $d \geq 2$.}
    \end{align*}
    Thus, (1) holds. Let $\sigma$ be the Nakayama permutatiom of $\Lambda_\con$.

    To prove (2), by \ref{min res gen} and since $t=d$, $e_i \Lambda_\con$ admits the partially minimal projective resolution
     \begin{equation} \label{pres d}
     \begin{tikzcd}[column sep=0.7em]
     0 \arrow[r] &[0.8em] \cE(X, Q_{d}) \oplus \cE(X, \Sigma^{-d+1} X_i) \arrow[r] &[0.8em] \cE(X, Q_{d-1})  \arrow[d, phantom, ""{coordinate, name=Z}] \arrow[r] & \hdots \arrow[r] & \cE(X, Q_{1}) \arrow[dll, rounded corners, to path={ -- ([xshift=2ex]\tikztostart.east) |- (Z) [near end]\tikztonodes -| ([xshift=-2ex]\tikztotarget.west) -- (\tikztotarget)}] \\ 
     {} & {} & \cE(X, X_i) \arrow[r] &  e_i \Lambda_\con \arrow[r] & 0 
     \end{tikzcd}       
     \end{equation}
    where $W_k \in \add P$ for all $1 \leq k \leq d$ and $\Sigma^{-d+1} X_i \cong X_j$ in $\cD$ for some $j \in I$. Hence, it suffices to prove that $j = \sigma(i)$.
    
     As in \ref{simples not}, let $S_j$ be the simple $\Lambda_\con$-module which is a quotient of $e_j \Lambda_\con$. By applying $\Hom_\Lambda(-, S_j)$ to the projective resolution \eqref{pres d}, we have that
    \begin{align*}
         \Hom_\Lambda( \cE(X, Q_d) \oplus \cE(X, \Sigma^{-d+1} X_i), S_j) & \cong \Ext^d_\Lambda(e_i \Lambda_\con, S_j) \tag{by partial minimality} \\
         & \cong D \Hom_{\Lambda}(S_j, e_i \Lambda_\con \Dtensor \omega_\Lambda) \tag{by Serre duality \eqref{serre duality}} \\
         & \cong D \Hom_{\Lambda}(S_j, e_i \Lambda_\con) \tag{by assumptions in \ref{spherical tensor setup}} \\
         & \cong \delta_{j \sigma(i)} S_j \tag{by \ref{Nakayama perm}}
    \end{align*}
    Consequently, 
    \begin{align*}
        &  \Hom_\Lambda( \cE(X, Q_d) \oplus \cE(X, \Sigma^{-d+1} X_i), S_{\sigma(i)}) =  \Hom_\Lambda( \cE(X, \Sigma^{-d+1} X_i), S_{\sigma(i)}) \neq 0
    \end{align*}
    Therefore, by \ref{projective cover of cS} (1), $\Sigma^{-d+1} X_i$ cannot be isomorphic to $X_k$ for $k \neq \sigma(i)$. Since $\Sigma^{-d+1} X_i  \cong X_j$ for some $j$, it must be that $\Sigma^{-d+1} X_i \cong X_{\sigma(i)}$ in $\cD$.
\end{proof}

\subsection{Periodic suspension} \label{periodic suspension section}

With the partially minimal projective resolution of $\Lambda_\con$ constructed, we are able to specify precisely when the action of the suspension functor is periodic up to additive closure.

\begin{definition}
    Let $Y \in \cD$ and $k \in \Z$. 
    \begin{enumerate}
        \item If $k >0$, we say that $Y$ is \textit{$k$-rigid} if  $\Ext^{i}_\cD(Y, Y) = 0$ for all $1 \leq i \leq k$.
        \item If $k < 0$, we say that $Y$ is \textit{$k$-rigid} if $\Ext^{i}_\cD(Y, Y) = 0$ for all $k \leq i \leq -1$. 
    \end{enumerate}
\end{definition}

\begin{theorem} \label{syz theorem}
    Under \cref{X setup}, fix $t \in \Z$ with $2 \leq t \leq d$. Then, the following are equivalent.
    \begin{enumerate}
        \item  $\Lambda_\con$ is perfect over $\Lambda$, and is $t$-relatively spherical in the sense of \ref{def rel sph}. \label{spherical}
        \item   $X$ is $(-t+2)$-rigid in $\cD$ and there exists a permutation $\tau$ of $I$ such that, for all $i \in I$, $\Sigma^{-t+1} X_i \cong X_{\tau(i)}$ in $\cD$. \label{finite order}
    \end{enumerate}
\end{theorem}
\begin{proof}
    We begin by proving \eqref{spherical} $\Rightarrow$ \eqref{finite order}. If $\Lambda_\con$ is $t$-relatively spherical and perfect over $\Lambda$, then it follows from \ref{min res gen}(2) that there exists a permutation $\tau$ of $I$ such that $X_{\tau(i)} \cong \Sigma^{-t+1} X_i$ for all $i \in I$.
    
    To see that $X$ is $(-t+2)$-rigid, fix $i \in I$ and note that via \eqref{ker alpha j} and \ref{b i} in the proof of \ref{min res gen}, there are exact sequences
    \begin{align}
       0 \to P_{j} \oplus \Sigma^{-j} X_i \to W_j \tow{\alpha_j} P_{j-1} \oplus \Sigma^{-j+1} X_i \to 0
    \end{align}
    for $1 < j < t$. Applying $\cE(X, -)$ to these induces following diagram where the first column is exact. 
        \[ \begin{tikzcd}
        0 \arrow[d] &[-3em] {} &[2em] {} &[-3em] 0 \arrow[d] \\
        \cE(X, P_{j} \oplus \Sigma^{-j} X_i) \arrow[d, "i_j"] & {} \arrow[r,equal] & {} & \cE(X, P_{j} \oplus \Sigma^{-j} X_i) \arrow[d, "i_j"] \\
        \cE(X, W_j) \arrow[d, "q_j"] & {} \arrow[r,equal] & {} & \cE(X, W_j) \arrow[d, "q_j"] \\
        \cE(X, P_{j-1} \oplus \Sigma^{-j+1} X_i) \arrow[d] & {} \arrow[r,equal] & {} &  \cE(X, P_{j-1} \oplus \Sigma^{-j+1} X_i) \arrow[d] \\
        \cE^1(X, P_{j} \oplus \Sigma^{-j} X_i) \arrow[d] & {} & {} & 0 \\
        \cE^1(X, W_j) & {} & {} &  {}
    \end{tikzcd}
    \]
    The second column is exact by \eqref{splitting n 2} for $1 < j < t$. Since $W_j \in \add P$ and is, therefore,  injective, $\Ext^1_\cE(X, W_j) = 0$. It thus follows that $\Ext^1_\cE(X, P_j \oplus \Sigma^{-j} X_i) = 0$ for $1 < j < t$. By \ref{frob-stable ext corresp}, if $\Ext^1$ in $\cE$ vanishes then so does $\Ext^1$ in $\cD$. Thus, 
    \[ 0 = \Ext^1_\cD(X, P_j \oplus \Sigma^{-j} X_i) = \Ext^1_\cD(X, \Sigma^{-j} X_i) \cong \Ext^{-j+1}_\cD(X, X_i) \]
    and so $\Ext^k_\cD(X, X_i) = 0$ for $-t+2 \leq j \leq -1$. Since this is true for all $i \in I$, we may conclude that $X$ is $(-t+2)$-rigid.

    For the implication \eqref{finite order} $\Rightarrow$ \eqref{spherical}, fix $i \in I$. Since $\cE$ has enough projectives, there are exact sequences
    \begin{align} \label{conflations for X_i}
        & 0 \to \Sigma^{-j-1} X_i \tow{\iota_j} Q_j \tow{p_j} \Sigma^{-j}  X_i \to 0
    \end{align}
     for $0 \leq j \leq t-2$ with $Q_j \in \add P$. Applying $\cE(X, -)$ to \eqref{conflations for X_i} yields exact sequences
     \begin{equation} \label{spliced p-res}
        \begin{tikzcd}
            0 \arrow[r] &  \cE(X, \Sigma^{-j-1} X_i) \arrow[r, "{\cE(X, \iota_j)}"] &   \cE(X, Q_j) \arrow[d, phantom,""{coordinate,name=Z}] \arrow[r, "{\cE(X, p_j)}"] & \cE(X, \Sigma^{-j} X_i) \arrow[dl, rounded corners, to path={ -- ([xshift=2ex]\tikztostart.east) |- (Z) [near end]\tikztonodes -| ([xshift=-2ex]\tikztotarget.west) -- (\tikztotarget)}] \\
            {} & {} &  \cE^1(X, \Sigma^{-j-1} X_i)  \arrow[r] & 0
        \end{tikzcd}
    \end{equation}
    because $Q_j$ is projective/injective and so $\Ext^1_\cE(X, Q_j) = 0$. By assumption, $X$ is $(-t+2)$-rigid so that 
    \[ 0 = \Ext^k_\cD(X, X_i) = \Ext^1_\cD(X, \Sigma^{k-1} X_i) \]
    for $-t+2 \leq k \leq -1$. Whence, again by \ref{frob-stable ext corresp}, $\cE(X, p_k)$ is an epi for all such $k$. Therefore, \eqref{spliced p-res} induces exact sequences
      \begin{align} 
        & 0 \to \cE(X, \Sigma^{-j-1} X_i) \tow{\cE(X, \iota_j)} \cE(X, Q_j) \tow{\cE(X, p_j)} \cE(X, \Sigma^{-j} X_i) \to 0 \label{spliced p-res j} \\
        & 0 \to \cE(X, \Sigma^{-1} X_i) \tow{\cE(X, \iota_0)} \cE(X, Q_0) \tow{\cE(X, p_0)} \cE(X,  X_i) \to C \to 0 \label{spliced p-res 0}
    \end{align}  
    for $1 \leq j \leq t-2$. Here, $C$ denotes the cokernel of $\cE(X, p_0)$. 
    
    We claim that $C \cong e_i \Lambda_\con$ as $\Lambda$-modules. It is clear that $\Img \cE(X, p_0) \subseteq e_i [\proj \cE]$. The inclusion $e_i [\proj \cE] \subseteq \Img \cE(X, p_0)$ follows from the fact that $p_0$ is a deflation so that any morphism from a projective to $X_i$ must factor through $p_0$. Hence, 
    \begin{align} \label{c is Acon}
        C \cong\cE(X, X_i)/ e_i [\proj \cE] \cong e_i \Lambda_\con.
    \end{align}  
    Thus, by splicing \eqref{spliced p-res j} for $q \leq j \leq t-2$ and \eqref{spliced p-res 0}, we obtain the exact sequence,
    \begin{equation} \label{pres ei Lambdacon}
        \begin{tikzcd}
            0 \arrow[r] &  \cE(X, \Sigma^{-t+1} X_i) \arrow[r] & \cE(X, Q_{t-2}) \arrow[d, phantom,""{coordinate,name=Z}] \arrow[r] & \hdots \arrow[r] & \cE(X, Q_0)  \arrow[dll, rounded corners, to path={ -- ([xshift=2ex]\tikztostart.east) |- (Z) [near end]\tikztonodes -| ([xshift=-2ex]\tikztotarget.west) -- (\tikztotarget)}] \\
            {} & {} &  \cE(X,  X_i) \arrow[r] & e_i \Lambda_\con  \arrow[r] & 0
        \end{tikzcd}
    \end{equation}
    where $Q_j \in \add P$ for $0 \leq j \leq t-2$. By assumption, $\Sigma^{-t+1} X_i \cong X_{\tau(i)}$ so that $\Sigma^{-t+1} X_i \in \add X$. It follows that \eqref{pres ei Lambdacon} is a projective resolution of $e_i \Lambda_\con$. Hence, $e_i \Lambda_\con$ is perfect over $\Lambda$. Since this is true for all $i \in I$, $\Lambda_\con$ is perfect over $\Lambda$.

    Due to \ref{projective cover of cS}(2), applying $\Hom_\Lambda(-, \cS)$ to \eqref{pres ei Lambdacon}, yields that $\Ext^k_\Lambda(e_i \Lambda_\con, \cS)$ vanishes for all $k \neq 0, t$. For $k = t$, 
   \begin{align*}
        \Ext^t_\Lambda(e_i \Lambda_\con, S_j) & \cong \Hom_\Lambda(\cE(X, X_{\tau(i)}), S_j) \\
        & \cong \Hom_\Lambda( e_{\tau(i)} \Lambda, S_j) \\
        & \cong \delta_{\tau(i)j} S_{j} \tag{by \ref{projective cover of cS}(1).} 
   \end{align*}
   Thus, we may conclude that (1) holds.
\end{proof}

\begin{cor} \label{syz cor}
    Assume \cref{spherical tensor setup} with $d \geq 3$. Then, the following are equivalent
    \begin{enumerate}
        \item  $\Lambda_\con$ is perfect over $\Lambda$ and $d$-relatively spherical,  \label{spherical dimension}
        \item   $X$ is $(-d+2)$-rigid in $\cD$ and there exists a permutation $\tau$ of $I$ such that, for all $i \in I$, $\Sigma^{-d+1} X_i \cong X_{\tau(i)}$ in $\cD$. \label{finite order dimension}
    \end{enumerate}
      If these conditions are satisfied, then $\Lambda_\con$ is self-injective and $\tau = \sigma$, the Nakayama permutation of $\Lambda_\con$.
\end{cor}
\begin{proof}
    The equivalence \eqref{spherical dimension} $\Leftrightarrow$ \eqref{finite order dimension} is the $t=d$ case of \ref{syz theorem}. The fact $\Lambda_\con$ is self-injective follows from \ref{min res cor}(1). The fact that $\tau = \sigma$ follows by combining the $\Ext$ calculations in the proof of \ref{min res cor}(2) and the $\Ext$ calculations at the end of the proof of \ref{syz theorem}. 
\end{proof}

\begin{cor} \label{perfect on both sides} 
    Under \cref{X setup}, suppose that either of the equivalent assumptions in \ref{syz theorem} hold. Then, $\Lambda_\con$ is perfect as both a right and a left $\Lambda$-module, and $X$ is $(t-2)$-rigid.
\end{cor}
\begin{proof}
    In order to construct a projective resolution of $\Lambda_\con$ as a left $\Lambda$-module, recall that for $j \in \Z$ we have exact sequences in $\cE$,
    \[ 0 \to \Sigma^{j-1} X_i \tow{\iota_j} W_j \tow{p_j} \Sigma^{j} X_i \to 0 \]
    with $W_j \in \proj \cE$. We may apply $\cE(-, X)$ to these exact sequences so that, since $X$ is $(-t+2)$-rigid in $\cD$, the sequences
    \begin{align}
        &  0 \to \cE(\Sigma^{j} X_i, X) \to \cE(W_j, X) \to \cE(\Sigma^{j-1} X_i, X) \to 0 \label{s1 p} \\
        & 0 \to \cE(\Sigma X_i, X) \tow{\cE(p_1, X)} \cE(W_1, X) \tow{\cE(\iota_1, X)} \cE(X_i, X) \to C \to 0.  \label{s2 p}
    \end{align}
    are exact for $2 \leq j \leq t-1$. Here, $C$ denotes the cokernel of $\cE(\iota_i, X)$. Similarly to the proof of \ref{syz theorem} (2) $\Rightarrow$ (1), $\Img \cE(i_0, X) =  [\proj \cE]e_i$ so that $C \cong \Lambda_\con e_i$. Hence, we may splice the sequences \eqref{s1 p} for $2 \leq j \leq t-1$ and \eqref{s2 p} together to obtain the exact sequence,
    \[ 0 \to \cE(\Sigma^{t-1} X_i, X) \to \cE(W_{t-1}, X) \to \hdots \to \cE(W_1, X) \to \cE(X_i, X) \to \Lambda_\con e_i \to 0. \]
    Since $\Sigma^{t-1} X_i \in \add X$ by assumption, we obtain a projective resolution for $\Lambda_\con e_i$ of length $t$. By taking the sum of $\Lambda_\con e_i$ over all $i \in I$, we obtain a finite projective resolution for $\Lambda_\con$ as a left $\Lambda$-module. 

    To see that $X$ is $(t-2)$-rigid, consider $j$ such that $1 \leq j \leq t-2$. Because $X$ is $(-t+2)$-rigid in $\cD$ and $\Sigma^{t-1} X_i \cong X_{\tau(i)}$, then 
    \[ \Hom_\cD(X, \Sigma^j X_i) = \Hom_\cD(X, \Sigma^j \Sigma^{-t+1} X_{\tau(i)}) = \Hom_\cD(X, \Sigma^{-t+1+j} X_{\tau(i)}) = 0 \] 
    since $-t + 2 \leq -t + 1 + j \leq -1$.
\end{proof}

\section{Spherical twists induced by Frobenius categories} \label{spherical by frob}

Assuming \cref{X setup}, this section specifies when the restriction of scalars functor $F$ induced by quotient morphism $\pi \colon \End_\cE(X) \to \uEnd_\cE(X)$ is spherical. 

\subsection{Tilting} \label{autoequivalence section}

We first extend \cite[section 5.3]{DW1} to the setting \ref{X setup} and show that the noncommutative twist functor
\[ \cT \colonequals \Rderived{\Hom}_\Lambda([\proj \cE], -) \colon \Db(\modu \Lambda) \to \Db(\modu \Lambda) \]
is an autoequivalence, by proving that it is functorially isomorphic to the composition of standard equivalences induced by tilting modules. 

Unless mentioned otherwise, we assume the following throughout.

\begin{setup} \label{tilting set up}
    Under \cref{X setup}, assume that either of the equivalent assumptions in \ref{syz theorem} hold.
\end{setup}

\subsubsection{Some tilting bimodules}

\begin{notation}
    For $k \in \Z$ consider
    \begin{enumerate}
        \item the algebra  $\Lambda_{k} \colonequals \End_\cE(P \oplus \Sigma^{k} X)$, 
        \item the $\Lambda_{{k-1}}$-$\Lambda_{{k}}$ bimodules $D_k \colonequals \cE(P \oplus \Sigma^{k} X, P \oplus \Sigma^{k-1} X)$.
        \item the $\Lambda_{{k}}$-$\Lambda_{{k-1}}$ bimodules $I_k \colonequals \cE(P \oplus \Sigma^{k-1} X, P \oplus \Sigma^{k} X)$.
    \end{enumerate}
\end{notation}

\begin{prop} \label{increasing tilting one}
    The $\Lambda$-$\Lambda_{-1}$-bimodule $I_{0} = \cE(P \oplus \Sigma^{-1} X, X)$ is tilting.
\end{prop}
\begin{proof}
    This is essentially \cite[5.1]{J}. It suffices to prove criteria (a)-(c) of \cite[1.8]{mi}. Namely, 
    \begin{enumerate}[label=(\alph*)]
        \item The bimodule $I_0$ is perfect as a right $\Lambda_{-1}$-module and as a left $\Lambda$-module. 
        \item The right multiplication map $\rho \colon \Lambda_{-1} \to \Rderived{\Hom_{\Lambda^\op}}(I_0, I_0)$ is a $\Lambda_{-1}$-bimodule isomorphism.
        \item The left multiplication map $\lambda \colon \Lambda \to \Rderived{\Hom_{\Lambda_{-1}}}(I_0, I_0)$ is a $\Lambda$-bimodule isomorphism.
    \end{enumerate}

    There are exact sequences
    \begin{align} \label{conflations tilting I proof}
        0 \to P \oplus \Sigma^{k-1} X \tow{\iota_k} Q_k \tow{p_k} P \oplus \Sigma^k X \to 0 
    \end{align}
    in $\cE$ to which we may apply $\cE(-, X)$ and $\cE(P \oplus \Sigma^{-1} X, -)$. Since $\Ext^1_\cD(X, X)$ vanishes by assumptions \ref{tilting set up}, then it follows from \ref{frob-stable ext corresp} and its dual that the morphisms $\cE(\iota_1, X)$ and $\cE(P \oplus \Sigma^{-1} X, p_0)$ are epi. Hence, the projective resolutions of $I_0$ as a $\Lambda^\op$-module and as $\Lambda_{-1}$-module are
    \begin{align}
        0 \to \cE(X, X) \to \cE(Q_0, X) \to & I_0 \to 0 \label{pres i0 op} \\
        0 \to \cE(P \oplus \Sigma^{-1} X, P \oplus \Sigma^{-1} X) \to \cE(P \oplus \Sigma^{-1} X, Q_0) \to & I_0 \to 0, \label{pres i0}
    \end{align}
    respectively. Thus, $I_0$ is biperfect.

    To prove that $\rho \colon \Lambda_{-1} \to \Rderived{\Hom_{\Lambda^\op}}(I_0, I_0)$ is a bimodule isomorphism, apply the functor $\Hom_\Lambda(-, I_0)$ to the exact sequence \eqref{pres i0 op}
    and $\cE(P \oplus \Sigma^{-1} X, -)$ to the exact sequence \eqref{conflations tilting I proof} for $k=0$. Since $\cE(P \oplus \Sigma^{-1} X, p_0)$ is an epimorphism, the following diagram with exact columns commutes.
\begin{center}
{\small
    \begin{tikzcd}[column sep = 7.2em]
        0 \arrow[d] &[-8em] {} & {} &[-8em] 0 \arrow[d] \\
        \cE(P \oplus \Sigma^{-1} X, P \oplus \Sigma^{-1} X) \arrow[d] &[-8em] {} \arrow[r, "{f_{0} = \cE(-, X)}"] & {}& \Hom_{\Lambda^\op}( I_0, I_0) \arrow[d] \\
        \cE(P \oplus \Sigma^{-1} X, Q_0) \arrow[d] &[-8em] {} \arrow[r, "{g_0 = \cE(-, X)}"] & {}& \Hom_{\Lambda^\op}(\cE(Q_0, X), I_0) \arrow[d] \\
        \cE(P \oplus \Sigma^{-1} X, X) \arrow[d] &[-8em] {} \arrow[r, "{f_{-1} = \cE(-, X)}"] & {}& \Hom_{\Lambda^\op}(\cE(X, X), I_0) \arrow[d] \\
        0 &[-8em] {} & {}&  \Ext^1_{\Lambda^\op}(  \cE(P \oplus \Sigma^{-1} X, X),I_0) \arrow[d] \\
        {} &[-8em] {} & {}& 0
    \end{tikzcd}
}
\end{center}
   By additive equivalence \ref{useful lemma} \eqref{yoneda}, $f_{-1}$ and $g_0$ are isomorphisms, so that $f_0$ is also an isomorphism. Hence, $\Ext^1_{\Lambda^\op}(I_0, I_0) = 0$ so that $\Rderived{\Hom}_{\Lambda^\op}(I_0, I_0)$ is concentrated at degree zero and $\rho = f_0$, which is an isomorphism. 
    
    For the proof that $\lambda \colon \Lambda \to \Rderived{\Hom_{\Lambda_{-1}}}(I_{0}, I_{0})$ is an isomorphism, we apply $\cE(-, X)$ to the exact sequence \eqref{conflations tilting I proof} for $k=0$ and $\Hom_{\Lambda_{-1}}(-, I_0)$ to the  the exact sequence \eqref{pres i0}. We thus have the commutative diagram with exact columns
    \begin{center}
{\small
    \begin{tikzcd}[column sep = 7.2em]
        0 \arrow[d] &[-8em] {} & {} &[-8em] 0 \arrow[d] \\
        \cE(X, X) \arrow[d] &[-8em] {} \arrow[r, "{h_{0} = \cE(P \oplus \Sigma^{-1} X, -)}"] & {}& \Hom_{\Lambda_{-1}}(I_0, I_0) \arrow[d] \\
        \cE(Q_0, X) \arrow[d] &[-8em] {} \arrow[r, "{g'_0 = \cE(P \oplus \Sigma^{-1} X, -)}"] & {}& \Hom_{\Lambda_{-1}}(\cE(P \oplus \Sigma^{-1} X, Q_0), I_0) \arrow[d] \\
        \cE(P \oplus \Sigma^{-1} X, X) \arrow[d] &[-8em] {} \arrow[r, "{h_{-1} = \cE(P \oplus \Sigma^{-1} X, -)}"] & {}& \Hom_{\Lambda_{-1}}(\cE(P \oplus \Sigma^{-1} X, P \oplus \Sigma^{-1} X), I_0) \arrow[d] \\
        0 &[-8em] {} & {}&  \Ext^1_{\Lambda_{-1}}( I_0, I_0) \arrow[d] \\
        {} &[-8em] {} & {}& 0
    \end{tikzcd}
}
\end{center}
   Since $h_{-1}$ and $g'_0$ are isomorphisms by \ref{useful lemma} \eqref{yoneda}, so is $h_0$. Hence, $ \Ext^1_{\Lambda^\op}(I_0, I_0) = 0$ so that $\Rderived{\Hom}_{\Lambda_{-1}}(I_0, I_0)$ is concentrated at degree zero and $\lambda = h_0$ is an isomorphism.  
\end{proof}

\begin{prop} \label{decreasing tilting one}
    The $\Lambda_{-1}$-$\Lambda$-bimodule $D_0 = \cE(X, P \oplus \Sigma^{-1} X)$ is tilting.
\end{prop}
\begin{proof}
    As in \ref{increasing tilting one}, it suffices to prove criteria (a)-(c) of \cite[1.8]{mi} for $D_0$. The exact sequences \eqref{spliced p-res j} and their dual show that $D_0$ is biperfect. For the statement that the corresponding left multiplication map $\lambda$ is an isomorphism, note that we may apply $\Hom_\Lambda(-, D_0)$ to the exact sequences \eqref{spliced p-res j} and $\cE(-, P \oplus \Sigma^{-1} X)$ to the exact sequences \eqref{conflations tilting I proof} for $-t+2 \leq k \leq -1$. By \ref{perfect on both sides}, $X$ is $(t-2)$-rigid in $\cD$ so 
   \[ \Ext^1_\cD(P \oplus \Sigma^{k} X, P \oplus \Sigma^{-1} X) =  \Ext^{-k-1}_\cD(P \oplus X, P \oplus X) = 0. \] 
 Hence, $\cE(\iota_k, P \oplus \Sigma^{-1} X)$ is an epimorphism by \ref{frob-stable ext corresp}. 
 
 Therefore, the following diagram with exact columns commutes for $-t+2 \leq k \leq -1$.
\begin{center}
{\small
    \begin{tikzcd}[column sep = 7.2em]
        0 \arrow[d] &[-8em] {} & {} &[-8em] 0 \arrow[d] \\
        \cE(P \oplus \Sigma^{k} X, P \oplus \Sigma^{-1} X) \arrow[d] &[-8em] {} \arrow[r, "{f_{k} = \cE(X, -)}"] & {}& \Hom_\Lambda( \cE(X, P \oplus \Sigma^{k} X), D_0) \arrow[d] \\
        \cE(Q_k, P \oplus \Sigma^{-1} X) \arrow[d] &[-8em] {} \arrow[r, "{g_k = \cE(X, -)}"] & {}& \Hom_\Lambda(\cE(X, Q_k), D_0) \arrow[d] \\
        \cE(P \oplus \Sigma^{k-1} X, P \oplus \Sigma^{-1} X) \arrow[d] &[-8em] {} \arrow[r, "{f_{k-1} = \cE(X, -)}"] & {}& \Hom_\Lambda(\cE(X, P \oplus \Sigma^{k-1} X), D_0) \arrow[d] \\
        0 &[-8em] {} & {}&  \Ext^1_\Lambda( \cE(X, P \oplus \Sigma^{k} X), D_0) \arrow[d] \\
        {} &[-8em] {} & {}& 0
    \end{tikzcd}
}
\end{center}
    For $k -1 = -t+1$, $P \oplus \Sigma^{-t+1} X \in \add_\cE X$. Since  $Q_k \in \add_\cE P$, it follows from \ref{useful lemma} \eqref{yoneda} that $f_{-t+1}$ and $g_{-t+2}$ are isomorphisms. Thus, $f_{-t+2}$ must be an isomorphism. Proceeding inductively, for each $-t + 2 \leq k \leq -1$, we find that $f_k$ is an isomorphism. In particular, $f_{-1}$ is an isomorphism so that 
    \[ \Ext^1_\Lambda( D_0, D_0) = \Ext^1_\Lambda( \cE(X, P \oplus \Sigma^{-1} X), D_0) = 0. \]
    It follows that $\Rderived{\Hom}_{\Lambda}(D_0, D_0)$ is concentrated at degree zero and $\lambda = f_{-1}$ is an isomorphism. 
    
    For the proof that $\rho$ is an isomorphism, we apply $\cE(X, -)$ to the exact sequences \eqref{conflations tilting I proof} and also $\Hom_{\Lambda_{-1}^\op}(-, D_0)$ to the  the exact sequences
    \[ 0 \to \cE(P \oplus \Sigma^k X, P \oplus \Sigma^{-1} X) \to \cE(Q_k, P \oplus \Sigma^{-1} X) \to \cE(P \oplus \Sigma^{k-1} X, P \oplus \Sigma^{-1} X) \to 0 \]
    for $1 \leq k \leq t-2$. Letting 
    {\small
    \[ h_k = \cE(-, P \oplus \Sigma^{-1} X) \colon \cE(X,  P \oplus \Sigma^{k} X) \to \Hom_{\Lambda_{-1}^\op}( \cE( P \oplus \Sigma^k X,  P \oplus \Sigma^{-1} X), \cE(X, P \oplus \Sigma^{-1} X)), \]
    }
    we inductively show that $h_k$ are isomorphisms for all $1 \leq k \leq t-2$. Therefore, $\Rderived{\Hom}_{\Lambda_{-1}^\op}(D_0, D_0)$ is concentrated in degree zero and $\rho = h_1$. 
\end{proof}

\begin{cor} \label{Ik Dk are tilting}
    For any $k \in \Z$, there are tilting bimodules $I_k = \cE(P \oplus \Sigma^{k-1} X, P \oplus \Sigma^k X)$ and $D_k = \cE(P \oplus \Sigma^k X, P \oplus \Sigma^{k-1} X)$.
\end{cor}
\begin{proof}
    By \ref{sigma X satisfies X setup}, if $X$ satisfies \cref{X setup}, so does $P \oplus \Sigma^k X$. Moreover, it is clear that if $X$ satisfies \ref{syz theorem}(2), so does $P \oplus \Sigma^k X$. Thus, $P \oplus \Sigma^k X$ satisfies \cref{tilting set up}, and so we may apply \ref{increasing tilting one} and \ref{decreasing tilting one} to $P \oplus \Sigma^k X$.
\end{proof}

Therefore, the functors 
\begin{align*}
    & \Psi_k \colonequals \Rderived{\Hom}_{\Lambda_{k-1}}(I_k, -) \colon \Db(\modu \Lambda_{{k-1}}) \to \Db(\modu \Lambda_{{k}}) \\
    & \Phi_k \colonequals \Rderived{\Hom}_{\Lambda_{k}}(D_k, -) \colon \Db(\modu \Lambda_{k}) \to \Db(\modu \Lambda_{k-1})
\end{align*}
 are equivalences and, thus, so is their composition
 \begin{align}
    \Rderived{\Hom}_{\Lambda}(I_0 \Dtensor_{\Lambda_{-1}} D_0, -) & \colon \Db(\modu \Lambda) \to \Db(\modu \Lambda). \label{psiphi0}
\end{align}
If $t > 2$, then we further consider the composition
\begin{align}
    \Rderived{\Hom_\Lambda}(I_{t-1} \Dtensor_{\Lambda_{t-2}} I_{t-2} \Dtensor_{\Lambda_{t-1}} \hdots \Dtensor_{\Lambda_{1}} I_{1}, -) & \colon \Db(\modu \Lambda) \to \Db(\modu \Lambda_{t-1}), \label{psi t-1}
\end{align}
which is also an equivalence. We will mildly abuse notation and write $\Psi_0 \cdot \Phi_0$ for \eqref{psiphi0} and $\Psi_1^{t-1}$ for the \eqref{psi t-1}. Since $\add_\cE \Sigma^{-t+1} X= \add_\cD X$, there is a Morita equivalence between $\Lambda_{t-1}$ and $\Lambda$. Hence, $\Psi_0 \cdot \Phi_0$ and $\Psi_1^{t-1}$ be viewed as autoequivalences of $\Lambda$. 

\medskip
For the remainder of this section we will prove two statements. First, we claim that  $\Psi_0 \cdot \Phi_0$ is naturally isomorphic to the functor
\[ \cT \colonequals \Rderived{\Hom}_\Lambda([\proj \cE], -) \colon \Db(\modu \Lambda) \to \Db(\modu \Lambda). \]
Moreover, for $t > 2$, we construct a Morita equivalence $\F \colon \Db(\modu \Lambda_{t-1}) \to \Db(\modu \Lambda)$ such that $\F \cdot \Psi_1^{t-1} \cong \cT$.

\subsubsection{Tracking the composition $\Psi_0 \cdot \Phi_0$}

To see that $\Psi_0 \cdot \Phi_0$ is naturally isomorphic to $\cT$, it suffices to prove the following. 

\begin{prop} \label{I times D is add}
    There is a $\Lambda$-bimodule isomorphism
    $I_0 \Dtensor_{\Lambda_{-1}} D_0 \tow{\sim} [\proj \cE].$
\end{prop}
\begin{proof}
    We first claim that the complex $I_0 \Dtensor_{\Lambda_{-1}} D_0$ is concentrated in degree zero, so that there is an isomorphism $I_0 \Dtensor_{\Lambda_{-1}} D_0 \cong I_0 \otimes_{\Lambda_{-1}} D_0$.
    
    To see this, recall that the projective resolution of $I_0$ as $\Lambda_{-1}$-module is \eqref{pres i0}. Hence, $I_0 \Dtensor_{\Lambda_{-1}} D_0$ is represented by the chain complex 
    \[ 0 \to \cE(P \oplus \Sigma^{-1} X, P \oplus \Sigma^{-1} X) \otimes_{\Lambda_{-1}} D_0 \tow{f} \cE(P \oplus \Sigma^{-1} X, Q_0) \otimes_{\Lambda_{-1}} D_0 \to 0 \]
    Thus, to show the claim that $I_0 \Dtensor_{\Lambda_{-1}} D_0$ has cohomology in degree zero, it suffices to check that $f$ is a monomorphism. 
    
    Now, by applying $- \otimes D_0$ to \eqref{pres i0}, we obtain first column of the following commutative diagram with exact columns.
 \begin{equation} \label{diagram ID}
         \begin{tikzcd}[column sep = 4em]
        {} &[-5em] {} & {} &[-5em] 0 \arrow[d] \\
        \cE(P \oplus \Sigma^{-1} X, P \oplus \Sigma^{-1} X) \otimes_{\Lambda_{-1}} D_0 \arrow[d, "f"] &[-5em] {} \arrow[r, "{m_{-1}}"] & {}& \cE(X, P \oplus \Sigma^{-1} X) \arrow[d] \\
        \cE(P \oplus \Sigma^{-1} X, Q_0) \otimes_{\Lambda_{-1}} D_0 \arrow[d] &[-5em] {} \arrow[r, "{n_{0}}"] & {}& \cE(X, Q_{0}) \arrow[d] \\
       I_0 \otimes_{\Lambda_{-1}} D_0 \arrow[d] &[-5em] {} \arrow[r, "{m_{0}}"] & {}& \cE(X, X) \arrow[d,"\pi"] \\
        0 &[-5em] {} & {}& \Lambda_\con \arrow[d] \\
        {} &[-5em] {} & {}& 0
    \end{tikzcd}
\end{equation}
    The second column is obtained by applying $\cE(X, -)$ to the exact sequence \eqref{conflations tilting I proof} for $k=0$. The horizontal maps are the natural composition maps and commutativity can be checked by diagram chasing. Since $m_{-1}$ and $n_0$ are isomorphisms by  \ref{useful lemma} \eqref{composition is iso}, it follows that $f$ is a monomorphism, as required. 

    It remains to show that there is a $\Lambda$-bimodule isomorphism $I_0 \otimes_{\Lambda_{-1}} D_0 \cong [\proj \cE].$ This follows from the snake lemma applied to the diagram \eqref{diagram ID}, since it implies that 
    \[ I_0 \otimes_{\Lambda_{-1}} D_0 \cong \ker \pi = [\proj \cE]. \]
    Moreover, since both $m_0$ and $\pi$ are bimodule morphisms, this isomorphism is indeed a $\Lambda$-bimodule isomorphism.
\end{proof}

\subsubsection{Tracking the composition $\Psi_1^{t-1}$}

In order to show that, up to a Morita equivalence, $\cT$ is naturally isomorphic to $\Psi_1^{t-1}$, we will establish some lemmas. 

\begin{lemma} \label{morita equivalence}
    There exists a Morita equivalence $\F \colon \Db(\modu \Lambda_{t-1}) \to \Db(\modu \Lambda)$
\end{lemma}
\begin{proof}
    Let $\Lambda_{t-1,\con} \colonequals \uEnd_\cE(\Sigma^{t-1} X)$, and write $[\proj \cE]_{t-1}$ for the kernel of the quotient morphism $\Lambda_{t-1} \to \Lambda_{t-1,\con}$.

    The existence of the Morita equivalence $\F$ essentially follows from the fact that, due to \ref{syz theorem}, $\add_\cD(\Sigma^{t-1} X) = \add_\cD(X)$. That is, $\add_\cE(P \oplus \Sigma^{t-1} X) = \add_\cE(X)$ and, thus, $\cE(P \oplus \Sigma^{t-1} X, X)$ is a projective generator in $\modu \Lambda_{t-1}$. Moreover, it follows from additive equivalence \ref{useful lemma} \eqref{yoneda} that
    \[ \Hom_{\Lambda_{t-1}}( \cE(P \oplus \Sigma^{t-1} X, X), \cE(P \oplus \Sigma^{t-1} X, X) ) \cong \cE(X, X) \]
    as algebras. Thus, $\F \colon \Hom_{\Lambda_{t-1}}(\cE(P \oplus \Sigma^{t-1} X, X), -) \colon \modu \Lambda_{t-1} \to \modu \Lambda$ is a Morita equivalence.
\end{proof}

\begin{lemma} \label{tensor t-2 times}
    Suppose that $2 < t \leq d$. Then, there is a bimodule isomorphism
    \begin{align} 
     I_{t-2} \Dtensor_{\Lambda_{t-1}} I_{t-1} \Dtensor_{\Lambda_{k-2}} \hdots \Dtensor_{\Lambda_{1}} I_{1} \cong \cE(X, P \oplus \Sigma^{t-2} X) 
 \end{align} 
\end{lemma}
\begin{proof}
    We begin by showing that for $1 \leq k \leq t-2$, there are bimodule isomorphisms
    \[ I_{k} \Dtensor_{\Lambda_{k-1}} I_{k-1} \Dtensor_{\Lambda_{k-2}} \hdots \Dtensor_{\Lambda_{1}} I_{1} \cong \cE(X, P \oplus \Sigma^{k} X) \]
    by induction on $k$. When $k = 1$, this is $I_1 = \cE(X, P \oplus \Sigma^{1} X)$, which holds trivially. Suppose that it holds for $k = s$. Then, 
    \[  I_{s+1} \Dtensor_{\Lambda_{s}} I_{s} \Dtensor_{\Lambda_{k-2}} \hdots \Dtensor_{\Lambda_{1}} I_{1} \cong I_{s+1} \Dtensor_{\Lambda_s} \cE(X, P \oplus \Sigma^{s} X). \]

    We first prove that $I_{s+1} \Dtensor_{\Lambda_s} \cE(X, P \oplus \Sigma^{s} X)$ is concentrated in degree zero. As a $\Lambda$-module, $I_{s+1} \Dtensor_{\Lambda_s} \cE(X, P \oplus \Sigma^{s} X)$ is represented by the complex
    \[
    \begin{tikzcd} 
    0 \arrow[r] & \cE(P \oplus \Sigma^{s} X, P \oplus \Sigma^{s} X) \otimes_{\Lambda_{s}} \cE(X, P \oplus \Sigma^s X) \arrow[d, phantom, ""{coordinate, name=Z}]  \arrow[d, "f", rounded corners, to path={ -- ([xshift=2ex]\tikztostart.east) |- (Z) [near end]\tikztonodes -| ([xshift=-2ex]\tikztotarget.west) -- (\tikztotarget)}, swap] & {} \\ 
    {} & \cE(P \oplus \Sigma^{s} X, Q_{s+1}) \otimes_{\Lambda_{s}} \cE(X, P \oplus \Sigma^s X) \arrow[r] & 0 
    \end{tikzcd}
    \]
    which is obtained by applying $\cE(P \oplus \Sigma^{s} X, -)$ to the exact sequences \eqref{conflations tilting I proof} for $k=s+1$ and then tensoring termwise with $\cE(X, P \oplus \Sigma^s X)$. As in the proof of \ref{I times D is add}, we need to show that $f$ is a monomorphism. 
    
    By applying $\cE(X, -)$ to the exact sequences \eqref{conflations tilting I proof} for $k = s+1$, we obtain the following commutative diagram
    \begin{equation} \label{diagram I t-1 proof}
         \begin{tikzcd}[column sep = 4em]
        {} &[-5em] {} & {} &[-5em] 0 \arrow[d] \\
        \cE(P \oplus \Sigma^{s} X, P \oplus \Sigma^{s} X) \otimes_{\Lambda_{s}} \cE(X, P \oplus \Sigma^s X) \arrow[d, "f"] &[-5em] {} \arrow[r, "{m_{s}}"] & {}& \cE(X, P \oplus \Sigma^{s} X) \arrow[d] \\
        \cE(P \oplus \Sigma^{s} X, Q_{s+1}) \otimes_{\Lambda_{s}} \cE(X, P \oplus \Sigma^s X) \arrow[d] &[-5em] {} \arrow[r, "{n_{s+1}}"] & {}& \cE(X, Q_{s+1}) \arrow[d, "{\cE(X, p_{s+1})}"] \\
      I_{s+1} \otimes_{\Lambda_{s}} \cE(X, P \oplus \Sigma^s X) \arrow[d] &[-5em] {} \arrow[r, "{m_{s+1}}"] & {}& \cE(X, P \oplus \Sigma^{s+1} X) \arrow[d, "q_s"] \\
        0 &[-5em] {} & {}& \Ext^1_\cE(X, P \oplus \Sigma^{s} X) \arrow[d] \\
        {} &[-5em] {} & {}& 0
    \end{tikzcd}
\end{equation}
where the second column is exact. Now, $m_s$ and $n_{s+1}$ are isomorphisms by \ref{useful lemma} \eqref{composition is iso}, hence $f$ is injective, proving the claim. Namely, that $I_{s+1} \Dtensor_{\Lambda_s} \cE(X, P \oplus \Sigma^{s} X)$ has cohomology only in degree zero. Truncating in the category of bimodules 
\[ I_{s+1} \Dtensor_{\Lambda_s} \cE(X, P \oplus \Sigma^{s} X) \cong I_{s+1} \otimes_{\Lambda_s} \cE(X, P \oplus \Sigma^{s} X). \]

We next show that 
\[  I_{s+1} \otimes_{\Lambda_s} \cE(X, P \oplus \Sigma^{s} X) \cong \cE(X, P \oplus \Sigma^{s+1} X). \]
Since $X$ is $(-t+2)$-rigid in $\cD$, $\Ext^1_\cD(X, P \oplus \Sigma^{s} X)$ vanishes for $0 \leq s \leq t-3$, which implies that $\cE(X, p_{s+1})$ is an epi by \ref{frob-stable ext corresp}. Therefore, $\Ext^1_\cE(X, \oplus \Sigma^{s} X)$ vanishes for $0 \leq s \leq t-3$. That is, $m_{s+1}$ is epi. Since $m_s$ and $n_{s+1}$ are isomorphisms, the snake lemma implies that $m_{s+1}$ is a mono, and, thus, an isomorphism. In fact,  $m_{s+1}$ is a bimodule morphism, and so
\[  I_{s+1} \Dtensor_{\Lambda_s} \cE(X, P \oplus \Sigma^{s} X) \cong \cE(X, P \oplus \Sigma^{s+1} X) \]
as bimodules whenever $1 \leq s+1 \leq t-2$.

We may thus conclude that 
 \[ I_{k} \Dtensor_{\Lambda_{k-1}} I_{k-1} \Dtensor_{\Lambda_{k-2}} \hdots \Dtensor_{\Lambda_{1}} I_{1} \cong \cE(X, P \oplus \Sigma^{k} X) \]
 for $1 \leq k \leq t-2$. In particular, 
 \begin{align*} 
     I_{t-2} \Dtensor_{\Lambda_{t-1}} I_{t-1} \Dtensor_{\Lambda_{k-2}} \hdots \Dtensor_{\Lambda_{1}} I_{1} \cong \cE(X, P \oplus \Sigma^{t-2} X) &\qedhere
 \end{align*} 
\end{proof}

\begin{prop} \label{I t-1 times is add}
    Suppose that $2 < t \leq d$. Then, 
    \[ \F \cdot \Psi_1^{t-1} \cong \cT \]
    where $\F$ is the Morita equivalence of \ref{morita equivalence}.
\end{prop}
\begin{proof}
    Since 
    \[\F \cdot \Psi^{t-1}_1 \cong \Rderived{\Hom_\Lambda}( \cE(P \oplus \Sigma^{t-1} X, X) \Dtensor_{\Lambda_{t-1}} I_{t-1} \Dtensor_{\Lambda_{t-2}} I_{t-2} \Dtensor_{\Lambda_{t-1}} \hdots \Dtensor_{\Lambda_{1}} I_{1}, -), \]
   it suffices to prove that
    \begin{equation} \label{bimod iso}
        [\proj \cE] \cong \cE(P \oplus \Sigma^{t-1} X, X) \Dtensor_{\Lambda_{t-1}} I_{t-1} \Dtensor_{\Lambda_{t-2}} I_{t-2} \Dtensor_{\Lambda_{t-1}} \hdots \Dtensor_{\Lambda_{1}} I_{1}
    \end{equation} 
    as $\Lambda$-bimodules.

 By \eqref{tensor t-2 times}, this reduces to proving that 
\begin{align} \label{almost projE}
    [\proj \cE] \cong \cE(P \oplus \Sigma^{t-1} X, X) \Dtensor_{\Lambda_{t-1}} I_{t-1} \Dtensor_{\Lambda_{t-2}} \cE(X, P \oplus \Sigma^{t-2} X)
\end{align}
 Consider the diagram \eqref{diagram I t-1 proof} for $s = t-2$. In this case, the group $\Ext^1_\cE(X, P \oplus \Sigma^s X)$ does not vanish, but the second column is still exact and the morphisms $m_s$ and $n_{s+1}$ are still isomorphisms. So, it follows that $I_{t-1} \Dtensor_{\Lambda_{t-2}}  \cE(X, P \oplus \Sigma^{t-2} X)$ is still concentrated in degree zero. That is, 
\[ I_{t-1} \Dtensor_{\Lambda_{t-2}}  \cE(X, P \oplus \Sigma^{t-2} X) \cong I_{t-1} \otimes_{\Lambda_{t-2}}  \cE(X, P \oplus \Sigma^{t-2} X).  \]

Moreover, since $\add_\cE(P \oplus \Sigma^{t-1} X) = \add_\cE(X)$, the bimodule $\cE(P \oplus \Sigma^{t-1} X, X)$ is projective on either side, and so the right hand side of \eqref{almost projE} is isomorphic to 
\[ M \colonequals \cE(P \oplus \Sigma^{t-1} X, X) \otimes_{\Lambda_{t-1}} I_{t-1} \otimes_{\Lambda_{t-2}}  \cE(X, P \oplus \Sigma^{t-2} X) \]
Furthermore, since  $\add_\cE(P \oplus \Sigma^{t-1} X) = \add_\cE(X)$, by \ref{useful lemma} \eqref{composition is iso}  the composition map induces an isomorphism
\[M \cong \cE(P \oplus \Sigma^{t-2} X, X) \otimes_{\Lambda_{t-2}} \cE(X, P \oplus \Sigma^{t-2} X).  \]
It thus remains to show that 
\[ [\proj \cE] \cong \cE(P \oplus \Sigma^{t-2} X, X) \otimes_{\Lambda_{t-2}} \cE(X, P \oplus \Sigma^{t-2} X) \] 
as $\Lambda$-bimodules. 

By applying $\cE( P \oplus \Sigma^{t-2} X, -)$ then $(-) \otimes_{\Lambda_{t-2}} \cE(X, P \oplus \Sigma^{t-2} X)$ to \eqref{conflations tilting I proof} for $k = 0$, and applying $\cE( P \oplus \Sigma^{t-2} X, -)$  to \eqref{conflations tilting I proof} for $k = 0$, we obtain the diagram
\[
        \begin{tikzcd}[column sep = 4em]
        {} &[-5em] {} & {} &[-5em] 0 \arrow[d] \\
        \cE(P \oplus \Sigma^{t-2} X, P \oplus \Sigma^{-1} X) \otimes_{\Lambda_{t-2}} \cE(X, P \oplus \Sigma^{t-2} X) \arrow[d] &[-5em] {} \arrow[r, "{h_{-1}}"] & {}& \cE(X, P \oplus \Sigma^{-1} X) \arrow[d, "{\cE(X, \iota_0)}"] \\
        \cE(P \oplus \Sigma^{t-2} X, Q_{0}) \otimes_{\Lambda_{t-2}} \cE(X, P \oplus \Sigma^{t-2} X) \arrow[d] &[-5em] {} \arrow[r, "{g_0}"] & {}& \cE(X, Q_{0}) \arrow[d, "{\cE(X, p_{0})}"] \\
       \cE(P \oplus \Sigma^{t-2} X, X) \otimes_{\Lambda_{t-2}} \cE(X, P \oplus \Sigma^{t-2} X) \arrow[d] &[-5em] {} \arrow[r, "{h_0}"] & {}& \cE(X, X) \arrow[d, "q_0"] \\
        0 &[-5em] {} & {}& \Ext^1_\cE(X, P \oplus \Sigma^{-1} X) \arrow[d] \\
        {} &[-5em] {} & {}& 0
    \end{tikzcd}
\]
of $\Lambda$-bimodules, where the columns are exact and the horizontal maps are the composition morphisms (Here, the first column is exact because $X$ is $(t-2)$-rigid by \ref{perfect on both sides}). 

Since $ \add_\cE(P \oplus \Sigma^{t-1} X) = \add_\cE(X)$, then $\add_\cE(P \oplus \Sigma^{t-2} X) = \add_\cE(P \oplus \Sigma^{-1} X)$. Hence, \ref{useful lemma} \eqref{composition is iso} implies that $h_{-1}$ and $g_0$ are isomorphisms. It follows from the snake lemma that $h_0$ is a mono. By commutativity of the diagram and exactness of its columns,
\[ \cE(P \oplus \Sigma^{t-2} X, X) \otimes_{\Lambda_{t-2}} \cE(X, P \oplus \Sigma^{t-2} X) \cong \Img h_0 \cong \Img \cE(X, p_0) \cong \ker q_0 \]
By \eqref{spliced p-res 0} and \eqref{c is Acon}, $q_0$ is the quotient map $q_0 \colon \Lambda \to \Lambda_\con$. Hence, 
\[ \cE(P \oplus \Sigma^{t-2} X, X) \otimes_{\Lambda_{t-2}} \cE(X, P \oplus \Sigma^{t-2} X) \cong [\proj \cE] \]
as $\Lambda$-bimodules, which concludes the proof.  
\end{proof}

\subsubsection{Establishing that $\cT$ is an equivalence}

\begin{theorem} \label{autoequivalence theorem}
    Under \cref{tilting set up}, the functor 
    \[ \cT = \Rderived{\Hom_\Lambda}([\proj \cE], -) \colon \Dcat(\Lambda) \to \Dcat(\Lambda) \]
    is an equivalence which preserves the bounded derived category of finitely generated modules. 
\end{theorem}
\begin{proof}
    The fact that $\cT$ is an equivalence follows by applying \ref{I times D is add} or \ref{I t-1 times is add}, and the fact that the bimodules $I_k$ and $D_k$ are tilting by \ref{Ik Dk are tilting}. 

    To see that $\cT$ restricts to an equivalence 
    \[ \cT = \Rderived{\Hom_\Lambda}([\proj \cE], -) \colon \Db(\modu \Lambda) \to \Db(\modu \Lambda), \]
    observe that $\Lambda_\con$ is biperfect as a $\Lambda$-module via \ref{perfect on both sides}. Therefore, the $\Lambda$-bimodule exact sequence
    \[ 0 \to [\proj \cE] \to \Lambda \to \Lambda_\con \to 0 \]
    implies that $[\proj \cE]$ is also biperfect. Thus, $\cT$ preserves bounded complexes. The fact that it preserves complexes of finitely generated modules is clear, since $\Lambda$ is module finite over a noetherian ring.
\end{proof}

\subsection{Spherical twist and cotwist} \label{tws and ctws}

    This section assumes \cref{tilting set up} throughout. We will describe the cotwist around the restriction of scalars functor $F$ associated to $\pi \colon \Lambda \to \Lambda_\con$. As a result, we will establish that $F$ is spherical. 
    
    Applying the results in \cref{twist res scalars} to the canonical quotient morphism $\pi \colon \Lambda \to \Lambda_\con$ induces the following diagram of adjoint pairs.
    \begin{equation}
    \begin{tikzcd}[column sep=10em]
        \Dcat(\Lambda_\con) \arrow[rr, "F  = - \otimes^{\Lderived{}} {}_{\Lambda_\con}  \Lambda_\con {}_\Lambda \cong \Rderived{\Hom_{\Lambda_\con}}{({}_\Lambda \Lambda_\con {}_{\Lambda_\con}, -)}" description, ""{name=F,above}] & {} & \arrow[ll, "F^{\mathrm{RA}} = \Rderived{\Hom_\Lambda}{( {}_{\Lambda_\con} \Lambda_\con {}_\Lambda, -)}"{name=RA}, bend left=12] \arrow[ll, "F^{\mathrm{LA}}  = - \otimes^{\Lderived{}} {}_\Lambda \Lambda_\con {}_{\Lambda_\con}"{name=LA}, swap, bend right=12] \Dcat(\Lambda)
        \arrow[phantom, from=LA, to=F, "\scriptstyle\boldsymbol{\perp}"description]
        \arrow[phantom, from=F, to=RA, "\scriptstyle\boldsymbol{\perp}"description]
    \end{tikzcd}
\end{equation}
    
    \begin{lemma} \label{F restricts}
    Under \cref{tilting set up}, the functors $F$, $F^{\LA}$ and $F^{\RA}$ preserve bounded complexes of finitely generated modules.
    \end{lemma}
    \begin{proof}
    Since $\Lambda_\con$ is biperfect by \ref{perfect on both sides}, this is a consequence of \cite[6.7]{DW1}. 
    \end{proof}
    
    To describe the cotwist, we require the following lemma.

    \begin{lemma} \label{K}
    Under \cref{tilting set up}, $K \colonequals F^{\LA} \cdot F (\Lambda_\con) = {}_{\Lambda_\con} \Lambda_\con {}_\Lambda \otimes^{\Lderived{}} {}_\Lambda \Lambda_\con {}_{\Lambda_\con}$ has cohomology concentrated in degrees $0$ and $-t$. 
\end{lemma}
\begin{proof}
    Let $B = \Lambda_\con {}_\Lambda \otimes^{\Lderived{}} {}_\Lambda \Lambda_\con {}_{\Lambda_\con}$, which is just $K$ with the left module structure forgotten. Then, $B$ can be calculated by applying $ - \otimes {}_\Lambda \Lambda_\con {}_{\Lambda_\con}$ term-wise to the projective resolution obtained by summing \eqref{pres ei acon t} over all $i \in I$ with the correct multiplicities $a_i$:
\begin{align*}
    0 \to \cE(X, Q_{t}) \oplus \cE(X, Y_i) \to \cE(X, Q_{t-1}) \to \hdots \to \cE(X, Q_{1}) \to \cE(X, X') \to \Lambda_\con \to 0.
\end{align*}
    Since each $Q_j \in \add P$ for $1 \leq j \leq t$, then $\cE(X, Q_j) \in \add e_0 \Lambda$. That is, the tensor product $\cE(X, Q_j) \otimes {}_\Lambda \Lambda_\con {}_{\Lambda_\con}$ is a summand of $(e_0 \Lambda)^{k_j} \otimes {}_\Lambda \Lambda_\con {}_{\Lambda_\con}$ for some $k_j \in \N$. However, $e_0$ annihilates $\Lambda_\con$ so that 
    $(e_0 \Lambda \otimes {}_\Lambda \Lambda_\con {}_{\Lambda_\con})^{k_j} = 0.$
    Hence, $B^{j} = 0$ for $i \neq 0, -t$. Since $B$ and $K$ are quasi-isomorphic as $\Lambda_\con$-modules, it follows that
    \[ \cohom{j}(K) = \cohom{j}(B) = 0 \]
    for $i \notin \{-t, 0\}$.
\end{proof}

\begin{theorem} \label{summary theorem}
     Under \cref{tilting set up}, the following hold. 
     \begin{enumerate}
         \item The functor \label{T is twist}
         \begin{align*}
             \cT \colonequals \Rderived{\Hom}_\Lambda([\proj \cE], -) \colon \Dcat(\Lambda) \to \Dcat(\Lambda) 
         \end{align*}
         is the twist around $F$.  Moreover, the restriction 
         \[ \cT \colonequals \Rderived{\Hom}_\Lambda([\proj \cE], -) \colon \Db(\modu \Lambda) \to \Db(\modu \Lambda) \]
        is the twist around $F$ restricted to $\Db(\modu \Lambda_\con)$.
         \item The functor \label{C is cotwist}
        \begin{align*}
            \cC & =\Rderived{\Hom}_{\Lambda_\con}(\Tor^\Lambda_{t}(\Lambda_\con, \Lambda_\con), -)[-t-1] \\
            & \cong \Hom_{\Lambda_\con}( \cD(X, \Sigma^{-t+1} X), -)[-t-1],
        \end{align*} 
        is the cotwist around $F$.  Hence, it is an equivalence. Moreover, the restriction 
         \[ \cC =\Rderived{\Hom}_{\Lambda_\con}(\Tor^\Lambda_{t}(\Lambda_\con, \Lambda_\con), -)[-t-1] \colon \Db(\modu \Lambda_\con) \to \Db(\modu \Lambda_\con). \]
         is the twist around $F$ restricted to $\Db(\modu \Lambda_\con)$.
        \item The functor $F$ is spherical. \label{F is spherical}
     \end{enumerate}
\end{theorem}
\begin{proof}
        (1) The fact that $\cT$ is the twist is immediate from \ref{twist t res}. The restriction statement follows from \ref{F restricts}.

        (2) That $\cC =\Rderived{\Hom}_{\Lambda_\con}({}_{\Lambda_\con} \Tor^\Lambda_{t}(\Lambda_\con, \Lambda_\con) {}_{\Lambda_\con}, -)[-t-1]$ is an immediate consequence of \ref{K} and \ref{ctw res cohom in 2 degrees}. Further, the isomorphism $\cC \cong \Rderived{\Hom}_{\Lambda_\con}({}_{\Lambda_\con} \cD(X, \Sigma^{-t+1} X) {}_{\Lambda_\con}, -)[-t-1]$ follows from the $\Lambda_\con$-bimodule isomorphism $\Tor^\Lambda_{t}(\Lambda_\con, \Lambda_\con) \cong \cD(X, \Sigma^{-t+1} X)$ due to \cite[3.1(3)]{DUG}.

       To see that $\cC$ is an equivalence, notice that $\add_\cD X = \add_\cD \Sigma^{-t+1} X$, and so $ \cD(X, \Sigma^{t-1} X)$ is projective on either side and, moreover, $\add \cD(X, \Sigma^{t-1} X) = \add \Lambda_\con$. Hence, $\cD(X, \Sigma^{t-1} X)$ is a projective generator in $\modu \Lambda_\con$. It immediately follows from Morita theory that $\cC$ is an equivalence.

       (3) The cotwist is an equivalence by (2), and the twist is an equivalence by  \ref{autoequivalence theorem}. Hence, by the the $2$ out of $4$ property \ref{spherical criteria}, $F$ is spherical. 
\end{proof}

\begin{remark}
    The theorem \cite[3.1(3)]{DUG} quoted above is stated for the case where the algebra $\Lambda_\con = \uEnd_\cE(X)$ is the endomorphism algebra of a cluster tilting object in $\cE$. However, the proof works word for word using only the rigidity assumption.
\end{remark}

\begin{cor} \label{biperfect tor}
    Under \cref{tilting set up}, the $\Lambda_\con$-bimodule ${}_{\Lambda_\con} \Tor_t^\Lambda(\Lambda_\con, \Lambda_\con) {}_{\Lambda_\con}$ is projective as a right and left bimodule.
\end{cor}
\begin{proof}
    By the proof of \ref{summary theorem} \eqref{C is cotwist},  
     \[ \Tor^\Lambda_{t}(\Lambda_\con, \Lambda_\con) \cong \cD(X, \Sigma^{-t+1} X) \]
     Our assumptions imply that $\add_\cD X = \add_\cD \Sigma^{-t+1} X$, so that the statement follows.
\end{proof}

\begin{example} \label{nc twist first ex}
    Let $R$ be a commutative noetherian complete local Gorenstein ring with at worst isolated hypersurface singularities and $\dim R = 3$. Moreover, consider $M \in \CM R$ a maximal Cohen-Macaulay module with $\Ext^1_R(M, M) =0$, so that $\End_R(M) \in \CM R$. Assume further that $R$ is a summand of $M$.
    
    Write $\Lambda = \End_R(M)$.  Let $[\add R]$ be the ideal of $\Lambda$ consisting of maps $M \to M$ which factor through $\add R$, and  set $\Lambda_\con = \Lambda/[\add R] = \uEnd_R(M)$, $\Lambda_\con = \Lambda/[\add R] = \uEnd_R(M)$. We will consider the natural surjection $\pi \colon \Lambda \to \Lambda_\con$, and argue that it satisfies the assumptions of \ref{tilting set up}.

    Since $R$ is Gorenstein, then $\CM R$ is a Frobenius exact category satisfying \ref{frob set up}. The suspension functor of $\uCM R$ is the cosyzygy functor, which we denote as $\Omega^{-1}$. From our choice of $M$, it is straightforward to check that $M$ satisfies \ref{X setup}. Thus, it suffices that \ref{syz theorem}(2) holds. 

    Since $R$ is a hypersurface, then there is an isomorphism $\Omega^2 \cong \id$ of endofunctors of $\uCM R$ \cite[6.1]{Ei}. Thus, it suffices to show that $M$ is $-1$-rigid.
    
    Observe that $\Omega^2 \cong \id$ implies the isomorphism of $\Lambda_\con$-bimodules
    \[ \Ext_{\uCM R}^{-1}(M, M) = \uHom_R( M, \Omega^{-1} M) \cong \uHom_R( M, \Omega^{-1} \Omega^{2} M) = \uHom_R( M, \Omega M).  \]
   It follows from AR-duality (see e.g.\ \cite[\S 1]{BIKR} for more details) that
   \[ \uHom_R( M, \Omega M) = D \Ext^1_R(M, M) \]
   where the right hand side vanishes by choice of $M$.
    Thus, $M$ is $-1$-rigid and so \ref{syz theorem}(2) holds. 
    
     Whence, we may apply \ref{summary theorem}, and we may conclude that \ref{summary theorem} extends \cite[5.10]{DW1}. 
\end{example}

\section{Spherical twists induced by crepant contractions onto affine schemes} \label{crepant contractions affine section}

This section applies the theory developed in section \ref{spherical by frob} to obtain derived autoequivalences of schemes with at worst Gorenstein singularities. 

\subsection{Setting}

    In this section, we introduce the setting in which we will work. For this, it will be important to recall a handful of key definitions and well-known results. 

    \begin{definition} 
    Let $(\cX, \cO_\cX)$ be a ringed space. Recall that a complex of sheaves $\cF \in \Dcat(\qcoh \cX)$ is \textit{perfect} if there exists an open covering  $X = \bigcup_i U_i$ such that $\cF |_{U_i}$ for each $i$ is quasi-isomorphic to a bounded complex of sheaves which are summands of finite free $\cO_X |_{U_i}$-modules.
    \end{definition}

    \begin{definition} \label{tilting bundle def}
       Let $\cX$ be a Noetherian scheme. Recall that a \textit{tilting complex} on $\cX$ is a complex $\cV \in \Dcat(\qcoh \cX)$ such that
        \begin{enumerate}
            \item $\cV$ is perfect, 
            \item \label{generates tb} $\cV$ generates $\Dcat(\qcoh \cX)$ as a triangulated category with infinite direct sums,
            \item $\Hom_{\Dcat(\qcoh \cX)}(\cV, \cV[n]) = 0$ for all $n \neq 0$.
        \end{enumerate}
        If, moreover, $\cV$ is a vector bundle, then we say that $\cV$ is a \textit{tilting bundle}.
    \end{definition}

    \begin{prop} \label{equivalence by tilting}
        Let $\cX$ be a Noetherian scheme admitting a tilting complex $\cV$. Moreover, write $\Lambda = \End_{\Dcat(\qcoh \cX)}(\cV)$. Then, the functor
        \[ \Rderived{\Hom_{\Dcat(\qcoh \cX)}}(\cV, -) \colon \Dcat(\qcoh \cX) \to \Dcat(\Mod \Lambda) \]
        is an equivalence. If furthermore $\cX$ is quasi-projective, then the equivalence restricts to 
        \[ \Psi_{\cV} \colonequals \Rderived{\Hom_{\Dcat(\coh \cX)}}(\cV, -) \colon \Db(\coh \cX) \to \Db(\modu \Lambda) \]
    \end{prop}
    \begin{proof}
        The fact the $\Rderived{\Hom_{\Dcat(\qcoh \cX)}}(\cV, -)$ is an equivalence is \cite[7.6]{VdBH}. The fact that the equivalence restrict when $\cX$ is quasi-projective is proved in \cite[6.4, 6.11, 6.12]{Rou}. It is also a special case of \ref{nc derived functors bg}\eqref{res qcoh} below. 
    \end{proof}
    
    \begin{definition}
    \begin{enumerate}
        \item As in \cite[2.1]{DW4}, we say that a \textit{contraction} is a projective birational morphism $f \colon X \to Y$ between normal varieties over a field $k$ with $Y$ quasi-projective and $\Rderived{f_*} \cO_X = \cO_Y$. 
        \item In view of the definition above of a contraction, we define a \textit{complete local contraction} to be a projective birational morphism $f \colon X \to Y = \spec R$ between normal schemes over a field $k$ with $R$ complete local and $\Rderived{f_*} \cO_X = \cO_Y$. 
        \item Suppose that $X$ and $Y$ are Gorenstein schemes with canonical sheaves $\omega_X$ and $\omega_Y$, respectively. Recall that a (complete local) contraction $f \colon X \to Y$ is \textit{crepant} if $f^* \omega_Y \cong \omega_X$.
        \end{enumerate}
    \end{definition}
    
    In this section we will work within the following setup

    \begin{setup} \label{z local tilting setup}
        Let $f \colon X \to Y := \spec R$ be a crepant (complete local) contraction and assume that $X$ admits a tilting bundle $\cV$ containing $\cO_X$ as a summand. 
    \end{setup} 
    
\begin{notation}
    With $f$ as in \ref{z local tilting setup}, write $Z$ for the locus of points of $Y$ onto which $f$ is not an isomorphism. 
\end{notation}

  \begin{lemma}[{\cite[2.5]{DW4}}] \label{gorenstein setting}
        If $R$ is Gorenstein, and $f \colon X \to \spec R$ satisfies \ref{z local tilting setup}, then the natural map $\End_X(\cV) \to \End_R(f_* \cV)$ is an isomorphism. Moreover, $f_* \cV \in \CM R$. Therefore, there is an equivalence $\Psi_\cV \colon \Dcat(\qcoh X) \tow{\sim} \Dcat(\Mod \Lambda)$ where $\Lambda = \End_R(f_* \cV)$ is the endomorphism algebra of an object $f_* \cV \in \CM R$ in a Frobenius exact category.
    \end{lemma}

The following lemma will be useful to construct spherical twists on $\Db(\coh X)$. 

\begin{lemma} \label{spherical composed with equiv}
    Let $\cC$ and $\cD$ be triangulated categories which admit Morita enhancements $\cA$ and $\cB$, respectively. Let $S \in \Dcat(\cA \text{-} \cB)$ be biperfect, and suppose that there is an equivalence $E \colon \cB \to \cB'$ of DG-categories. Observe that one may view $E$  as a $\cB$-$\cB'$-bimodule and $E^{-1}$ as a $\cB'$-$\cB$-bimodule. 
    
    If $T$ is the twist around $S$, then $E^{-1} \Dtensor_\cB T \Dtensor_\cB E$ is the twist around $ES$. Moreover, if $C$ is the cotwist around $S$, then it is also the cotwist around $ES$. Hence, if $S$ is spherical, so is $ES$.
\end{lemma}
\begin{proof}
    As in \cite[2.1,2.2]{AL} there is an adjunction
    \[ (-) \Dtensor_\cA \ S \colon \Dcat(\cA \text{-} \cA) \to \Dcat(\cA \text{-} \cB) \dashv (-) \Dtensor_\cB \ R \colon \Dcat(\cA \text{-} \cB) \to \Dcat(\cA \text{-} \cA) \]
    which induces the Hom-space isomorphisms
    \[ \Phi_{(A, B)} \colon \Hom_{\Dcat(\cA \text{-} \cA)}(A, B \Dtensor_\cB R) \to \Hom_{\Dcat(\cA \text{-} \cB)}(A \Dtensor_\cA  S , B). \]
    that are natural for all $A \in \Dcat(\cA \text{-} \cA)$ and $B \in \Dcat(\cA \text{-} \cB)$. Using these isomorphisms, it is easy to check that there are adjoint pairs
    \[ (-) \Dtensor_\cA ES \colon \Dcat(\cA \text{-} \cA) \to \Dcat(\cA \text{-} \cB') \dashv (-) \Dtensor_{\cB'} E^{-1} \Dtensor_\cB R \colon \Dcat(\cA \text{-} \cB') \to \Dcat(\cA \text{-} \cA). \]
    To see this, observe that for each  $A \in \Dcat(\cA \text{-} \cA)$ and $C \in \Dcat(\cA \text{-} \cB')$, there are Hom-set isomorphisms
    \[ \Psi_{A, C} \colon \Hom_{\Dcat(\cA \text{-} \cA)}(A, C \Dtensor_{\cB'} E^{-1} \Dtensor_{\cB} R) \tow{\Phi_{(A, C \Dtensor_{\cB'} E^{-1})} \Dtensor_\cB E} \Hom_{\Dcat(\cA \text{-} \cB')}(A \Dtensor_\cA S \Dtensor_\cB E, C) \] 
    which are natural in $A$ and $C$. Let $RE^{-1} \colonequals  E^{-1} \Dtensor_\cB R$.
    
    Hence, as in \ref{funct cones}, the twist $T$ and $T'$ around $S$ and $ES$  are defined (up to isomorphism) as 
    \begin{align*}
        & T \colonequals \cone(\tr \colon SR \to \cB) \in \Dcat(\cB\text{-}\cB), \\
        & T' \colonequals \cone(\tr' \colon E^{-1} \Dtensor_\cB SR \Dtensor_{\cB} E \to \cB') \in \Dcat(\cB'\text{-}\cB'),
    \end{align*}
    respectively. Here,
    \begin{align*}
        \tr \colonequals \epsilon_{\cB}, \text{ and } \tr' \colonequals \epsilon'_{\cB'}
    \end{align*}
    where $\epsilon$ and $\epsilon'$ are the counits of the adjunctions $(- \Dtensor_\cA \ S,- \Dtensor_\cA \ R )$ and $(- \Dtensor_\cA \ ES, - \Dtensor_\cA \ RE^{-1} )$, respectively. 

    Observe, then, that if $\tr' = E^{-1} \Dtensor_\cB \ \tr \ \Dtensor_\cB \ E$, then the triangle 
    \[ SR \tow{\tr}\cB \to T \to^+ \]
    in $\Dcat(\cB \text{-} \cB)$ induces the triangle
    \[ E^{-1} \Dtensor_\cB SR \Dtensor_{\cB} E \tow{\tr'} \cB' \to E^{-1} \Dtensor_\cB T \Dtensor_{\cB} E \to^+ \]
    in $\Dcat(\cB \text{-} \cB)$, from which we may conclude that $T' \cong E^{-1} \Dtensor_\cB T \Dtensor_{\cB} E$, as required. Thus, it suffices to show that $\epsilon'_{\cB'} =  E^{-1} \Dtensor_{\cB} \epsilon_\cB  \Dtensor_{\cB} E$.
    
    The counit at $\cB'$ $\epsilon'_{\cB'}$ is specified by 
    \begin{align*}
        \epsilon'_{\cB'}&  = \Psi_{(\cB' \ \Dtensor_{\cB'} \ E^{-1} \Dtensor_\cB \ R, \ \cB')}(1_{\cB' \ \Dtensor_{\cB'} \ E^{-1} \Dtensor_\cB \ R}) \\
        & = \Psi_{(E^{-1} \ \Dtensor_\cB \ R, \ \cB')}(1_{E^{-1} \ \Dtensor_{\cB} \ R}) \\
        & = (\Phi_{(E^{-1} \ \Dtensor_\cB \ R, \ E^{-1})}(1_{E^{-1} \ \Dtensor_{\cB} \ R}) \Dtensor_\cB  E \\
        & = \epsilon_{E^{-1}} \Dtensor_{\cB} E. 
        \end{align*}
    However, it follows from \cite[p.12]{AL} that $\epsilon_{E^{-1}} = E^{-1} \Dtensor_{\cB} \epsilon_\cB$. Hence, 
    \[\epsilon'_{\cB'} = E^{-1} \Dtensor_{\cB} \epsilon_\cB  \Dtensor_{\cB} E \]
    as required. 
    
    The fact that the cotwist is the same follows from $\act = \act'$. 
    \end{proof}
\subsection{Complete local setting}

Assume now that $f \colon X \to Y = \spec \cR$ is a complete local contraction as in \ref{crepant setup} with $(\cR, \mathfrak{m})$ a complete local ring. In this context, let $\Lambda = \End_\cR(f_* \cP)$ and write $\Lambda_\con = \uEnd_\cR(f_* \cP)$.

\begin{remark} \label{properties of fP}
    Since $Y$ is affine and $f$ is crepant, it follows from \cite[2.5(1)]{DW4} that $f_* \cP \in \CM \cR$.
\end{remark}

\begin{lemma} \label{alg conditions hold}
   The category $\cE = \CM \cR$ satisfies \cref{frob set up} and the module $f_* \cP$ satisfies \cref{X setup}. Moreover, $\Lambda$ is a Gorenstein $\cR$-order (i.e.\@ there exists an isomorphism $\Lambda_\con \otimes \omega_\Lambda \cong \Lambda_\con$ of $\Lambda$-bimodules).
\end{lemma}
\begin{proof}
    Since $\cR$ is a complete local Gorenstein ring, $\CM \cR$ is a Krull-Schmidt Frobenius category with $\proj \cE = \add \cR$. It follows that $\CM \cR$ satisfies \ref{frob set up}. 

    By the definition of a relative tilting bundle, $\cP = \cO_X \oplus \cP_0$. Since $\CM \cR$ is Krull-Schmidt, $f_* \cP$ isomorphic to $R \oplus Q \oplus \bigoplus_{i = 1}^n M_i^{a_i}$ where $Q \in \add R$, $a_i \in \N$ with $a_i > 0$, and $M_i$ are pairwise non-isomorphic indecomposable modules in $\CM \cR$ which are not projective. It follows that $M_i$ are also indecomposable in $\uCM \cR$. Hence, $f_* \cP$ satisfies \ref{X setup}.

    Finally, we will argue that $\omega_\Lambda \cong \Lambda$ as $\Lambda$-bimodules. Well, since $f$ is crepant then, by \cite[4.8]{IW2}, $\Lambda = \End_\cR(f_* \cP) \in \CM \cR$. Because of this and because Cohen-Macaulay modules are reflexive, \cite[2.22(2)]{IW1} and \cite[3.8 (1)$\Rightarrow$(3)]{IR} imply that $\omega_\Lambda \cong \Lambda$ as $\Lambda$-bimodules.
\end{proof}

\begin{prop} \label{complete local equiv}
    Suppose that $\Lambda_\con$ is perfect over $\Lambda$ and $t$-relatively spherical for some $t \in \Z$ with $2 \leq t \leq d$. Then, $[\add \cR]$ is a tilting $\Lambda$-bimodule. Whence, $X$ admits the derived autoequivalence
    \begin{align}
       \Phi_{\cR} \colonequals \Psi_{\cP}^{-1} \cdot \Rderived{\Hom}_\Lambda([\add \cR], -) \cdot \Psi_{\cP} \colon \Db(\coh X) \tow{\sim} \Db(\coh X) 
    \end{align}
\end{prop}
\begin{proof}    
    As $\Lambda_\con$ is perfect and relatively spherical, the pair $\cE = \CM \cR$ and $X = f_* \cP$ satisfies \cref{tilting set up}, by \ref{alg conditions hold}. Thus, \ref{autoequivalence theorem} implies that $[\add \cR]$ is a tilting bimodule and $\Phi_{\cR}$ is an equivalence.
\end{proof}

In order to describe the cotwist, the following lemma will be useful. 

\begin{lemma} \label{tor ext duality}
        Suppose that $\Lambda_\con$ is perfect over $\Lambda$ and $d$-relatively spherical. Then, there is a $\Lambda_\con$-bimodule isomorphism
        \[ {}_{\Lambda_{\con}} \Tor^\Lambda_t(\Lambda_{\con}, \Lambda_{\con}) {}_{\Lambda_{\con}}  \cong {}_{\Lambda_{\con}} \Ext_\Lambda^{d-t}(D \Lambda_{\con}, \Lambda_{\con}) {}_{\Lambda_{\con}} \]
\end{lemma}
\begin{proof}
    To avoid cluttered notation, let $B:= \Lambda_\con$. Consider the following isomorphisms,  
    \begin{align*}
        \Rderived{\Hom_\Lambda}({}_{B} B {}_{\Lambda} {}^\dagger , {}_{B} B {}_{\Lambda}) & \cong \Rderived{\Hom_\Lambda}({}_{B} B {}_{\Lambda} {}^\dagger, {}_{B} B {}_{\Lambda} \Dtensor  {}_{\Lambda} (\omega_\Lambda) {}_\Lambda  ) \tag{by \ref{alg conditions hold}}\\
        & \cong \Rderived{\Hom_\Lambda}({}_{B} B {}_{\Lambda}, {}_{B} B {}_{\Lambda}  {}^\dagger )^\dagger  \tag{by Serre duality \eqref{derived serre duality}}
    \end{align*}
    Moreover, it follows from the derived tensor-hom adjunction that
    \[ \Rderived{\Hom_\Lambda}({}_{B} B {}_{\Lambda}, {}_{B} \Rderived{\Hom_\cR}(B, \omega_\cR) {}_{\Lambda} ) \cong \Rderived{\Hom_\cR}({}_{B} B {}_{\Lambda} \Dtensor {}_{\Lambda} B {}_{B}, \omega_\cR ). \]
    Hence, 
    \begin{align*}
        \Rderived{\Hom_\Lambda}({}_{B} B {}_{\Lambda}, {}_{B} B {}_{\Lambda} {}^\dagger )^\dagger & \cong ({}_{B} B {}_{\Lambda} \Dtensor {}_{\Lambda} B {}_{B})^{\dagger \dagger} \\
        & \cong {}_{B} B {}_{\Lambda} \Dtensor {}_{\Lambda} B {}_{B}. \tag{since $(-)^\dagger$ is a duality}
    \end{align*}
    Thus, there is are bimodule isomorphisms
    \begin{align*}
         {}_{B} \Ext_\Lambda^{-t}(B^\dagger, B) {}_{B} & = \cohom{-t}( \Rderived{\Hom_\Lambda}({}_{B} B^\dagger {}_{\Lambda}, {}_{B} B {}_{\Lambda}))  \\
         & \cong \cohom{-t}({}_{B} B {}_{\Lambda} \Dtensor {}_{\Lambda} B {}_{B}) \\
         & = {}_{B} \Tor^\Lambda_t(B, B) {}_{B}. 
    \end{align*}

    Lemma \cite[3.6]{IR} states that there is a functorial isomorphism $(-)^\dagger \cong [-d] \cdot D$ on $\Db(\fl \cR)$. Hence, 
    \[ {}_{B} \Tor^\Lambda_t(B, B) {}_{B}  \cong  {}_{B} \Ext_\Lambda^{-t}(B^\dagger, B) {}_{B} \cong  {}_{B} \Ext_\Lambda^{d-t}(D B, B) {}_{B}. \qedhere\]
\end{proof}

\begin{prop} \label{complete local twist}
    Suppose that $\Lambda_\con$ is perfect over $\Lambda$ and $t$-relatively spherical for some $t \in \Z$ with $2 \leq t \leq d$. Then, 
    \[\Phi_{\cR} \colon \Db(\coh X) \to \Db(\coh X) \]
    is the spherical twist around the functor
    \[ \Psi^{-1}_\cP \cdot F \colon \Db(\modu \Lambda_\con) \to \Db(\coh X). \]
    The cotwist is 
    \begin{align} \label{ctw formula c local}
            \cC & =\Rderived{\Hom}_{\Lambda_\con}({}_{\Lambda_\con} \Tor^\Lambda_{t}(\Lambda_\con, \Lambda_\con) {}_{\Lambda_\con}, -)[-t-1].
    \end{align} 
    
    If $\dim_k \Lambda_\con < \infty$ and $t =d$, then the cotwist is naturally isomorphic to $[-d-1] \cdot \cN$, where $\cN$ is the Nakayama functor of $\Lambda_\con$. 
\end{prop}
\begin{proof}
   The fact that $\cT = \Rderived{\Hom_\Lambda}([\add R], -) \colon \Db(\modu \Lambda) \to \Db(\modu \Lambda)$ is the twist around the functor $F = - \Dtensor_{\Lambda_\con} \Lambda_\con \colon \Db(\modu \Lambda_\con) \to \Db(\modu \Lambda)$ follows from  \ref{summary theorem}\eqref{T is twist}. Moreover, since the pair $\cE = \CM \cR$ and $X = f_* \cP$ satisfies \cref{tilting set up} by \ref{alg conditions hold}, we may apply \ref{summary theorem}\eqref{C is cotwist} to conclude that the cotwist around $F$ is given by the functor 
   \[ \Rderived{\Hom}_{\Lambda_\con}({}_{\Lambda_\con} \Tor^\Lambda_{t}(\Lambda_\con, \Lambda_\con) {}_{\Lambda_\con}, -)[-t-1].\]
   Hence, \ref{spherical composed with equiv} imply that $\Phi_{\cR}$ is the twist around $\Psi^{-1}_\cP \cdot F$ and that the cotwist around this functor is given by \eqref{ctw formula c local}. Since both the twist and cotwist are equivalences by \ref{summary theorem}, it follows that $\Psi^{-1}_\cP \cdot F$ is spherical. 

   It follows from \ref{tor ext duality} that there is a $\Lambda_\con$-bimodule isomorphism
   \[  \Tor^\Lambda_t(\Lambda_{\con}, \Lambda_{\con})  \cong {}_{\Lambda_{\con}}\Hom_{\Lambda}(D \Lambda_{\con}, \Lambda_{\con}). \]
   Combining this with \eqref{ctw formula c local} necessarily
   \[ \cC = \Rderived{\Hom_{B}}(\Hom_A(D B, B), -).\]
   
    Moreover, if $\dim_k \Lambda_\con < \infty$ and $t=d$, then \ref{min res cor}(1) implies that $\Lambda_\con$ is self-injective. For self-injective algebras, the functor $\Rderived{\Hom_{B}}(\Hom_A(D B, B), -)$ is naturally isomorphic to the Nakayama functor by \cite[3.3]{IS}. Whence, the cotwist is naturally isomorphic to $[-d-1] \cdot \cN$.
\end{proof}

We end this section by remarking that checking whether $\Lambda_\con$ is relatively spherical can be simplified in this setting. 

\begin{prop}
    The algebra $\Lambda_\con$ is $t$-relatively spherical if and only if $\Ext^j_\Lambda(\Lambda_\con, \cS) = 0$ for $0 < j < t$, $t < j < d$.
\end{prop}
\begin{proof}
    Suppose that  $\Ext^j_\Lambda(\Lambda_\con, \cS) = 0$ for $0 < j < t$ and $t < j < d$. Then, $\Lambda_\con$ is $t$-relatively spherical if $\Ext^j_\Lambda(\Lambda_\con, \cS) = 0$ for $j > d$. Well, the proof of \ref{alg conditions hold} shows that $\Lambda \in \CM \cR$ and that $\omega_\Lambda$ is a projective $\Lambda$-module. Thus, $\Lambda$ is a Gorenstein $\cR$-order in the sense of \cite[1.6, 2.4]{IW1}. Whence, we may apply the Auslander-Buchsbaum formula for Gorenstein orders \cite[2.16]{IW1}
    \[ \pdim_\Lambda(\Lambda_\con) \leq \pdim_\Lambda(\Lambda_\con) + \depth_\cR(\Lambda_\con) = d. \]
    That is, necessarily, $\Ext^j_\Lambda(\Lambda_\con, \cS) = 0$ for $j > d$. 
\end{proof}

\begin{example} \label{BB example c local}
    Assume that $Y = \spec \cR$ has at worst  hypersurface singularities. Moreover, suppose that the exceptional locus of $f$ has codimension greater than one and that $f$ has fibres which are at most one-dimensional. Then, \cite[3.2.11]{VdB04} implies that $X$ has a tilting bundle $\cV = \cO_X \oplus \cP$. Let $\Lambda = \End_\cR(f_* \cV)$ and $\Lambda_\con = \uEnd_\cR(f_* \cV)$. 
    
    It is not hard to check that $\Ext^1_\cR(f_*\cP, f_*\cP) = 0$. Moreover, \cite[6.1]{Ei} implies that $\Omega^2 \cong \id$ on $\CM \cR$. It thus follows from \ref{syz theorem} that $\Lambda_\con$ is $3$-relatively spherical and perfect over $\Lambda$. Hence, we may apply \ref{complete local twist}. 

    Additionally, since the isomorphism is $\Omega^2 \cong \id$ is functorial and by the proof of \ref{summary theorem}\eqref{C is cotwist}, it follows that 
     \[ \Tor_t^\Lambda(\Lambda_\con, \Lambda_\con) \cong \uHom_\cR(f_* \cV, \Sigma^2 f_* \cV) \cong \uHom_\cR(f_* \cP, f_* \cP) = \Lambda_\con. \]
     as $\Lambda_\con$-bimodules. Hence, the cotwist is $1_{\Db(\modu \Lambda_\con)}[-4]$.
\end{example}

\subsection{Zariski local setting}

In this section assume that $f \colon X \to Y = \spec R$ is as in \ref{crepant setup} where $R$ need not be a complete local ring. In this context, let $\Lambda = \End_\cR(f_* \cP)$ and write $\Lambda_\con = \uEnd_R(f_* \cP)$.

\begin{notation}
    For each $x \in \maxspec R$, let $R_x$ denote the localisation of $R$ at $x$ and let $\cR_x$ denote the completion of $R_x$ at the unique maximal ideal. Moreover, given an $R$-module, let $M_x$ denote its localisation at $x$ and write $\widehat{M}$ for its completion. Recall that if $M$ is a finitely generated $R$-module, $\widehat{M} = M \otimes_R R_x \otimes_{R_x} \cR_x$. 
\end{notation}

\subsubsection{Constructing the equivalence}

The first goal of this section is to prove that the functor
 \begin{align}
       \Phi_R = \Psi_{\cP}^{-1} \cdot \Rderived{\Hom}_\Lambda([\add R], -) \cdot \Psi_{\cP} \colon \Db(\coh X) \tow{\sim} \Db(\coh X) 
\end{align}
is an equivalence if $\Lambda_\con$ is perfect over $\Lambda$ and relatively spherical.  This amounts to proving that the ideal $[\add R] \subseteq \Lambda$ is tilting. It is possible to reduce this fact to a statement in the complete local setting.

\begin{lemma} \label{c local implies z local pre}
    Let $A$ be a module finite $R$-algebra. Let $M$ be a finitely generated $A$-bimodule. If there is an $n$ such that for all $x \in \maxspec R$, $\pdim_{\widehat{A}}(\widehat{M}) \leq n$, $\pdim_{\widehat{A}^\op}(\widehat{M}) \leq n$, and $\widehat{M}$ is tilting, then $M$ is tilting.
\end{lemma}
\begin{proof}
   We will show that $M$ is tilting because it satisfies (a)-(c) of \cite[1.8]{mi}. For (a), to see that $M$ is perfect as a right $A$-module, let 
   \[ \hdots \to P_n \tow{p_n} P_{n-1} \tow{p_{n-1}} \hdots \to P_0 \to M \to 0 \]
  be a projective resolution $p \colon P \to M \to 0$ in $\modu A$. For finitely generated modules over noetherian rings, localisation and completion are exact functors, and so the complex $\widehat{p} \colon \widehat{P} \to \widehat{M} \to 0$ is  exact for all $x \in \maxspec R$. Moreover, it is clear that if $P_i$ is a summand of $A^k$ for some $k \in \Z$, then $\widehat{P}_i = P_i \otimes_R \cR_x$ is a summand of $A^k \otimes_R \cR_x = \widehat{A}^k$. Hence, $\widehat{P}_i$ is a projective $\widehat{A}$-module, and we may conclude that $\widehat{p} \colon \widehat{P} \to \widehat{M} \to 0$ is a projective resolution of $\widehat{M}$. 
  
  Because $\pdim_{\widehat{A}}(\widehat{M}) \leq n$, necessarily $\ker \widehat{p}_{n}$ is a projective $\widehat{A}$-module. Equivalently (because localisation and completion are exact), $\widehat{\ker p}_{n}$ is a projective $\widehat{A}$-module  for all $x \in \maxspec R$. That is, 
    \[ \Ext^1_{A}(\ker p_{n}, N) \otimes_R \cR_x = \Ext^1_{\widehat{A}}(\widehat{\ker p}_{n}, \widehat{N}) = 0 \]
    for all $N \in \modu A$ by \cite[2.15]{DW1}. Since $\Ext^1_A(\ker p_n, N) \otimes_R \cR_x$ vanishes for all $x \in \maxspec R$, it must be that $\Ext^1_A(\ker p_{n}, N) = 0$. Consequently,  $\ker p_{n}$ is a finitely generated projective $A$-module, and we may truncate $p \colon P \to M \to 0$ to a projective resolution of length $n$. That is, $\pdim_A(M) \leq n$. The proof that $\pdim_{A^\op}(M) \leq n$ is similar, and so it is omitted. 

   To prove (b), we will show that the natural right multiplication map $\lambda \colon A \to \Rderived{\Hom_A}(M, M)$ is a bimodule isomorphism in the derived category. First, since $\widehat{M}$ is tilting for all $x \in \maxspec R$, then the natural right multiplication map $\widehat{A} \to \Rderived{\Hom_{\widehat{A}}}(\widehat{M}, \widehat{M})$ is a quasi-isomorphism. Thus, for $i \neq 0$,
   \[ 0 = \Ext^i_{\widehat{A}}(\widehat{M}, \widehat{M}) = \Ext^i_A(M, M) \otimes_R \cR_x. \]
   Since $\Ext^i_A(M, M) \otimes_R \cR_x$ vanishes for all $x \in \maxspec R$, necessarily $\Ext^i_A(M, M) = 0$ for $i \neq 0$. Therefore, it suffices to show that the degree zero multiplication map $\cohom{0}(\lambda) \colon A \to \Hom_A(M, M)$ is an isomorphism. 
   
   Well, $\widehat{M}$ is tilting so that the natural multiplication map $\gamma_x \colon \widehat{A} \to \Hom_{\widehat{A}}(\widehat{M}, \widehat{M})$ is an isomorphism. It easy to check that $\cohom{0}(\lambda) \otimes_R \cR_x = \gamma_x$, and so $\cohom{0}(\lambda) \otimes_R \cR_x$ is an isomorphism for all $x \in \maxspec R$. Thus, $\cohom{0}(\lambda)$ must be an isomorphism, as required. The dual argument works to show that (c) holds, that is, the natural left multiplication map is an isomorphism. 
\end{proof}

Since $[\add R] \otimes_R R_x \otimes_{R_x} \cR_x \cong [\add \cR_x]$ as $\widehat{\Lambda}$-bimodules by \cite[2.16]{DW1}, we obtain the following corollary.

\begin{cor} \label{c local implies z local}
    If $[\add \cR_x]$ is tilting for all closed points $x \in Z$, then $[\add R]$ is tilting.
\end{cor}
\begin{proof}
    The goal is to apply  \ref{c local implies z local pre} to $[\add R]$ since $[\add R] \otimes_R R_x \otimes_{R_x} \cR_x \cong [\add \cR_x]$. Therefore, we need to check that the assumptions of \ref{c local implies z local pre} are satisfied. 
    
    If $x \in Z$ and $[\add \cR_x]$ is a tilting module, then $\pdim_{\widehat{\Lambda}} [\add \cR_x] < \infty$ and $\pdim_{\widehat{\Lambda}^\op} [\add \cR_x] < \infty$. We claim that $\pdim_{\widehat{\Lambda}} [\add \cR_x] \leq d$ and $\pdim_{\widehat{\Lambda}^\op} [\add \cR_x] \leq d$ for all $x \in Z$.

    First, note that the pair $\cE = \CM \cR_x$ and $f_* \cP \otimes_R \cR_x$ are in the situation of \ref{alg conditions hold}, so that $\widehat{\Lambda} = \End_{\cR_x}(\widehat{f_* \cP})$ is a Gorenstein $\cR$-order. Thus, we may apply the Auslander-Buchsbaum formula for Gorenstein orders \cite[2.16]{IW1} to conclude that
    \[ \pdim_{\widehat{\Lambda}} [\add \cR_x] \leq \pdim_{\widehat{\Lambda}} [\add \cR_x] + \depth_{\cR_x}([\add \cR_x]) = d \]
    (Here, we use the fact that $R$ is equicodimensional). The dual argument works to prove that $\pdim_{\widehat{\Lambda}^\op} [\add \cR_x] \leq d$. 
   
    Therefore, it suffices to prove that $\pdim_{\widehat{\Lambda}}([\add \cR_x]) \leq d$, $\pdim_{\widehat{\Lambda}^\op}([\add \cR_x]) \leq d$, and $[\add \cR_x]$ is tilting for all $x \notin Z$. This is an immediate consequence of \cite[4.4]{DW4} since this result implies that for $x \notin Z$, $\widehat{\Lambda}_\con = 0$. Therefore, because $[\add \cR_x]$ is the kernel of the quotient $\widehat{\Lambda} \to \widehat{\Lambda}_{z, \con}$, it follows that $[\add \cR_x] \cong \widehat{\Lambda}$, and clearly this satisfies the desired properties.
\end{proof}

The following shows that the relatively spherical property can also be checked complete locally.

\begin{lemma} \label{can check on stalks}
    The following are equivalent
    \begin{enumerate}
        \item $\Lambda_\con$ is t-relatively spherical and perfect over $\Lambda$,
        \item $\widehat{\Lambda}_{z, \con}$ is t-relatively spherical and perfect over $\widehat{\Lambda}_z $ for all closed points $z \in Z$.
    \end{enumerate}
\end{lemma}
\begin{proof}
        The key to this proof is that $\widehat{\Lambda}_\con \cong \widehat{\Lambda}/[\add \cR_x]$ for all $x \in \maxspec R$, due to \cite[2.16]{DW1}.

        The implication $(1) \Rightarrow (2)$ is clear. Suppose that $(2)$ holds. Then, for $j \neq 0, t$ and $z \in Z$
        \[0 = \Ext^j_{\widehat{\Lambda}}( \widehat{\Lambda}_{z, \con}, \cS_z ) \cong \Ext^j_{\Lambda}(\Lambda_\con, \cS) \otimes \cR_z. \]
       Moreover, $\widehat{\Lambda}_{x, \con} = 0$ for $x \notin Z$ \cite[4.4]{DW4}. Hence, $\Ext^j_{\Lambda}(\Lambda_\con, \cS)$ vanishes on all closed points, and so it must be zero. 

       Moreover, since $(2)$ holds, $\pdim_{\widehat{\Lambda}_z}(\widehat{\Lambda}_{z, \con}) = t$ for all $z \in Z$. Hence, the same argument to prove (a) in the proof of \ref{c local implies z local pre} works to show that $\pdim_{\Lambda}(\Lambda_\con) \leq t$. 
\end{proof}

\begin{prop} \label{zariski local equiv}
    Suppose that $\Lambda_\con$ perfect over $\Lambda$ and $t$-relatively spherical for some $t \in \Z$ with $2 \leq t \leq d$. Then there is a derived autoequivalence
    \begin{align}
       \Phi_R = \Psi_{\cP}^{-1} \cdot \Rderived{\Hom}_\Lambda([\add R], -) \cdot \Psi_{\cP} \colon \Db(\coh X) \tow{\sim} \Db(\coh X) 
    \end{align}
\end{prop}
\begin{proof}
    To prove the statement, it suffices to show that the ideal $[\add R] \subseteq \Lambda$ is tilting which, by \ref{c local implies z local}, follows from the fact that $[\add_{\cR_x}]$ is tilting for all $x \in Z$. By \ref{can check on stalks} and the fact that completion and localisation are exact functors, it must be that $\widehat{\Lambda}_\con$ is perfect over $\widehat{\Lambda}$ and $t$-relatively spherical. Thus, \ref{complete local equiv} implies that $[\add_{\cR_x}]$ is tilting for all $x \in Z$.
\end{proof}

\subsubsection{Proving that the twist is spherical}

Consider the derived restriction of scalars functor given by $F \colonequals - \Dtensor_{\Lambda_\con} \Lambda_\con {}_\Lambda \colon \Db(\modu \Lambda_\con) \to \Db(\modu \Lambda)$. The goal for the remainder of this section is to show that if $\Lambda_\con$ is perfect over $\Lambda$ and $t$-relatively spherical, then $\Phi_{R} \colon \Db(\coh X) \to \Db(\coh X)$ is a spherical twist around the functor $\Psi^{-1}_\cP \cdot F \colon \Db(\modu \Lambda_\con) \to \Db(\coh X)$.

 \begin{lemma} \label{complete local tor}
    Fix $x \in \maxspec R$. There is a $\widehat{\Lambda}_\con$-bimodule isomorphism 
     \[ \Tor^{\Lambda}_{i}(\Lambda_\con, \Lambda_\con) \otimes \cR_x \cong \Tor^{\widehat{\Lambda}}_{i}(\widehat{\Lambda}_\con, \widehat{\Lambda}_\con) \]
 \end{lemma}
\begin{proof}
    The proof follows as in e.g.\@ \cite[3.2.10]{wei}.
\end{proof}

\begin{prop} \label{z local tw and ctw}
    If $\Lambda_\con$ is perfect over $\Lambda$ and $t$-relatively spherical then 
     \[ \Phi_{R} \colon \Db(\coh X) \to \Db(\coh X) \]
     is the twist around the functor 
     \[ \Psi^{-1}_\cP \cdot F \colon \Db(\modu \Lambda_\con) \to \Db(\coh X). \]
     The cotwist is 
     \begin{align} \label{ctw z local}
         \cC_{R} = \Rderived{\Hom}_{\Lambda_\con}({}_{\Lambda_\con} \Tor^\Lambda_{t}(\Lambda_\con, \Lambda_\con) {}_{\Lambda_\con}, -)[-t-1]. 
     \end{align}
\end{prop}
\begin{proof}
    In light of \ref{twist t res}, the functor $\cT =\Rderived{\Hom_\Lambda}([\add R], -) \colon \Dcat(\Mod \Lambda) \to \Dcat(\Mod \Lambda)$ is the twist around $F = - \Dtensor_{\Lambda_\con} \Lambda_\con\colon \Dcat(\Mod \Lambda_\con) \to \Dcat(\Mod \Lambda)$. 

    Moreover, we claim that $F$, as well as its right and left adjoints, preserve bounded complexes of finitely generated modules. By \cite[6.7]{DW1}, it suffices to prove that $\Lambda_\con$ is biperfect. By assumption, $\pdim_\Lambda \Lambda_\con = n < \infty$. It follows from \ref{can check on stalks} that $\widehat{\Lambda}_{x, \con}$ is t-relatively spherical and perfect. By \ref{perfect on both sides}, this means that $\pdim_{\widehat{\Lambda}^\op} \widehat{\Lambda}_\con \leq t$ for all $x \in \maxspec R$. By the dual of proof of \ref{c local implies z local pre}, it must be that $\pdim_{\Lambda^\op} \Lambda_\con \leq t$. Consequently, $\Lambda_\con$ is biperfect.

    Thus, we conclude that the restriction of $\cT$ to $\Db(\modu \Lambda)$ is the twist around the restriction of $F$ to $\Db(\modu \Lambda_\con)$. As a result, the fact that $\Phi_R$ is the twist around $\Psi^{-1}_\cP \cdot F$ follows immediately from \ref{spherical composed with equiv}. 

    Observe, moreover, that the cotwist around $\Psi^{-1}_\cP \cdot F$ can be identified with the cotwist around $F$ by \ref{spherical composed with equiv}. Therefore, by \ref{ctw res cohom in 2 degrees}, the cotwist is as claimed if the complex $K = {}_{\Lambda_\con} \Lambda_\con {}_\Lambda \otimes^{\Lderived{}} {}_\Lambda \Lambda_\con {}_{\Lambda_\con}$ has cohomology concentrated in degrees $0$ and $-t$. We will show that these properties hold in this setup because they hold complete locally.

    Fix $x \in \maxspec R$. By \ref{complete local tor},
    \[ \cohom{-i}(K) \otimes \cR_x= \Tor^\Lambda_i(\Lambda_\con, \Lambda_\con) \otimes \cR_x = \Tor^{\widehat{\Lambda}}_{i}(\widehat{\Lambda}_\con, \widehat{\Lambda}_\con) \]
    Hence, \ref{K} implies that $\cohom{i}(K) \otimes \cR_x = 0$ for $i \neq 0, -t$. Since this is true for all $x \in \maxspec R$, it must be that $\cohom{i}(K) = 0$ for $i \neq 0, -t$. Thus, even in the Zariski local setting, we may specify the cotwist and dual cotwist as in \eqref{ctw z local}. 
\end{proof}

\medskip
 We next show that $ {}_{\Lambda_\con} \Tor^\Lambda_{t}(\Lambda_\con, \Lambda_\con) {}_{\Lambda_\con}$ is a tilting bimodule so that, by \ref{z local tw and ctw}, the cotwist is an equivalence. To do this, we will use \ref{c local implies z local pre}.

\begin{lemma} \label{tor is tilting z locally}
     Suppose that $\Lambda_\con$ is perfect and $t$-relatively spherical. Then, the $\Lambda_\con$-bimodule $\Tor^\Lambda_{t}(\Lambda_\con, \Lambda_\con)$ is tilting.
\end{lemma}
\begin{proof}
    Set $T = {}_{\Lambda_\con} \Tor^\Lambda_{t}(\Lambda_\con, \Lambda_\con) {}_{\Lambda_\con}$. Then,  $\widehat{T} = \Tor^{\widehat{\Lambda}}_{t}(\widehat{\Lambda}_\con, \widehat{\Lambda}_\con)$ by \ref{complete local tor}. It follows from \ref{can check on stalks} that $\widehat{\Lambda}_\con$ is $t$-relatively spherical and perfect, and so \ref{complete local twist} implies that $\widehat{T}$ is tilting for all $x \in \maxspec R$. Note, moreover that $\widehat{T}$ is projective on either side by \ref{biperfect tor}. Hence, \ref{c local implies z local pre} implies that $T$ is tilting.
\end{proof}

In summary, we have proven the following theorem in this section. 

\begin{theorem} \label{z local twist}
     If $\Lambda_\con$ is $t$-relatively spherical and perfect over $\Lambda$, then 
     \[ \Phi_{R} \colon \Db(\coh X) \to \Db(\coh X) \]
     is a spherical twist around the functor 
     \[ \Psi^{-1}_\cP \cdot F \colon \Db(\modu \Lambda_\con) \to \Db(\coh X). \]
    The cotwist is 
    \begin{align} \label{ctw formula z local}
            \cC & =\Rderived{\Hom}_{\Lambda_\con}({}_{\Lambda_\con} \Tor^\Lambda_{t}(\Lambda_\con, \Lambda_\con) {}_{\Lambda_\con}, -)[-t-1].
    \end{align} 
\end{theorem}
\begin{proof}
    The fact that $\Phi_{R}$ is a spherical twist follows from \ref{zariski local equiv}, \ref{z local tw and ctw} and \ref{tor is tilting z locally}.
\end{proof}

\begin{example} \label{BB example}
    In this example we will show that \ref{z local twist} extends \cite[5.18]{BB} by dropping the one-dimensional fibre assumption on $f$ and the hypersurface singularity assumption on $Y$. 

    Assume that $Y = \spec \cR$ has at worst  hypersurface singularities. Moreover, suppose that the exceptional locus of $f$ has codimension greater than one and that $f$ has fibres which are at most one-dimensional. Similarly to \ref{BB example c local}, it follows from \cite[3.2.11]{VdB04} that $X$ has a tilting bundle $\cV = \cO_X \oplus \cP$. Let $\Lambda = \End_\cR(f_* \cV)$ and $\Lambda_\con = \uEnd_\cR(f_* \cV)$.

    By \ref{BB example c local}, $\widehat{\Lambda}_\con$ is $3$-relatively spherical and perfect over $\widehat{\Lambda}$ for all $x \in \maxspec R$. Therefore, by \ref{can check on stalks}, $\Lambda_\con$ is $3$-relatively spherical and perfect over $\Lambda$, and so we may apply \ref{z local twist}.

    Furthermore, \ref{z local twist} implies that the cotwist is $\cC=\Rderived{\Hom}_{\Lambda_\con}(T, -)[-4]$ where the bimodule $T = {}_{\Lambda_\con} \Tor^\Lambda_{t}(\Lambda_\con, \Lambda_\con) {}_{\Lambda_\con}$. It follows from \ref{BB example c local} and \ref{complete local tor} that $\widehat{T} = \widehat{\Lambda}_\con$. Hence, we may conclude that $T$ is a projective $\Lambda_\con$-module, whence
    \[ \cC = \Hom_{\Lambda_\con}(T, -)[-4]. \]
\end{example}

\section{Spherical twists induced by crepant contractions more generally} \label{crepant contractions section}

\subsection{Background on noncommutative schemes}

    We begin by recalling some basic theory about noncommutative schemes. Mostly, the point of this subsection is to set notation and also to convince the reader that noncommutative schemes behave somewhat similarly to commutative schemes.

\subsubsection{Noncommutative schemes and their modules}

\begin{definition}
\begin{enumerate}
    \item A \textit{noncommutative scheme} is a pair $(\cX, \cA)$ where $\cX$ is a scheme and $\cA$ is a sheaf of $\cO_\cX$-algebras which is quasi-coherent as an $\cO_\cX$-module. We will often abbreviate noncommutative scheme as NC scheme.
    \item A NC scheme $(\cX, \cA)$ is said to be \textit{quasi-projective} if the underlying scheme $\cX$ is quasi-projective. 
    \item Moreover, $(\cX, \cA)$ is said to be \textit{noetherian} if $\cX$ is a noetherian scheme and $\cA$ is a coherent $\cO_\cX$-module. 
    \item Given a NC scheme $(\cX, \cA)$, let $\Mod (\cX, \cA)$ denote the category of $\cA$-modules. Moreover, the category of quasi-coherent (resp. coherent) $\cA$-modules will be denoted as $\qcoh (\cX, \cA)$ (resp., $\coh (\cX, \cA)$).
    \item Given NC schemes $(\cX, \cA)$ and $(\cX, \cB)$, an \textit{$\cA$-$\cB$-bimodule} is a sheaf of $\cB \otimes_{\cO_X} \cA^\op$-modules on $\cX$.
    \item A \textit{morphism of NC schemes} $f \colon (\cX, \cA) \to (\cY, \cB)$ is a pair $(f_\cX, f^\#)$ where $f_\cX \colon \cX \to \cY$ is a morphism of schemes and $f^\# \colon f_\cX^{-1} \cB \to \cA$ is a morphism of $f_\cX^{-1} \cO_Y$-algebras.
\end{enumerate}
\end{definition}

\begin{remark}
    Due to standard sheaf theory, since $f_\cX^{-1}$ is left adjoint to $(f_\cX)_*$, the morphism $f^\# \colon f_\cX^{-1} \cB \to \cA$ induces a morphism $f^\flat \colon \cB \to (f_\cX)_* \cA$. 
\end{remark}

\begin{definition}
    Similarly to the commutative case, a morphism $f \colon (\cX, \cA) \to (\cY, \cB)$ of NC schemes induces the pullback and pushforward adjoint pair.

\[\begin{tikzcd}
     \Mod (\cX, \cA) \arrow[r, "f_*"{name=LA}, bend right = 4ex, swap] & \Mod (\cY, \cB). \arrow[l, "f^*"{name=RA}, bend right =4ex, swap]
     \arrow[phantom, from=LA, to=RA, "\scriptstyle\boldsymbol{\perp}"description]
\end{tikzcd}
\]
    \begin{enumerate}
        \item The functor
        \[ f_* \colon \Mod (\cX, \cA) \to \Mod (\cY, \cB) \]
        sends a $\cB$-module $\cF$ to the sheaf $f_* \cF$ whose sections on a open set $V \subset \cY$ are defined as
        \[ f_* \cF(V) = \cF( f^{-1} V). \]
        The $\cB$-module structure on this sheaf is induced by the morphism $f^\flat \colon \cB \to (f_\cX)_* \cA$. 
        \item The functor
    \[ f^* \colon \Mod (\cY, \cB) \to \Mod (\cX, \cA) \]
    sends a $\cB$-module $\cF$ to the tensor product $f^{-1} \cF \otimes_{f^{-1} \cB} \cA$. 
    \end{enumerate}
\end{definition}

Observe that we may restrict the functor $f^*$ to $f^* \colon \qcoh(\cX, \cB) \to \qcoh(\cX, \cA)$. If, further, $(\cX, \cA)$ and $(\cY, \cB)$ are noetherian, then $f^*$ restricts to $f^* \colon \coh(\cX, \cB) \to \coh(\cX, \cA)$. 

\begin{lemma} \label{category properties}
    Let $(\cX, \cA)$ be a noetherian NC scheme. Then, the categories $\Mod (\cX, \cA)$, $\qcoh (\cX, \cA)$ and $\coh(\cX, \cA)$ are abelian.
\end{lemma}
\begin{proof}
    The statement for $\Mod (\cX, \cA)$ is due to \cite[18.1.6(v)]{KS} (they actually prove the stronger result that $\Mod (\cX, \cA)$ is Grothendieck). The fact that $\qcoh (\cX, \cA)$ is abelian (and also Grothendieck) follows from \cite[5.6]{Ku}. Finally, $\coh(\cX, \cA)$ is abelian by \cite[Exercise 8.23]{KS}.
\end{proof}

Since $\qcoh (\cX, \cA)$ is abelian, so is the category $\chain( \qcoh (\cX, \cA) )$ of chain complexes of quasi-coherent $\cA$-modules. Moreover, we may also construct the derived category $\Dcat(\qcoh (\cX, \cA))$. 

\subsubsection{Resolution properties}

We are interested in derived functors between derived categories of NC schemes, and so it will be important to establish the existence of \textit{$K$-flat} and \textit{$K$-injective} resolutions (recalled in \cref{derived functors}) in $\qcoh (\cX, \cA)$ . The existence of \textit{$Lp$-resolutions} under certain settings will also play an important role. 

\begin{definition} Let $(\cX,\cA)$ be a NC scheme.
\begin{enumerate}
    \item An $\cA$-module $\cF$ is \textit{locally projective} if $\cF(U)$ is a projective $\cA(U)$-module for all affine open $U \subset \cX$.
    \item Let $\cF \in \Dcat(\qcoh (\cX, \cA))$. Then, an \textit{$Lp$-resolution} of $\cF$ is a quasi isomorphism $\cP \to \cF$ where $\cP$ is a $K$-flat complex of locally projective $\cA$-modules.
\end{enumerate}
\end{definition}

\begin{lemma} \label{resolution properties}
    Let $(\cX, \cA)$ be a noetherian NC scheme. Then,
    \begin{enumerate}
        \item Any complex in $\chain( \qcoh (\cX, \cA) )$ admits a $K$-injective resolution. \label{exists Kinj}
        \item Any complex in $\chain( \qcoh (\cX, \cA) )$ admits a $K$-flat resolution. \label{exists Kflat}
        \item If $(\cX, \cA)$ is quasi-projective, then any complex in $\chain( \qcoh (\cX, \cA) )$ admits an Lp-resolution.  \label{exists Lp}
    \end{enumerate}
\end{lemma}
\begin{proof}    
    \eqref{exists Kinj} follows from \cite[5.4]{AJS}, \eqref{exists Kflat} is \cite[3.9]{BDG}, and \eqref{exists Lp} is \cite[3.7, 3.10]{BDG}.
\end{proof}

\subsubsection{Derived functors}

Given NC schemes $(\cX, \cA)$ and $(\cX, \cB)$, let $\cF$ be a complex of $\cB$-$\cA$-bimodules and $\cG$ a complex of $\cA$-$\cB$-bimodules. In what follows, we present several technical lemmas, the point of which is to establish that the functors 
\begin{align*}
    \Rderived{\Hom_\cA}(\cF, -) & \colon \Dcat( \Mod (\cX, \cA) ) \to \Dcat( \Mod( \cB(\cX) ) \\
     \Rderived{\cHom_\cA}(\cF, -) & \colon \Dcat( \Mod (\cX, \cA) ) \to \Dcat( \Mod (\cX, \cB) )  \\
      (-) \Dtensor_\cA \cG & \colon \Dcat( \Mod (\cX, \cA) ) \to \Dcat( \Mod (\cX, \cB) ) 
\end{align*}
behave as expected. 

\begin{prop}[{\cite[section 6]{Sp}, \cite[3.12]{BDG}}] \label{nc derived functors bg}

Given NC schemes $(\cX, \cA)$ and $(\cX, \cB)$, let $\cF$ be a complex of $\cB$-$\cA$-bimodules and $\cG$ a complex of $\cA$-$\cB$-bimodules.

\begin{enumerate}
    \item  The derived functors
\begin{align}
    \Rderived{\Hom_\cA}(\cF, -) & \colon \Dcat( \Mod (\cX, \cA) ) \to \Dcat( \Mod( \cB(\cX) ) \label{nc global derived hom} \\
     \Rderived{\cHom_\cA}(\cF, -) & \colon \Dcat( \Mod (\cX, \cA) ) \to \Dcat( \Mod (\cX, \cB) ) \label{nc local derived hom} \\
      (-) \Dtensor_\cA \cG & \colon \Dcat( \Mod (\cX, \cA) ) \to \Dcat( \Mod (\cX, \cB) ) \label{nc derived tensor}
\end{align}
    exist. \label{existence of f}
    \item If $\cF$ is a complex of bimodules which are locally finitely presented as $\cA$-modules and  $\cG$ is a complex of quasi-coherent bimodules, then the functors \eqref{nc local derived hom} and \eqref{nc derived tensor} restrict to the derived category of quasi-coherent modules. \label{res qcoh}
    \item Given $\cM \in \Dcat(\qcoh(\cX, \cA))$, the complexes $\Rderived{\Hom_\cA}(\cF, \cM)$ and $\Rderived{\cHom_\cA}(\cF, \cM)$ can be calculated by using a $K$-injective resolution of $\cM$. Moreover, the complex $\Rderived{\cHom_\cA}(\cF, \cM)$ can also be calculated using an Lp-resolution of $\cF$, if it exists. \label{how to compute h}
    \item The complex $\cM \Dtensor_\cA \cG$ can be computed by using a $K$-flat resolution of $\cM$. \label{how to compute t}
    \item There are natural isomorphisms
    \begin{align}
        & \Rderived{\Hom_\cB}(\cF \Dtensor_\cA \cG, - ) \cong \Rderived{\Hom_\cA}(\cF,\Rderived{\cHom_\cB}(\cG, - ) ) \\
        & \Rderived{\cHom_\cB}(\cF \Dtensor_\cA \cG, - ) \cong \Rderived{\cHom_\cA}(\cF,\Rderived{\cHom_\cB}(\cG, - ) ).
    \end{align}
    which induce the adjunction $- \Dtensor_\cA \cG \dashv \Rderived{\cHom}_{\cB}(\cG, -)$. 
\label{adjunction}
\end{enumerate}
\end{prop}
\begin{proof}
    \eqref{existence of f}, \eqref{how to compute h}, \eqref{how to compute t} and \eqref{adjunction} are \cite[3.12]{BDG}. The proof of the fact that the functors \eqref{nc global derived hom}, \eqref{nc local derived hom} and \eqref{nc derived tensor} preserve quasi-coherence under the assumptions of \eqref{res qcoh} is the same as in the commutative case, which is standard. 
\end{proof}

\begin{definition} Recall that an object $\cF \in \Dcat(\qcoh (\cX, \cA))$ is called \textit{perfect} if there exists an affine open covering  $X = \bigcup_i U_i$ such that $\cF |_{U_i}$ for each $i$ is quasi-isomorphic to a bounded complex of sheaves which are summands of finite free $\cA |_{U_i}$-modules.
\end{definition}

\begin{prop} \label{res to coh}
    Let $(\cX, \cA)$ and $(\cX, \cB)$ be quasi-projective noetherian NC schemes. Suppose that $\cF$ is a complex of coherent $\cB$-$\cA$-bimodules which is perfect as complex of $\cA$-modules. Moreover, suppose that $\cG$ is a complex of $\cA$-$\cB$-bimodules which is perfect as a complex of $\cA^\op$-modules. Then, the functors \eqref{nc global derived hom}, \eqref{nc local derived hom} and \eqref{nc derived tensor} restrict to 
    \begin{align}
    \Rderived{\Hom_\cA}(\cF, -) & \colon \Db( \coh (\cX, \cA) ) \to \Db( \modu \cB(\cX) ) \label{nc global derived hom Db} \\
     \Rderived{\cHom_\cA}(\cF, -) & \colon \Db( \coh (\cX, \cA) ) \to \Db( \coh (\cX, \cB) ) \label{nc local derived hom Db} \\
      (-) \Dtensor_\cA \cG & \colon \Db( \coh (\cX, \cA) ) \to \Db( \coh (\cX, \cB) ) \label{nc derived tensor Db}
\end{align}
\end{prop}
\begin{proof}
    Since $\cX$ is quasi-projective, it follows from \ref{resolution properties} \eqref{exists Lp} that $\coh (\cX, \cA \otimes \cB^\op)$ and $\coh (\cX, \cA)$ have enough Lp-modules and so the bounded complex $\cF$ admits an Lp-resolution $\cQ' \to \cF$ as a $\cA$-module and an Lp-resolution $\cQ \to \cF$ as an $\cA \otimes \cB^\op$-module. Since $\cF$ is perfect as an $\cA$-module, we may assume that $\cQ'$ is bounded. Let $\cM \in \Db( \coh (\cX, \cA) )$. 
    
    We would first like to show that $\Rderived{\cHom_\cA}(\cF, \cM)$ is a bounded complex of coherent $\cB$-modules. To see that it is bounded, note that, as a complex of vector spaces, $\Rderived{\cHom_\cA}(\cF, \cM) \cong \cHom_\cA^*(\cQ', \cM)$. Since both $\cQ'$ and $\cM$ are bounded, so is $\cHom_\cA^*(\cQ', \cM)$. Hence, necessarily,  $\Rderived{\cHom_\cA}(\cF, \cM)$ is bounded as a complex of $\cB$-modules.

    Next, we would like to show that $\Rderived{\cHom_\cA}(\cF, \cM)$ is a complex of coherent modules. The question is local, so we may assume that $\cX = \spec R$ where, by assumption, $R$ is a noetherian ring, and $\cA = A^{\sim}$, $\cB = B^\sim$ with $A$, $B$ finitely presented $R$-algebras. Note that $A$ and $B$ are noetherian, since they are module finite over a noetherian ring. 
    
    Since $\cQ$ and $\cM$ are coherent, then there are module $Q \in \modu A \otimes_R B^\op$ and $M \in \modu A$ with $\cQ = Q^\sim$ and $\cM = M^\sim$. So, coherence of $\Rderived{\cHom_\cA}(\cF, \cM) = \cHom_\cA^*(\cQ, \cM)$ as a complex $\cB$-modules is equivalent to $\Hom^*_A(Q, M)$ being a complex of finitely presented $B$-modules. Since $M$ is bounded and $A$ and $B$ are noetherian, it suffices to show that $\Hom_A(U, N)$ is a finitely generated $B$-module if $U \in \modu A \otimes_R B^\op$ and $N \in \modu A$. 

    Since $A$, $B$ and $A \otimes_R B^\op$ are module finite over $R$, then finite generation of a module over $B$ and over $A \otimes_R B^\op$ is equivalent to finite generation over $R$. Since $R$ is noetherian, the standard argument in the commutative case implies that $\Hom_A(U, N)$ is finitely generated over $R$. Therefore, it is finite over $B$, which concludes the proof.
    
    The proofs of \eqref{nc global derived hom Db} and \eqref{nc derived tensor Db} are analogous to the above proof, and so are omitted.
\end{proof}

This section applies the theory developed in section \ref{spherical by frob} to obtain derived autoequivalences of schemes with at worst Gorenstein singularities.

\subsection{Setting}

   Consider a morphism $f \colon X \to Y$ of noetherian schemes. Recall that a bundle $\cP$ on $X$ is \textit{tilting relative to $Y$} in the sense of e.g.\@ \cite[2.3(2)]{DW4} if there is an equivalence
    \begin{align}
        \Psi_{\cP} \colonequals \Rderived{f_*}\Rderived{\cHom}_X(\cP, -) \colon \Db(\coh X) \tow{\sim} \Db(\coh(Y, \cA))
    \end{align}
    where $\cA \colonequals \Rderived{f_*} \cEnd_X(\cP) = f_* \cEnd_X(\cP)$ is a sheaf of $\cO_Y$-algebras and $\Db(\coh(Y, \cA))$ is the bounded derived category of modules over $\cA$. 
  
Throughout, we will work within the following setup.

\begin{setup} \label{crepant setup}
    Let $f \colon X \to Y$ be a crepant (complete local) contraction. Assume further that $Y$ is a Gorenstein $d$-fold with $d \geq 2$, and that $X$ admits a relative tilting bundle $\cP$ containing $\cO_X$ as a summand. 
\end{setup}

\begin{notation}
    With $f$ as in \ref{crepant setup}, write $Z$ for the locus of points of $Y$ onto which $f$ is not an isomorphism. 
\end{notation}

\begin{lemma}[{\cite[2.5]{DW4}}] \label{DW end iso over base}
    Under the setup \ref{crepant setup}, the natural map $f_* \cEnd_X(\cP) \to \cEnd_Y(f_* \cP)$ is an isomorphism, and so $\cP$ induces a derived equivalence 
    \begin{align*}
        \Psi_{\cP} \colonequals \Rderived{f_*}\Rderived{\cHom}_X(\cP, -) \colon \Db(\coh X) \tow{\sim} \Db(\coh(Y, \cA))
    \end{align*}
    where $\cA = \cEnd_Y(f_* \cP)$.
\end{lemma}

\subsection{The sheaf of contraction algebras}

Given a crepant contraction $f \colon X \to Y$ satisfying the general setup \ref{crepant setup}, this section recalls the construction of the noncommutative enhancement $(Y, \cA_\con)$. The NC scheme $(Y, \cA_\con)$ was first constructed in \cite{DW4} and is an invariant associated to $f$ which retains much information about the contraction $f$ and the category $\Dcat(\qcoh X)$. We may view $(Y, \cA_\con)$ as the global analogue of the contraction algebra introduced by \cite{DW1} and recalled in \cref{crepant contractions affine section}.

\begin{notation} \label{sheafy notation} Under  the general setup \ref{crepant setup},  
    \begin{enumerate}
        \item Let $\cQ \colonequals f_* \cP$.
        \item Write $\cA$ for the sheaf of $\cO_Y$-algebras 
        \[ \cA \colonequals \Rderived{f_*} \cEnd_X(\cP) = f_* \cEnd_X(\cP) \cong \cEnd_Y(\cQ) \]
        where the last isomorphism is \ref{DW end iso over base}.
        \item As in \cite[2.8]{DW4}, let $\cI$ be the ideal sheaf of $\cA$ specified by the rule: For every open subset $j \colon V \hookrightarrow Y$, 
        \[ \cI(V) = \{ s \in \End_V(\cQ |_V) \mid s_v \colon \cQ_v \to \cQ_v \in [\proj \cO_{Y, v}] \, \forall v \in V \} \]
        where $[\proj \cO_{Y, v}]$ denotes the ideal of $\End_{\cO_{Y, v}}(\cQ_v)$ consisting of maps which factor through $\proj \cO_{Y, v}$.
        \item The sheaf $\cA_\con \colonequals \cA / \cI$ is called the sheaf of contraction algebras in \cite[2.12]{DW4}.
        \item For an affine open subset $j \colon U = \spec R \hookrightarrow Y$ of $(Y, \cO_Y)$, write $\Lambda_U \colonequals \cA(U)$. It follows that  $\cI(U) = [\proj R]$ \cite[2.10]{DW4} and $\cA_\con(U) = \Lambda_U / [\proj R] \colonequals (\Lambda_\con)_U$ \cite[2.15]{DW4}. 
    \end{enumerate}
\end{notation}

We will be interested in the NC schemes $(Y, \cA)$ and $(Y, \cA_\con)$. 

\begin{lemma} \label{noetherian}
    The NC schemes $(Y, \cA)$ and $(Y, \cA_\con)$ are noetherian. 
\end{lemma}
\begin{proof}
    By assumption, $Y$ is a noetherian scheme. To see that $\cA$ and $\cA_\con$ are coherent as $\cO_Y$-modules, we may assume that $Y = \spec R$ for a noetherian ring $R$. Then, we may view $f_* \cP$ as an $R$-module which is finitely generated because $\cP$ is coherent and $f$ is proper. 
    
    Since $f_* \cP$ is finitely generated over $R$ and $R$ is noetherian, $\cA = \cEnd_R(f_* \cP) = \Lambda^\sim$ where $\Lambda = \End_R(f_* \cP)$. Thus, $\Lambda = \End_R(f_* \cP)$ is also finitely generated. Similarly, $\cA_\con = \cA / \cI = \Lambda_\con^\sim$ where $\Lambda_\con = \Lambda / [\proj R]$ is finitely generated since $\Lambda$ is. 

    Therefore, we conclude that $\cA$ and $\cA_\con$ are coherent as $\cO_Y$-modules.
\end{proof}

An interesting property of the noncommutative scheme $(Y, \cA_\con)$ is that the contraction theorem \cite[2.16]{DW4} implies that $\cA_\con$ is supported on the singular locus $Z$ of $f$. As a consequence, we often reduce questions about $(Y, \cA_\con)$ to questions about stalks $(\cA_\con)_z$ for $z \in Z$. 

\subsection{The twist}

\begin{prop}
    There are adjoint pairs 
    $(G^\LA, G)$, $(G, G^\RA)$ as described below in \eqref{res-ext diagram sheafy}.
\begin{equation} \label{res-ext diagram sheafy}
    \begin{tikzcd}[column sep=8em]
        \Dcat(\qcoh (Y, \cA_\con) ) \arrow[rr, "G  = - \otimes^{\Lderived{}} {}_{\cA_\con}  \cA_\con \cong \Rderived{\Hom_{\cA_\con}}{(\cA_\con, -)}" description, ""{name=G,above}] & {} & \arrow[ll, "G^{\mathrm{RA}} = \Rderived{\Hom_\cA}{(\cA_\con, -)}"{name=RA}, bend left=12] \arrow[ll, "G^{\mathrm{LA}}  = - \otimes^{\Lderived{}}_\cA \cA_\con"{name=LA}, swap, bend right=12] \Dcat(\qcoh (Y, \cA) )
        \arrow[phantom, from=LA, to=F, "\scriptstyle\boldsymbol{\perp}"description]
        \arrow[phantom, from=F, to=RA, "\scriptstyle\boldsymbol{\perp}"description]
    \end{tikzcd}
\end{equation}
Moreover, if $\cA_\con$ is perfect as a $\cA$-module, then $G^\LA$, $G$ and $G^\RA$ restrict to the bounded derived categories of coherent sheaves.
\end{prop}
\begin{proof}
    This follows from \ref{nc derived functors bg}, \ref{res to coh} and \ref{noetherian}.
\end{proof}

Our goal is to describe the twist around the functor $G$.

\medskip
Recall that by the construction of $\cA$ and $\cA_\con$, there is an exact sequence
\begin{align} \label{sheafy Acon seq}
    0 \to \cI \tow{\iota} \cA \tow{\pi} \cA_\con \to 0 
\end{align}
of $\cA$-bimodules. 

\begin{lemma} \label{perfect I}
    If $\cA_\con$ is perfect as an $\cA$-module, so is $\cI$.
\end{lemma}
\begin{proof}
    We may assume $Y = \spec R$ with $\cA = \Lambda^\sim$, $\cA_\con = \Lambda_\con^\sim$ and $\cI = [\add R]^\sim$. Then, $\cI$ is perfect if and only if it has finite projective dimension. By assumption, $\Lambda_\con$ is quasi-isomorphic to a bounded complex $P$ of projective $\Lambda$-modules and, thus, has finite projective dimension as a $\Lambda$-module.
    
    The equation \eqref{sheafy Acon seq} reduces to 
    \[ 0 \to [\add R] \to \Lambda \to \Lambda_\con \to 0, \]
    so that finite projective dimension of $\Lambda_\con$ implies finite projective dimension of $[\add R]$. 
\end{proof}

\begin{lemma} \label{twist cd sheafy}
There is a commutative diagram
\[
\begin{tikzcd}[column sep =4em]
    \Rderived{\cHom_\cA}(\cA_\con, -) \arrow[r, "{\Rderived{\cHom_\cA}(\pi, -)}"] \arrow[d, "\sim"{anchor=north, rotate=90}] & \Rderived{\cHom_\cA}( \cA, -) \arrow[d, "\sim"{anchor=north, rotate=90}] \\
   \Rderived{\cHom_\cA}(\cA_\con, -) \Dtensor_{\cA_\con} \cA_\con \arrow[r, "\epsilon"] &  1_{\Dcat((Y, \cA))} 
\end{tikzcd}
\]
of endofunctors on $\Dcat(\qcoh (Y, \cA))$. Here, $\epsilon$ is the counit of the adjunction induced by \ref{nc derived functors bg}\eqref{adjunction}.
\end{lemma}
\begin{proof}
    Let $\cM \in \Dcat(\qcoh (Y, \cA))$ and let $j_\cM \colon \cM \to \cJ$ be a $K$-injective resolution of $\cM$. Then, \begin{align*}
        & \Rderived{\cHom_\cA}(\cA_\con, \cM) = \cHom_\cA^*(\cA_\con, \cJ), \\
        & \Rderived{\cHom_\cA}(\cA_\con, -) \Dtensor_{\cA_\con} \cA_\con =  \tot( \cHom_\cA^*(\cA_\con, \cJ) \otimes_{\cA_\con} \cA_\con).
    \end{align*}
    Consider the isomorphisms of chain complexes 
    \begin{align*}
        m & \colon  \tot( \cHom_\cA^*(\cA_\con, \cJ) \otimes_{\cA_\con} \cA_\con) \to \cHom_\cA^*(\cA_\con, \cJ) \\
        \ev_1 & \colon \cHom_\cA^*(\cA, \cJ) \to \cJ
    \end{align*}
    which are natural in $\cJ$. We may view $m$ as the sheafification of the morphism which at every open set $U$ acts as the natural multiplication morphism 
    \[ \tot( \Hom_{\cA |_U}^*(\cA_\con |_U, \cJ |_U) \otimes_{\cA_\con(U)} \cA_\con(U)) \to  \Hom_{\cA |_U}^*(\cA_\con |_U, \cJ |_U) \]
    Moreover, $\ev_1$ is the morphism which at every open set $U$ acts as the natural evaluation morphism at the identity section of $\cA(U)$.
    \[ \Hom_{\cA | U}^*(\cA |_U, \cJ |_U) \to \cJ(U) \]

    Note, moreover, that $\epsilon_\cM = j_\cM \cdot \gamma_\cM$ where $\gamma_\cM$ is the natural evaluation map which defines the tensor-hom adjunction in the homotopy category of chain complexes of $\cA$-modules. Hence, the commutativity of the diagram in the statement of the lemma is equivalent to commutativity for all $\cJ$ of the following diagram.
    \[
\begin{tikzcd}[column sep =4em]
     \cHom_\cA^*(\cA_\con, \cJ) \arrow[r, "{\Rderived{\cHom_\cA}(\pi, -)}"] \arrow[d, "m"', "\sim"{anchor=north, rotate=90}] &  \cHom_\cA^*(\cA, \cJ) \arrow[d, "\ev_1"', "\sim"{anchor=north, rotate=90}] \\
     \tot( \cHom_\cA^*(\cA_\con, \cJ) \otimes_{\cA_\con} \cA_\con) \arrow[d, "\gamma_\cM"'] \arrow[r, "\gamma_\cM", dashed] &  \cJ \arrow[d, "j_\cM"'] \\
     \cJ \arrow[r, "j_\cM^{-1}"] & \cM
\end{tikzcd}
\]
    Whence, it suffices to check commutativity of the top square, which only involves morphisms of chain complexes. So, it is enough to check that the diagram commutes affine locally, which is just \ref{twist cd}.
\end{proof}

\begin{cor} \label{twist t res sheafy}
    The functor
    \[ \cT \colonequals \Rderived{\cHom}_\cA(\cI, -) \colon \Dcat(\qcoh (Y, \cA) ) \to \Dcat(\qcoh (Y, \cA) ) \]
    is the twist around $G$. If, moreover, $\cA_\con$ is perfect as an $\cA$-module, then $\cT$ restricts to the bounded derived category of coherent $\cA$-modules.
\end{cor}
\begin{proof}
    The fact that $\cT$ is the twist around $G$ is just the global analogue of \ref{twist t res}. 
    
    If $\cA$ is perfect, then so is $\cI$ by \ref{perfect I}. Thus,  $\cT$ restricts to the bounded derived category of coherent $\cA$-modules by \ref{res to coh}.
\end{proof}
    
\begin{cor} \label{twist geometric}
    Suppose that $\cA_\con$ is perfect as an $\cA$-module. Then, the functor
    \[ \Psi_\cP^{-1} \cdot \cT \cdot \Psi_\cP \colon \Db(\coh X) \to \Db(\coh X)  \]
    is the twist around $\Psi_\cP^{-1} \cdot G$.
\end{cor}
\begin{proof}
    This is immediate from \ref{spherical composed with equiv} and \ref{twist t res sheafy}.
\end{proof}

\subsection{Equivalence criterion}

 The goal of this subsection is to prove that the twist $\cT$ is an equivalence when the NC scheme $(Y, \cA_\con)$ satisfies certain spherical assumptions. In order to do so, the following technical lemma will be essential.

\begin{lemma} \label{adjunction is local}
    Let $(\cX, \cA)$ and $(\cX, \cB)$ be noncommutative schemes. Let $i \colon U \hookrightarrow \cX$ an open subset. Let $\cF \in \Dcat( \qcoh(\cX, \cB \otimes_{\cO_X} \cA^\op))$ be a complex of bimodules which are locally finitely presented as $\cA$-modules. Consider the adjunctions induced by \ref{nc derived functors bg}\eqref{adjunction}
    \begin{align}
        - \Dtensor_\cA \cF & \dashv \Rderived{\cHom}_{\cB}(\cF, -), \label{adj 1} \\
        - \Dtensor_{i^*\cA} i^*\cF & \dashv \Rderived{\cHom}_{i^*\cB}(i^*\cF, -) \label{adj 2}
    \end{align}
    and let $\epsilon^\cX$ and $\eta^\cX$ denote the counit and unit of the adjunction \eqref{adj 1}, respectively. Similarly, let $\epsilon^U$ and $\eta^U$ denote the counit and unit of the adjunction \eqref{adj 2}. Then, $i^*(\epsilon^\cX_\cF) = \epsilon^U_{i^*\cF}$ and $i^*(\eta^\cX_\cF) = \eta^U_{i^*\cF}$.
\end{lemma}
\begin{proof}
    We prove the statement $i^*(\epsilon^\cX_\cF) = \epsilon^U_{i^*\cF}$, and omit the analogous proof of $i^*(\eta^\cX_\cF) = \eta^U_{i^*\cF}$.

    Write $\cO(\cX)$ for the collection of open sets of $\cX$. Let $\cM, \cN \in \Dcat(\qcoh(\cX, \cB))$, and fix a $K$-injective resolution $j \colon \cM \to \cJ$  of $\cM$. Moreover, let $p \colon \cQ \to \cF$ be a $K$-flat resolution of $\cF$ as a $\cA$-$\cB$-module. From  \ref{nc derived functors bg}\eqref{adjunction}, there is an isomorphism 
    \[ \begin{tikzcd}
         \Rderived{\cHom_\cA}(\cN,\Rderived{\cHom_\cB}(\cF, \cM ) ) \arrow[r, "\sim"'] \arrow[d, equals] & \Rderived{\cHom_\cB}(\cN \Dtensor_\cA \cF, \cM ) \arrow[d, equals] \\
         \cHom^*_\cA(\cN, \cHom^*_\cB(\cQ, \cJ ) ) \arrow[r, "\sim"', "{\Phi^{(\cN, \cM)}}"] & \cHom^*_\cB(\tot(\cN  \Dtensor_\cA \cQ), \cJ ) 
    \end{tikzcd}
    \]
    where $\Phi^{(\cN, \cM)}$ is an isomorphism of complexes of sheaves. Therefore,
    \[ \Phi^{(\cN, \cM)} = (\Phi^{(\cN, \cM)}_V)_{V \in \cO(\cX)} \]
    where each $\Phi^{(\cN, \cM)}_V$ is a map of complexes
    
    \begin{align*}
        \Phi^{(\cN, \cM)}_V \colon \Hom^*_{\cA |_V}(\cN |_V ,\cHom^*_\cB(\cQ |_V, \cJ |_V ) ) \to \Hom^*_{\cB |_V}(\tot(\cN |_V  \Dtensor_{\cA |_V} \cQ |_V), \cJ |_V ),
    \end{align*}
    and these maps are compatible with restriction. That is, the diagram
    \begin{equation} \label{sheaf hom adj}
        \begin{tikzcd}
        \Hom^*_{\cA}(\cN ,\cHom^*_\cB(\cQ, \cJ ) ) \arrow[r, "{\Phi^{(\cN, \cM)}_\cX}"] \arrow[d, "(-) |_U"] & \Hom^*_{\cB}(\tot(\cN  \Dtensor_{\cA} \cQ), \cJ) \arrow[d, "(-) |_U"] \\
        \Hom^*_{\cA |_U}(\cN |_U ,\cHom^*_\cB(\cQ |_U, \cJ |_U ) ) \arrow[r, "{\Phi^{(\cN, \cM)}_U}"] & \Hom^*_{\cB |_U}(\tot(\cN |_U  \Dtensor_{\cA |_U} \cQ |_U), \cJ |_U )
        \end{tikzcd}
    \end{equation}
    commutes. Moreover, 
    \begin{align*}
        \Rderived{\Hom_\cA}(\cF, \cM) |_U & = \Rderived{\Hom_{\cA |_U}}(\cF |_U, \cM  |_U), \tag{by \cite[18.4.6]{KSc}} \\
        (\cN \Dtensor_\cA \cF) |_U & = \cN |_U \Dtensor_{\cA |_U} \cF |_U. \tag{by  \cite[18.5.2]{KSc}}
    \end{align*}
    With these two equations in mind, note that the morphisms $\cohom{0}(\Phi^{(\cN, \cM)}_X)$ and $\cohom{0}(\Phi^{(\cN, \cM)}_U)$ induce the Hom-set isomorphisms defining the adjunctions \eqref{adj 1} and \eqref{adj 2}, respectively. Therefore, when the complex $\cN = \Rderived{\cHom_\cB}(\cF, \cM )$, then $\epsilon^\cX_\cF = \cohom{0}(\Phi^{(\cN, \cM)}_X)(\id)$ and $\epsilon^U_{i^*\cF} = \cohom{0}(\Phi^{(\cN, \cM)}_U)(\id)$.
    
    For $\cN = \Rderived{\cHom_\cB}(\cF, \cM )$, the diagram \eqref{sheaf hom adj} is

    \medskip
    \adjustbox{scale=0.95,center}{%
        \begin{tikzcd}
        \Hom^*_{\cA}(\cHom^*_\cB(\cQ, \cJ) ,\cHom^*_\cB(\cQ, \cJ ) ) \arrow[r, "\Phi^{(\cN, \cM)}_\cX"] \arrow[d, "(-) |_U"] & \Hom^*_{\cB}(\tot(\cHom^*_\cB(\cQ, \cJ)  \Dtensor_{\cA} \cQ), \cJ) \arrow[d, "(-) |_U"] \\
        \Hom^*_{\cA |_U}(\cHom^*_{\cB |_U}(\cQ |_U, \cJ |_U) ,\cHom^*_{\cB |_U}(\cQ |_U, \cJ |_U ) ) \arrow[r, "\Phi^{(\cN, \cM)}_U"] & \Hom^*_{\cB |_U}(\tot(\cHom^*_{\cB |_U}(\cQ |_U, \cJ |_U)  \Dtensor_{\cA |_U} \cQ |_U), \cJ |_U )
        \end{tikzcd}}
    
    \medskip
    Thus, 
    \[ \epsilon^\cX_\cF |_U =  \cohom{0}(\Phi^{(\cN, \cM)}_X)(\id) |_U = \cohom{0}(\Phi^{(\cN, \cM)}_U)(\id |_U) = \cohom{0}(\Phi^{(\cN, \cM)}_U)(\id) = \epsilon^U_{i^*\cF} \]
    as required.
\end{proof}

\begin{definition}
    Similarly to \cite[5.6]{DW4}, we will say that $\cA_\con$ is $t$-relatively spherical if for all closed points $z \in Z$,
    \[ \Ext^i_{\widehat{\cA}_z}(\widehat{\cA}_{\con, z}, \cS) = 0 \]
    for all $i \neq 0,t$. Here, $\cS$ is the direct sum of all simple $\widehat{\cA}_{\con,z}$ modules. 
\end{definition}

\begin{theorem} \label{global equiv theorem}
   If $\cA_\con$ is perfect in $\Db(\coh(Y, \cA))$ and $t$-relatively spherical, then 
    \[ \Rderived{\cHom_\cA}(\cI, -) \colon \Db(\modu \cA) \to \Db(\modu \cA) \]
    is an equivalence. Consequently,
       \begin{align}
       \Phi_X = \Psi_{\cP}^{-1} \cdot \Rderived{\Hom}_\cA(\cI, -) \cdot \Psi_{\cP} \colon \Db(\coh X) \tow{\sim} \Db(\coh X) 
    \end{align}
    is an autoequivalence of $\Db(\coh X)$.
\end{theorem}
\begin{proof}
    First, we show that $(Y, \cA)$ admits an affine open cover $\{U_i = \spec R_i \to Y \}$ such that the functor $\Rderived{\cHom_\cA}(\cI, -) |_{U_i}$ is an equivalence. Well, recall that by \cref{zariski local equiv}, the functor $\Rderived{\cHom_\cA}(\cI, -) |_{U_i} \cong \Rderived{\Hom}_{\Lambda_{U_i} }([\add R_i], -)$ is an equivalence if $(\Lambda_\con)_{U_i}$ is perfect as a $\Lambda_{U_i}$-module and $t$-relatively spherical. 

    Well, since $\cA_\con$ is assumed to be perfect in $\Db(\coh(Y, \cA))$, there exists an open affine cover $\{ j_i \colon U_i = \spec R_i \to Y \}$ such that $\cA_\con |_{U_i}$ is perfect as a $\cA |_{U_i}$-module for all $i$. Equivalently, since each $U_i$ is affine, $(\Lambda_\con)_{U_i}$ is perfect as a $\Lambda_{U_i}$-module. 
        
    Moreover, since the definition of $t$-relatively spherical is stalk-local, the condition that $\cA_\con$ is $t$-relatively spherical implies $\cA_\con |_{U_i}$ is $t$-relatively spherical. That is, for $z \in U_i$, 
    \[ ( \cA_\con|_{U_i} )_z = ((\widehat{\Lambda}_\con)_{U_i})_z. \]
    Hence, $(\Lambda_\con)_{U_i}$ is $t$-relatively spherical by \ref{can check on stalks}. Thus,  by the proof of \ref{zariski local equiv}, we may conclude that the functor $\Rderived{\cHom_\cA}(\cI, -) |_{U_i} \cong \Rderived{\Hom}_{\Lambda_{U_i} }([\add R_i], -)$ is an equivalence for all $U_i$ in the affine open cover.
    
    We next show that this implies that the functor $\Rderived{\cHom_\cA}(\cI, -)$ is an equivalence. Notice that it is an equivalence if and only if the unit $\eta \colon 1_\cA \to \Rderived{\cHom}_\cA(\cI, - \Dtensor_\cA \cI)$ and the counit $\epsilon \colon \Rderived{\cHom}_\cA(\cI, -)  \Dtensor_\cA \cI \to 1_\cA$ are an isomorphisms. Since the proofs that $\eta$ and $\epsilon$ are isomorphisms are very similar, we prove the statement for $\eta$ and omit the proof for $\epsilon$. 
    
    Observe that $\eta$ is an isomorphism if and only if $\eta_a$ is an isomorphism for all $a \in \Db(\modu \cA)$. To show that this latter condition holds, consider the triangle
    \[ a \tow{\eta_a} \Rderived{\cHom}_\cA(\cI, a \Dtensor_\cA \cI) \to K_a \to^+ \]
    in $\Db(\coh (Y, \cA) )$. We will show that $K_a = 0$.
    
    Since $U_i$ is affine, the inclusion $j_i \colon U_i \to Y$ is flat so that $j_i^*$ is an exact functor. We thus have the following triangle 
    \[ j_i^*a \tow{j_i^* \eta_a} j_i^*\Rderived{\cHom}_\cA(\cI, a \Dtensor_\cA \cI) \to j_i^* (K_a) \to^+ \]
     Now, by \ref{adjunction is local}, this triangle is the same as 
    \[ j_i^*a \tow{\eta_{j_i^*a}} \Rderived{\cHom}_\cA(j_i^*\cI, j_i^*a \Dtensor_\cA j_i^*(\cI) ) \to j_i^* (K_a) \to^+ \]
    Where $\eta_{j_i^*a}$ is the unit of the Zariski local adjunction 
    \begin{align*}
        & - \Dtensor_\cA j_i^*\cI \dashv \Rderived{\cHom}_{\cA |_{U_i}}(j_i^*\cI, -).
    \end{align*}
   However, the functor $\Rderived{\cHom}_{\cA |_U}(j^*\cI, -) = \Rderived{\cHom}_{\cA |_{U_i}}([\add R_i], -)$ is an equivalence  by the above. It follows that $j_i^*(K_a) = 0$. Since $K_a$ is zero on every $U_i$ and $\{U_i\}_i$ is an open cover of $Y$, necessarily $K_a = 0$. Thus, $\eta$ is an isomorphism.
\end{proof}

\appendix

\section{Chasing the evaluation map} \label{dth adj}

Let $A, B$ be rings such that there is a ring morphism $p \colon A \to B$. In this section, we compute the counit of the derived tensor-Hom adjunction by diagram chasing the identity along the isomorphism
\[  \scriptstyle
\begin{tikzcd}[column sep=1.5em]
    \Hom_{\Dcat(B)}(\Rderived{ \Hom_{A} }(B, -), \Rderived{ \Hom_{A} }(B, -) ) \arrow[r, "\sim"] &
    \Hom_{\Dcat(A )}(\Rderived{ \Hom_{A} }(B, -) \Dtensor_{B}  B, -) 
\end{tikzcd}
\]
which defines the derived tensor-hom adjunction. 

\begin{lemma} \label{hom complex and total tensor adj}
    Let $L \in \Kcat(B)$, $M \in \Kcat(A \otimes_\Z B^\op)$ and $N \in \Kcat(A)$.
    There is a canonical isomorphism
    \[ \phi \colon \Hom_B^*(L, \Hom^*_A(M, N) ) \tow{\sim} \Hom^*_A( \tot(L \otimes_B M), N). \]
\end{lemma}
\begin{proof}
    Let $\alpha$ be a degree $j$ element in $\Hom_B^*(L, \Hom^*_A(M, N) )$. Then, 
    \[ \alpha \in \prod_{p+q = j}  \Hom_B(L^{-p}, \Hom^q_A(M, N) ) = \prod_{p+r+s = j} \Hom_B(L^{-p}, \Hom_A(M^{-r}, N^s) ). \]
    Hence, we may write $\alpha = (\alpha^{p, r, s})$. By the standard tensor-hom adjunction there are isomorphisms
    \[ \phi^{p, r, s} \colon \Hom_B(L^{-p}, \Hom_A(M^{-r}, N^s)) \tow{\sim} \Hom_A(L^{-p} \otimes_B M^{-r}, N^s). \]
    Hence, there is an isomorphism $ \phi^j = \prod_{p+r+s = j}  \phi^{p, r, s}$:
    \begin{align*}
        \Hom_B^*(L, \Hom^*_A(M, N) )^j = & \prod_{p+r+s = j} \Hom_B(L^{-p}, \Hom_A(M^{-r}, N^s) ) \\ 
        \tow[\phi^{j}]{\sim} & \prod_{p+r+s = j} \Hom_A(L^{-p} \otimes_B M^{-r}, N^s) \\
         = & \prod_{k+s = j} \Hom_A( \bigoplus_{p+r = k} L^{-p} \otimes_B M^{-r}, N^s) \\
         = & \Hom_A^*( \tot(L \otimes_B M), N)^j 
    \end{align*}
    It is straightforward to check that $\phi = (\phi^j)$ respects the differentials so that it is indeed a chain complex isomorphism. 
\end{proof}

\begin{lemma} \label{cohom of hom complex}
    Let $L, M \in \Kcat(B)$, then there is an isomorphism
    \[ \cohom{j}(\Hom^*(L, M) ) \tow{f} \Hom_{\Kcat(B)}(L, M[j]). \]
\end{lemma}
\begin{proof}
    This isomorphism is well known. Explicitly, given a cocycle $\beta = (\beta^{p, q})$ in the product $\prod_{p+q = j} \Hom_R(L^{-p}, M^q)$, the isomorphism  maps the equivalence class $[\beta]$ to the chain morphism $L \to M[j]$ in $\Kcat(B)$ defined by the data $(\beta^{p, q})$.
\end{proof}

\begin{lemma} \label{hom is k injective}
        Let $I \in \Kcat(A)$ be a $K$-injective complex, and let $P$ be a complex of $A$-$B$-bimodules which is $K$-flat as a complex of $B^\op$-modules. Then, $\Hom_A^*(P, I)$ is a $K$-injective complex in $\Kcat(B)$. 
\end{lemma}
\begin{proof}
    Let $M$ be an acyclic complex. Then, 
    \begin{align*}
        \Hom_{\Kcat(B)}(M, \Hom_A^*(P, I)) & \cong \cohom{0}(\Hom^*_B(M, \Hom_A^*(P, I)) \tag{\text{by \ref{cohom of hom complex} }} \\
        & \cong \cohom{0}(\Hom^*_A(\tot(M \otimes P), I) ) \tag{\text{by \ref{hom complex and total tensor adj}}} \\
        & \cong  \Hom_{\Kcat(A)}(\tot(M \otimes P), I). \tag{\text{by \ref{cohom of hom complex}}}
    \end{align*}
    Now, since $M$ is acyclic and $P$ is $K$-flat as a $B^\op$-module, it must be that $\tot(M \otimes P)$ is acyclic. Since $I$ is $K$-injective, then $\Hom_{K(A)}(\tot(M \otimes P), I) = 0$ by definition. Hence, by the definition of $K$-injective, $\Hom_A^*(P, I)$ is a $K$-injective complex. 
\end{proof}

\begin{lemma} \label{Kcat and Dcat hom iso}
    Let $L, J \in \Dcat(B)$ such that $J$ is $K-$injective. Then, there is an isomorphism 
    \[\Hom_{\Kcat(B)}(L, J) \tow{g} \Hom_{\Dcat(B)}(L, J). \]
\end{lemma}
\begin{proof}
    The isomorphism $g \colon \Hom_{\Kcat(R)}(L, J) \to \Hom_{\Dcat(R)}(L, J)$ is well-known. It maps a morphism $L \tow{\alpha} J$ in $\Hom_{\Kcat(R)}(L, J)$ to the roof $L = L \tow{\alpha} J$ in $\Hom_{\Dcat(R)}(L, J)$.
\end{proof}

\begin{prop} \label{derived adjunction}
    Let $L \in \Kcat(B)$, $M \in \Kcat(A \otimes_\Z B^\op)$ and $N \in \Kcat(A)$. Then, there is an isomorphism
    \[ \phi^0_{L,M,N} \colon \Hom_{\Dcat(B)}(L, \Rderived{\Hom_A}(M, N) ) \tow{\sim} \Hom_{\Dcat(A)}(L \Dtensor_B M, N) \]
    which is natural in $L,M,N$.
\end{prop}
\begin{proof}
        Let $q_M \colon P \to M$ be a bimodule resolution of $M$ which is $K$-flat as a complex of $B^\op$-modules. Then, the statement follows from the diagram
        \[
        \begin{tikzcd}[column sep=7em]
            \Hom_{\Dcat(B)}(L, \Rderived{\Hom_A}(M, N) ) \arrow[d, equals] \arrow[r, "\psi^0_{L, M, N}"] & \Hom_{\Dcat(A)}(L \Dtensor_B M, I) \arrow[d, equals] \\
            \Hom_{\Dcat(B)}(L, \Hom_A^*(M, I) ) \arrow[d, "\sim"'{anchor=south, rotate=90}, "{\Hom_{\Dcat(B)}(L, \Hom_A^*(q_M, I) )}"] & \Hom_{\Dcat(A)}(\tot(L \otimes P), I) \arrow[d, equals] \\
            \Hom_{\Dcat(B)}(L, \Hom_A^*(P, I) ) \arrow[d, "\sim"'{anchor=south, rotate=90}, "g", "{\text{by \ref{Kcat and Dcat hom iso} and \ref{hom is k injective}}}"'{xshift=-1em}] & \Hom_{\Dcat(A)}(\tot(L \otimes P), I) \arrow[d, "\sim"'{anchor=north, rotate=90}, "g", "{\text{by \ref{Kcat and Dcat hom iso}}}"'{xshift=1em}, swap] \\
            \Hom_{\Kcat(B)}(L, \Hom_A^*(P, I) ) \arrow[d, "\sim"'{anchor=south, rotate=90}, "f^{-1}", "{\text{by \ref{cohom of hom complex}}}"'{xshift=-1em}] & \Hom_{\Kcat(A)}(\tot(L \otimes P), I) \arrow[d, "\sim"'{anchor=north, rotate=90}, "f^{-1}", "{\text{by \ref{cohom of hom complex}}}"'{xshift=1em}, swap]\\
            \cohom{0}( \Hom^*_{B}(L, \Hom_A^*(P, I) ) \arrow[r, "\sim"', "\cohom{0}(\phi)", "{\text{by \ref{hom complex and total tensor adj}}}"'{xshift=2em}] &  \cohom{0}( \Hom^*_{A}(\tot(L \otimes P), I )
        \end{tikzcd}
        \]
        where $\psi^0_{L,M,N}$ is defined so that the diagram commutes.
\end{proof}
 
As in \ref{derived adjunction}, let  $M \in \Kcat(A \otimes_\Z B^\op)$ and choose a bimodule resolution $P \to M$ of $M$ that is $K$-flat as a $B^\op$-module. Then, there is a quasi-isomorphism 
\[ s \colon \Rderived{ \Hom_{A} }(M, N) \Dtensor_{B}  M \to  \Hom_A^*(P, I) \Dtensor_{B}  M.\] 

\begin{lemma} \label{standard counit gen}
     The counit $\epsilon$ of the adjunction $(-) \Dtensor_B M \dashv \Rderived{\Hom_A}(M, -)$ has component at $N \in \Dcat(A)$
       \[ \epsilon_N \colon \Rderived{ \Hom_{A} }(M, N) \Dtensor_{B}  M \to N \]
      given as the composition $j_N^{-1} \cdot \gamma_N \cdot s$. Here, $j_N \colon N \to I$ is a $K$-injective resolution of $N$ and $\gamma_N = (\gamma_N^j)$ is a chain map 
      \[ \gamma_N^j \colon \bigoplus_{p+q=j} ( \prod_{r+s=p} \Hom_A(P^{-r}, I^s) )\otimes_B P^{q} \to I^j \]
      sending $\alpha \otimes a \in \Hom_A(P^{-r}, I^s) \otimes P^{q}$ to zero if $-r \neq q$ and sending an element $\alpha \otimes a$ in $\Hom_A(P^{q}, I^j) \otimes P^{q}$ to $\alpha(a) \in I^j$.
\end{lemma}
\begin{proof}
     The component $\epsilon_N$ of the counit is the image of the identity under the isomorphism $\phi^0_{\Rderived{ \Hom_{A} }(M, N), M, N}$ defined in \ref{derived adjunction}. That is, $\epsilon_N = j^{-1}_N \cdot \psi^0_{\Rderived{ \Hom_{A} }(M, N), M, N}(\id)$. Moreover, observe that since $\psi^0_{L, M, N}$ is natural in $L$, the following diagram commutes.
   \[\begin{tikzcd}[column sep=6.5em]
            \Hom_{\Dcat(B)}(\Hom_A^*(P, I), \Rderived{\Hom_A}(M, N) ) \arrow[d, "{\Hom_{\Dcat(B)}(\Hom_A^*(q_M, I), \Rderived{\Hom_A}(M, N) )}"] \arrow[r, "\psi^0_{\Hom_A^*(P, I), M, N}"] & \Hom_{\Dcat(A)}(\Hom_A^*(P, I) \Dtensor_B M, I) \arrow[d, "{\Hom_{\Dcat(A)}(s, I)}"] \\
            \Hom_{\Dcat(B)}(\Rderived{\Hom_A}(M, N), \Rderived{\Hom_A}(M, N) )  \arrow[r, "{\psi^0_{\Rderived{\Hom_A}(M, N), M, N}}"] & \Hom_{\Dcat(A)}(\Rderived{\Hom_A}(M, N) \Dtensor_B M, I) 
        \end{tikzcd}
        \]
        where the vertical arrows arrows are quasi-isomorphisms. Hence, $\psi^0_{\Rderived{ \Hom_{A} }(M, N), M, N}(\id)$ equals
       \begin{align*}
          & \Hom_{\Dcat(A)}(s, I) \cdot \psi^0_{\Hom_A^*(P, I), M, N} \cdot \Hom_{\Dcat(B)}(\Hom_A^*(q_M, I)^{-1}, \Rderived{\Hom_A}(M, N))(\id) \\
           & = \Hom_{\Dcat(A)}(s, I) \cdot \psi^0_{\Hom_A^*(P, I), M, N}(\Hom_A^*(q_M, I)^{-1}) \\
           & = \psi^0_{\Hom_A^*(P, I), M, N}(\Hom_A^*(q_M, I)^{-1}) \cdot s.
       \end{align*} 
        That is, $\epsilon_N = j^{-1}_N \cdot \psi^0_{\Hom_A^*(P, I), M, N}(\Hom_A^*(q_M, I)^{-1}) \cdot s$. Hence it suffices to show that 
        \[\psi^0_{\Hom_A^*(P, I), M, N}(\Hom_A^*(q_M, I)^{-1}) = \gamma_N. \]
        
         Consider the diagram
          \[
        \begin{tikzcd}[column sep=7em]
            \Hom_{\Dcat(B)}(\Hom_A^*(P, I), \Rderived{\Hom_A}(M, N) ) \arrow[d, equals] \arrow[r, "\psi^0_{\Hom_A^*(P, I), M, N}"] & \Hom_{\Dcat(A)}(\Hom_A^*(P, I) \Dtensor_B M, I) \arrow[d, equals] \\
            \Hom_{\Dcat(B)}(\Hom_A^*(P, I), \Hom_A^*(M, I) ) \arrow[d, "\sim"'{anchor=south, rotate=90}, "{\Hom_{\Dcat(B)}(\Hom_A^*(P, I), \Hom_A^*(q_M, I) )}"] & \Hom_{\Dcat(A)}(\tot(\Hom_A^*(P, I) \otimes P), I) \arrow[d, equals] \\
            \Hom_{\Dcat(B)}(\Hom_A^*(P, I), \Hom_A^*(P, I) ) \arrow[d, "\sim"'{anchor=south, rotate=90}, "g", "{\text{by \ref{Kcat and Dcat hom iso} and \ref{hom is k injective}}}"'{xshift=-1em}] & \Hom_{\Dcat(A)}(\tot(\Hom_A^*(P, I) \otimes P), I) \arrow[d, "\sim"'{anchor=north, rotate=90}, "g", "{\text{by \ref{Kcat and Dcat hom iso}}}"'{xshift=1em}, swap] \\
            \Hom_{\Kcat(B)}(\Hom_A^*(P, I), \Hom_A^*(P, I) ) \arrow[d, "\sim"'{anchor=south, rotate=90}, "f^{-1}", "{\text{by \ref{cohom of hom complex}}}"'{xshift=-1em}] & \Hom_{\Kcat(A)}(\tot(\Hom_A^*(P, I) \otimes P), I) \arrow[d, "\sim"'{anchor=north, rotate=90}, "f^{-1}", "{\text{by \ref{cohom of hom complex}}}"'{xshift=1em}, swap]\\
            \cohom{0}( \Hom^*_{B}(\Hom_A^*(P, I), \Hom_A^*(P, I) ) \arrow[r, "\sim"', "\cohom{0}(\phi)", "{\text{by \ref{hom complex and total tensor adj}}}"'{xshift=2em}] &  \cohom{0}( \Hom^*_{A}(\tot(\Hom_A^*(P, I) \otimes P), I )
        \end{tikzcd}
        \]
        Then, it is clear that 
        \[ \Hom_{\Dcat(B)}(\Hom_A^*(M, I), \Hom_A^*(q_M, I) )(\Hom_A^*(q_M, I)^{-1}) = \id. \]
        So that $f^{-1} \cdot g(\id) = [\beta]$, where 
        \[ \beta = (\beta^j) \in \Hom_B^*(\Hom^*_A(P, I), \Hom^*_A(P, I))^0 = \prod_j \Hom_B(\Hom^*_A(P, I)^j, \Hom^*_A(P, I)^j) \]
        is the product over $j$ of identity maps $\beta_j \colon \Hom^*_A(P, I)^j \to \Hom^*_A(P, I)^j$. That is, 
         \[ \beta^j \colon \prod_{r + s = j} \Hom_A(P^{-r}, I^s) \to \prod_{k+n =j} \Hom_A(P^{-k}, I^n) \]
         is the product of identity maps $\beta^j = (\beta^{r, s} \colon \Hom_A(P^{-r}, I^s) \to \Hom_A(P^{-r}, I^s))$. We may thus rewrite $\beta = (\beta^{j, r, s})$.
         
        It follows from \ref{hom complex and total tensor adj} that
        \[ \cohom{0}(\phi)([\beta]) = [\phi(\beta)] = [\phi( (\beta^{j, r, s}))] = [(\phi^{-j, r, s}( \beta^{j, r, s})] \]
        where 
        \[ \phi^{-j, r, s}( \beta^{j, r, s}) \in \Hom_A(\Hom_A^*(P^{-r}, I^s) \otimes P^{-r}, I^s) \]
        is the map sending $\alpha \otimes a \mapsto \beta^{j, r, s}(\alpha)(a) = \alpha(a)$. Thus, 
        \[  \phi(\beta)^j \in \Hom_A( \bigoplus_{q + n = j} (\prod_{r + s = q} \Hom_A(P^{-r}, I^s)) \otimes P^q, I^j) \]
        sends $\alpha \otimes a \in \Hom_A(P^{-r}, I^s) \otimes P^{q}$ to zero if $-r \neq q$ and sends $\alpha \otimes a \in \Hom_A(P^{q}, I^j) \otimes P^{q}$ to $\alpha(a) \in I^j$. 
        
        Mapping $\cohom{0}( \phi([\beta]) )$ across $f$, we find that it corresponds to the chain map $\gamma_N$. For $g^{-1}(\gamma_N)$, we just view this map in the derived category.

        In summary,  the counit at $N$ is 
        \[ \epsilon_N = j^{-1}_N \cdot \psi^0_{\Hom_A^*(P, I), M, N}(\Hom_A^*(q_M, I)^{-1}) \cdot s\]
        where $\psi^0_{\Hom_A^*(P, I), M, N}(\Hom_A^*(q_M, I)^{-1})$ is represented in the derived category by the chain map $\gamma_N$, which concludes the proof.
\end{proof}

\begin{cor} \label{standard counit B}
     The counit $\epsilon$ of the adjunction $(-) \Dtensor_B B \dashv \Rderived{\Hom_A}(B, -)$ has component at $a \in \Dcat(A)$
       \[ \epsilon_a \colon \Rderived{ \Hom_{A} }(B, a) \Dtensor_{B}  B \to a \]
      given as the composition $j_a^{-1} \cdot \gamma_a$, where $j_a \colon a \to I$ is a $K$-injective resolution of $a$ and $\gamma_a$ is the chain map 
     \[ \gamma_a = (\gamma^j_a) \colon \tot( \Hom_{A}^*( B , I)  \otimes_{B}  B ) \to I \]
     which at degree $j$ is specified by the evaluation map
     \[ \gamma^j_a \colon \Hom_{A}(B, I^j) \otimes_{B} B \to I^j \]
     sending $f \otimes x \mapsto f(x)$.
\end{cor}
\begin{proof}
    This follows from \ref{standard counit gen} by noting that $B$ is flat as  $B^\op$-module, hence we need not take a $K$-flat resolution of $B$.
\end{proof}

  \section{Semi-perfect and self-injective algebras} \label{app: self inj}

    This appendix recalls some classic results about semi-perfect and self-injective algebras. 

    \subsection{Semi-perfect algebras}
    
    \begin{definition}[{\cite[4.1]{HK14}}] \label{semi perf def}
        Recall that an algebra $A$ is \textit{semi-perfect} if any of the following equivalent conditions hold. 
        \begin{enumerate}
            \item The category of finitely generated projective $A$-modules is Krull-Schmidt.
            \item As a module over itself, $A$ admits a direct sum decomposition $A = \bigoplus_{i = 1}^n P_i$ where $P_i$ have local endomorphism rings.
            \item Every simple $A$-module admits a projective cover.
            \item Every finitely generated $A$-module admits a projective cover.
        \end{enumerate}
    \end{definition}

    Let $A$ be a semi-perfect algebra and consider its direct sum decomposition $A = \bigoplus_{i \in I} P_i$ as in the definition above. Since $P_i$ have local endomorphism rings, they are indecomposable projective $A$-modules. Moreover, note that $\{ P_i\}_{i \in I} $ are all of the indecomposable projective $A$-modules up to isomorphism. 

    \begin{lemma} \label{semi perfect gives finitely many simples}
    If $A$ is a semi-perfect algebra, then it has finitely many simple modules.
    \end{lemma}
    \begin{proof}
        It  follows from the definition of semi-perfect algebras and from \cite[3.6]{HK14} that given a simple $A$-module $S$, there is an $i$ such that $\pi \colon P_i \to S$ is the projective cover of $S$. We claim that there is a bijection between isomorphism classes of indecomposable projectives and isomorphism classes of simples. 
    
        Suppose that $S$ and $S'$ are simple modules which admit projective covers $\pi \colon P_i \to S$ and $\pi' \colon P_i \to S'$. Then, $S \cong P_i / \ker \pi$ and $S' \cong P_i / \ker \pi'$ so that $\ker \pi$ and $\ker \pi'$ are maximal submodules of $P_i$. Therefore, either $\ker \pi = \ker \pi'$ or $\ker \pi + \ker \pi' = P_i$. In the first case, it follows that $S \cong S'$. In the latter case, since both $\ker \pi$ and $\ker \pi'$ are essential by definition of projective covers, it must be that $\ker \pi = \ker \pi' = 0$ so that $S \cong S'$.  
    \end{proof}

    \begin{notation}
        We let $S_i$ be the simple $A$-module whose projective cover is $P_i$. When $A$ is finite-dimensional, let $I_i$ be the indecomposable injective whose socle is $S_i$.
    \end{notation}
    
   \subsection{Self-injective algebras}
    
    An algebra $A$ ifs said to be \textit{self-injective} if it is injective as a module over itself. When $A$ is Noetherian, there is a well-known characterisation of self-injective algebras. 

    \begin{prop}[{See e.g. \cite[4.2.4]{wei}}]
        The following are equivalent.
        \begin{enumerate}
            \item $A$ is Noetherian and self-injective,
            \item All projective $A$-modules are injective,
            \item All injective $A$-modules are projective.
        \end{enumerate}
    \end{prop}
    
    Further, When $A$ is finite dimensional then self-injectivity is equivalent to the property that finitely generated indecomposable projective modules coincide with finitely generated indecomposable injectives \cite[IV 3.7]{FroAlg}.
    
    Consider a finite-dimensional algebra $A$ viewed as the path-algebra of a quiver $Q$ with relations $I$. Then, self-injectivity of $A$ can be thought of as a symmetry condition on $(Q, R)$, since we want the class of indecomposable injective modules to match the class of indecomposable projective modules.

    \begin{example} \label{self inj examples} Let $Q_1$ be the quiver
        \[
        \begin{tikzcd}
            Q_1: & {}& 0 \arrow[dr, "x_0"] & {} \\
            {} & 2 \arrow[ur, "x_2"] & {} & 1 \arrow[ll, "x_1"]
        \end{tikzcd}
        \]
        Consider $Q_1$ with relations $R_1 = \langle x_0 x_1 x_2, x_1 x_2 x_0, x_2 x_0 x_1 \rangle$. Then, the indecomposable projective and injective representations are, up to isomorphism,
        
        \[ \begin{tikzcd}[sep=small]
            I_2 \cong P_0: &[-2em] {}& k \arrow[dr, "1"] & {} \\
            {} & k \arrow[ur, "0"] & {} & k \arrow[ll, "1"]
        \end{tikzcd}
        \begin{tikzcd}[sep=small]
           I_0 \cong P_1: &[-2em] {} & k \arrow[dr, "0"] & {} \\
           {} &  k \arrow[ur, "1"] & {} & k \arrow[ll, "1"]
        \end{tikzcd}
        \begin{tikzcd}[sep=small]
            I_1 \cong P_2: &[-2em] {}& k \arrow[dr, "1"] & {} \\
            {} & k \arrow[ur, "1"] & {} & k. \arrow[ll, "0"]
        \end{tikzcd}
        \]
        Thus, the path algebra of $Q_1$ subject to relations $R_1$ is self-injective. 
    \end{example}
    
    \begin{definition} \label{Nakayama perm}
        Let $A$ be a finite-dimensional self-injective algebra. Then, $P_i$ is also an indecomposable injective $A$-module. It follows that its socle is a simple module which, by \ref{semi perfect gives finitely many simples}, must be $S_j$ for some $j \in I$. The permutation of $I$ defined by $\sigma \colon i \mapsto j$ is called the \textit{Nakayama permutation} of $A$.
    \end{definition}
        
    \begin{example}
    Consider the example \ref{self inj examples}. Then, since $P_0 \cong I_2$, the socle of $P_0$ is $S_2$. That is, $\sigma(0) = 2$. Similarly, $\sigma(1) = 0$ and $\sigma(2) = 1$. 
    \end{example}
    
    When $A$ is is self-injective and finite dimensional, then the category $\modu A$ of finitely generated modules over $A$ is equipped with an automorphism 
    \[ \cN \colonequals - \otimes_A D A \colon \modu A \to \modu A \]
    called the \textit{Nakayama functor}. Here, the dual $D A = \Hom_k(A, k)$ is equipped with its usual bimodule structure: for $f \in D(A)$ and $x \in A$, $(a \cdot f \cdot b) (x) = f(bxa)$. The Nakayama functor has the quasi-inverse \cite[3.3]{IS}
    \[ \cN^{-1} \colonequals - \otimes_A \Hom_A(DA, A) \colon \modu A \to \modu A. \]
    Since $A$ is self-injective, then injective modules are projective modules. It thus follows that $\cN$ and $\cN^{-1}$ are exact and extend to automorphisms of $\Db(\modu A)$.     

\section{Frobenius exact categories} 

A \textit{Frobenius exact category} $\cE$ is an exact category which has enough projective and injective objects and is such that $\proj \cE = \inj \cE$. For more details, see e.g.\ \cite[I.2]{Happel}. This appendix recalls some well-known results about Frobenius categories. 

 Let $\cA$ be an additive category. An \textit{ideal} $\cI$ of $\cA$ consists of subgroups $\cI(a, b) \subseteq \Hom_\cA(a, b)$ such that for any morphism $f \colon a \to b \in \cI(a, b)$ and any pair of morphisms $\alpha \colon a' \to a$ and $\beta \colon b \to b'$ in $\cA$, then $\beta \cdot f \cdot \alpha \in \cI(a', b')$.

Given an Frobenius exact category $\cE$ and objects $a, b \in \cE$, let $[\proj \cE]$ be the ideal of $\cE$ defined by the subgroups $[\proj \cE]_{a,b} \subseteq \Hom_\cE(a, b)$ consisting of all maps which factor through a projective object.  The \textit{stable category} of $\cE$, denoted $\cD$, is the category whose objects are the same objects as $\cE$ and whose morphisms between objects $a, b \in \uE$ are specified by 
    \[ \Hom_{\uE}(a, b) = \uHom_\cE(a, b) \colonequals \Hom_\cE(a, b)/[\proj \cE]_{a,b}. \]
It turns out that $\cD$ carries a triangulated structure, with suspension functor $\Sigma$. 

    \begin{lemma}[{\cite[2.2]{DUG}}] \label{dugas lemma}
        For any map $f \colon x \to y \in \cE$, there exists an object $z \in \cE$ and a morphism $p \colon q \to y \in \cE$ where $q \in \proj \cE$ such that 
        \[ 0 \to z \tow{g} x \oplus q \tow{(f, p)} y \to 0 \]
        is an admissible exact sequence in $\cE$ that induces a triangle
       \[ z \tow{\underline{g}} x \tow{\underline{f}} y \to^+ \]
       in $\cD$.
    \end{lemma}

    This lemma has the following useful corollary. 
    
    \begin{cor} \label{lift of stable iso}
        Suppose that $x, y \in \cD$ are isomorphic. Then, there exists $p_0, p_1 \in \proj \cE$ such that $p_0 \oplus x \cong p_1 \oplus y$ in $\cE$. 
    \end{cor} 

      \begin{lemma} [{\cite[I.2.3,I.2.7]{Happel}, see also \ref{dugas lemma}}] \label{triangle-confl corresp}
    Let $\cE$ be a Frobenius exact category. Then, an admissible sequence 
        \[ 0 \to a \tow{u} b \tow{v} c \to 0\]
        in $\cE$ gives rise to a triangle 
        \[ a \tow{\underline{u}} b \tow{\underline{v}} c \to^+ \]
    in $\cD$. Moreover, every triangle in $\cD$
        \[ x \tow{\underline{u}} y \tow{\underline{v}}  z \to^+ \]
        lifts to an admissible sequence in $\cE$
        \[0 \to x \tow{u'} y \oplus p \tow{(u, \pi)} z \to 0 \]
       where $p \in \proj \cE$ and $u'$ is equivalent to $u$ in $\cD$.
    \end{lemma}

    \begin{lemma} \label{frob-stable ext corresp}
        Let $x, y \in \cE$. 
        \begin{enumerate}
            \item \label{if} If $\Ext^1_\cE(x, y) = 0$, then 
        \[ \Ext^1_{\cD}(x, y) \colonequals \Hom_{\cD}(x, \Sigma y) = 0. \] 
        \item Given an admissible sequence
        \[ 0 \to y \tow{u} z \tow{v} x \to 0 \]
        in $\cE$ and an object $a \in \cE$ such that $\Hom_{\uE}(a, v)$ is epi, then $\Hom_{\cD}(a, v)$ is epi. \label{only if}
        \end{enumerate}
    \end{lemma}
    \begin{proof}
        To prove \eqref{if}, suppose that $\Ext^1_\cE(x, y) = 0$ and let $\underline{w} \colon x \to \Omega^{-1} y$ be a morphism in $\uE$. Then, there is a triangle
        \[ y \tow{\underline{u}} z \tow{\underline{v}} x \tow{\underline{w}} \Omega^{-1} y \]
        in $\uE$ which by \ref{triangle-confl corresp} lifts to an admissible exact sequence
        \[ 0 \to y \tow{u'} p \oplus z \tow{(f, v)} x \to 0 \]
        with $\underline{u'} = \underline{u}$. Since $\Ext^1_\cE(x, y) = 0$, this sequence splits. In particular, $\underline{u'} = \underline{u}$ splits in $\uE$. Whence, the triangle defined by the morphism $\underline{w} \colon x \to \Omega^{-1} y$ splits, forcing $\underline{w} \colon x \to \Omega^{-1} y$ to be zero. We may thus conclude that $\Ext^1_{\uE}(x, y) = 0$.

        For statement \eqref{only if}, let $a$ be an object in $\cE$ and consider an admissible sequence in $\cE$
        \[ 0 \to y \tow{u} z \tow{v} x \to 0 \]
        such that $\Hom_{\uE}(a, v)$ is epi. We claim that $\Hom_\cE(a, v) \colon \Hom_\cE(a, z) \to \Hom_\cE(a, x)$ is epi.

        Let $f \in \Hom_\cE(a, x)$. Since $\Hom_{\uE}(a, v)$ is an epimorphism, there is a $\underline{g} \in \Hom_{\uE}(a, z)$ such that $\underline{v} \cdot \underline{g} = \underline{f}$. Equivalently, $\alpha = v \cdot g - f \colon a \to x$ factors through a projective $p \in \proj \cE$. Note that $f = v \cdot g - \alpha$. We will show that $\alpha$ factors through $v$. 

        Indeed, consider the following diagram
        \[ \begin{tikzcd}
            z \arrow[r, "v", two heads] & x & \arrow[l, "\alpha", swap] \arrow[dl, "\lambda"] a \\
            {} & p \arrow[u, "q"] \arrow[ul, "\gamma", dashed] & {}
        \end{tikzcd}\]
        where the arrow $\gamma$ exists because $v$ is an admissible epi and $p$ is projective. Here, the inner triangles commute by construction. Hence, 
        \[ (v \cdot \gamma) \cdot \lambda = q \cdot \lambda = \alpha \] 
        Thus, we may rewrite $f = v\cdot g - v\cdot \gamma \cdot \lambda = \Hom_\cE(a, v)(g - \gamma \cdot \lambda)$ proving that $\Hom_\cE(a, v)$ is epi. 
    \end{proof}
    
\printbibliography

\end{document}